\documentclass[11pt]{article}
\usepackage{mathrsfs}
\usepackage{amsmath}
\usepackage{color}
\usepackage{bbm}
\usepackage{amsmath} \allowdisplaybreaks
\usepackage{amsmath,amsthm,amssymb,amscd}
\usepackage{latexsym}
\usepackage{hyperref}
\usepackage[numbers,sort&compress]{natbib}
\usepackage{hypernat}
\newtheorem{theorem}{Theorem}[section]

\newtheorem{lemma}[theorem]{Lemma}
\newtheorem{proposition}[theorem]{Proposition}
\newtheorem{definition}[theorem]{Definition}
\newtheorem{remark}[theorem]{Remark}

\theoremstyle{definition} \theoremstyle{remark}
\numberwithin{equation}{section}

\begin{document}

\title{{\bf Sharper $L^1$-convergence rates of weak entropy
solutions to damped compressible Euler equations}\thanks{The work was supported partly by the NSF of China (12171094, 11831011), and the Shanghai Key
Laboratory for Contemporary Applied Mathematics (08DZ2271900).}}
\author{Jun-Ren Luo$^{a}$,  Ti-Jun Xiao$^{b}$ \thanks{Corresponding author. \ E-mail: \ tjxiao@fudan.edu.cn}
\\{\small $^a$  College of Science, University of Shanghai for Science and Technology, Shanghai 200093, China}\\
{\small $^b$  Shanghai Key
Laboratory for Contemporary Applied Mathematics}\\ {\small School of Mathematical Sciences, Fudan University, Shanghai 200433, China}}

\date{}
\maketitle

\begin{minipage}[2cm]{14cm}

\noindent {\bf Abstract:} We consider the asymptotic behavior of compressible isentropic flow when the initial mass is finite, which is modeled by the compressible Euler equation with frictional damping. It is shown in \cite{HUA} (resp.\cite{GEN}) that any $L^{\infty}$ weak entropy solution of damped compressible Euler equation converges to the Barenblatt solution with finite mass in $L^1$ norm, with convergence rates depending on the adiabatic gas exponent $\gamma$ in the case of $1<\gamma<3$ (resp.$\gamma\ge2$). Whether or not these convergence rates can be improved remains an interesting and challenging open question. In this paper, we obtain a better $L^1$ convergence rate than that in \cite{GEN},
for any $\gamma\ge2$, through a new perspective on the relationship between the density function and the Barenblatt solution of the porous medium equation.
Furthermore, making intensive analysis of some relevant convex functions, we are able to obtain the same form of $L^1$ convergence rate for $1<\gamma<9/7$, which is  better than that in \cite{HUA} as well.

\vspace{0.4cm}


\end{minipage}

\section{Introduction}

 In this paper, we consider the compressible Euler equations with frictional damping
\begin{eqnarray}\label{1.1}
\left\{\begin{array}{ll}
\rho_t+(\rho u)_x=0, \\
(\rho u)_t+\left(\rho u^2+p(\rho)\right)_x=-\alpha\rho u,
\end{array} \right.
\end{eqnarray}
with finite initial mass
\begin{equation}\label{r0u0}
(\rho, u)(x, 0)=\left(\rho_0, u_0\right)(x), \quad \rho_0(x) \geq 0, \quad \int^{+\infty}_{-\infty} \rho_0(x) d x=M>0,
\end{equation}
where $\rho, u$, and $p$ denote the density, velocity, and pressure, respectively. We assume the flow is a polytropic perfect gas, then $p(\rho)=\kappa \rho^\gamma$ with $\kappa=\frac{(\gamma-1)^2}{4\gamma}$, and $\gamma>1$ is the adiabatic gas exponent. The positive constant $\alpha$ models frictional force induced by the medium, and we assume $\kappa=\alpha$; the identity of the two constants will simplify the form of the entropy functions we use below (cf. \cite{HUA,LIO}).

The inertial terms in the momentum equation decay to zero faster than other terms, which is caused by the dissipative nature of frictional force. That is, the pressure gradient force is balanced by the frictional force, which was described as Darcy's law; see Hsiao and Liu \cite{HSI}. Hence, the system \eqref{1.1} is equivalent to the following decoupled system, if taking time asymptotically,
\begin{eqnarray}\label{1.3}
\left\{\begin{array}{ll}
\bar{\rho}_t=\left(\bar{\rho}^\gamma\right)_{x x}, \\
\bar{m}=-\left(\bar{\rho}^\gamma\right)_x,
\end{array} \right.
\end{eqnarray}
which was first justified by Hsiao and Liu in \cite{HSI, HSI2}. The first equation is the famous porous medium equation (PME), which is a typical degenerate parabolic equation. The second equation is the Darcy's law, which is expected to formulate the momentum.  Regarding the convergence of the solution of \eqref{1.1} toward the Barenblatt solution $\bar{\rho}(x,t)$ (a fundamental solution of PME), it is particularly relevant to consider the finite mass case (see the inspiring paper of Liu \cite{LIU}).

Mathematical study of system \eqref{1.1} originated from the pioneering work of Nishida \cite{NIS} in the 1970s, and since then, the system has been
extensively studied in the literature. Among the numerous works, we mention mainly two cases. One case is away from vacuum (cf. \cite{MAR, SMO}), in which the system \eqref{1.1} was transferred to the damped $p$-system by changing to the Lagrangian coordinates. The global existence of weak solutions in $L^p$ was established in \cite{DIN, DIP, LIO1}, by the method of compensated compactness, and in \cite{DAF1, DAF2, LUS} the global BV solutions were proved. For small smooth or piecewise smooth solutions away from vacuum based on the energy estimates, we refer the reader to \cite{LUO0, NIS1, ZHA}, and references therein.
The other case is the occurrence of a vacuum in the solution, which becomes quiet different and difficult to handle, mainly due to the interaction of nonlinear convection, lower order dissipation of damping and the resonance due to the vacuum. In spite of these difficulties, the global weak solution with vacuum was constructed in $L^{\infty}$ space by using the method of compensated compactness; we refer to Ding et al.\cite{DIN} for $1<\gamma\le\frac{5}{3}$, and Huang and pan \cite{HUA1} for $1<\gamma<3$. Using the vanishing viscosity method, Zhu in \cite{ZHU0} studied the asymptotic behavior of weak entropy solution. In addition, under the assumption of physical vacuum free boundary condition, Luo and Zeng \cite{LUO} considered the existence and pointwise convergence rates of classical solutions. Later, Zeng \cite{ZEN}  considered the global existence of smooth solutions and the convergence to Barenblatt's self-similar solutions for spherically symmetric motions in three dimensions. The relationship between the Barenblatt solutions and the nonlinear fractional heat equation was researched in \cite{VAZ}.

In \cite{HUA0}, the authors introduced the entropy dissipation method, which is an effective approach in studying the large time asymptotic behavior for $L^{\infty}$ weak entropy solutions for hyperbolic conservation laws with dissipation. It is proved in \cite{HUA} that any $L^{\infty}$ weak entropy solution of problem \eqref{1.1}- \eqref{r0u0} converges to the Barenblatt solution of equation \eqref{1.3} with the same mass in $L^1$ as
$$
\|(\rho-\bar{\rho})(\cdot, t)\|_{L^1} \leq C(1+t)^{-\frac{1}{4(\gamma+1)}+\varepsilon},
$$
for any $\varepsilon>0$ with $1<\gamma<3$. The author used a comprehensive entropy analysis, capturing the dissipative character of the problem. Later, the convergence rates were improved to
$$
\|(\rho-\bar{\rho})(\cdot, t)\|_{L^1} \leq C(1+t)^{-\frac{1}{(\gamma+1)^2}+\varepsilon},
$$
for $\gamma\ge2$ in \cite{GEN}, where the authors took the decay property of $\|(\rho-\bar{\rho})(\cdot, t)\|_{L^1}$ into account and introduced an iterative method.

Moreover, we refer the reader to \cite{CHE1, CUI1, CUI2, ZHU}, and references therein, regarding the convergence rates to the Barenblatt solutions for the compressible Euler equations with time-dependent damping, that is, the right-hand side of $\eqref{1.1}_2$ becomes $-\frac{\alpha}{(1+t)^{\lambda}}\rho u$. Pan \cite{PAN1, PAN2} proved that $\lambda= 1, \alpha = 2$ was the critical threshold between the global existence and non-existence of $C^1$ solutions in one dimension. Later, the critical case $\lambda = 1, \alpha > 2$ with one dimension was studied in \cite{GEN1}; see \cite{HOU, HOU1} for the multi-dimensional problem. In fact, it is natural and worthy to study system behaviors when the damping term is dependent on time. For hyperbolic equations with time-dependent damping, we refer to \cite{GHI, LUO1, LUO2,WIR1,WIR2} and references therein.

On the other hand, the original $L^{1}$ and $L^{\gamma}$ convergence rate estimates due to \cite{HUA0,HUA}, for $L^{\infty}$ weak entropy solutions to the damped compressible Euler equations \eqref{1.1}, have remained unimproved. One wonders whether these rates can be actually improved or not (cf.\cite[page 687]{HUA0}). In fact, it appears rather tricky to make some improvements. In the present paper, we establish better $L^{1}$ and $L^{\gamma}$ convergence rate estimates than those in \cite{HUA} for any $\gamma\ge2$, and furthermore obtain better $L^{1}$ and $L^{\gamma+1}$ estimates than those in \cite{HUA0} for $1<\gamma<\frac{9}{7}$. In both cases, the $L^{1}$ estimates we obtain have the same form. See Theorem \ref{thm1} and Remark \ref{rm}.

The outline of this paper is the following. In section 2, we give a quick review of some information on Barenblatt's solution, and present our main result: Theorem \ref{thm1}.  Then, in section 3, we state or derive some important lemmas to help prove the main result. Section 4 shows some important properties of the entropy-flux pair. In section 5, we devote ourselves to proving Theorem \ref{thm1}. Some technical results for the case of $1<\gamma<\frac{9}{7}$ used in section 5 are stated/proved in Appendix.

\section{Preliminaries and main result}
Throughout this paper, $||\cdot||_p$ stands for $L^{p}(R)$-norm ($1\le p\le+\infty$). For simplicity of notation, in particular, we use $\| \cdot\|$ instead of $\| \cdot\|_2$

For a binary function $F(\rho, m)$, set $v=(\rho, m)^t, \bar{v}=(\bar{\rho}, \bar{m})^t$, where $(\cdot, \cdot)^t$ means the transpose of the vector $(\cdot, \cdot)$. Then, $F(\rho, m)=F(v^{t})$ and we write, for convenience,
\begin{align*}
F_*&=F(v)-F(\bar{v})-\nabla F(\bar{v})\left(v-\bar{v}\right)\\
&=F(v)-F(\bar{v})-F_{\bar{\rho}}\left(\rho-\bar{\rho}\right)-F_{\bar{m}}\left(m-\bar{m}\right),
\end{align*}
where $\nabla F(v):=\left(F_\rho, F_m\right)$. Specially, for a unary function $F(\rho)$, one has
$$
\left(F(\rho)\right)_*=F(\rho)-F(\bar{\rho})-F_{\bar{\rho}}\left(\rho-\bar{\rho}\right).
$$

Next, we state the definition of $L^{\infty}$ weak entropy solution (cf. \cite{GEN, HUA}).
\begin{definition}
The pair $(\rho, m)(x, t) \in L^{\infty}$ is called an entropy solution of system \eqref{1.1}-\eqref{r0u0}, if for any nonnegative test function $\phi \in \mathcal{D}\left(\mathbf{R} \times \mathbf{R}_{+}\right)$, it holds that
$$
\left\{\begin{array}{l}
\iint_{t>0}\left(\rho \phi_t+m \phi_x\right) d x d t+\int_{\mathbf{R}} \rho_0(x) \phi(x, 0) d x=0, \\
\iint_{t>0}\left[m \phi_t+\left(\frac{m^2}{\rho}+p(\rho)\right) \phi_x-m \phi\right] d x d t+\int_{\mathbf{R}} m_0(x) \phi(x, 0) d x=0,
\end{array}\right.
$$
and
\begin{align*}
\eta_t+q_x+\alpha\eta_m m \leq 0,
\end{align*}
in the sense of distributions, where $(\eta, q)$ is any weak convex entropy-flux pair.
\end{definition}

According to \cite{LIO}, all weak entropies of \eqref{1.1} are given by the following formula
\begin{equation}\label{etaq}
\begin{aligned}
\eta(\rho, m) & =\int g(\xi) \chi(\xi ; \rho, u) \mathrm{d} \xi=\rho \int_{-1}^1 g\left(u+z \rho^\theta\right)\left(1-z^2\right)^\lambda \mathrm{d} z, \\
q(\rho, m) & =\int g(\xi)(\theta \xi+(1-\theta) u) \chi(\xi ; \rho, u) \mathrm{d} \xi \\
& =\rho \int_{-1}^1 g\left(u+z \rho^\theta\right)\left(u+\theta z \rho^\theta\right)\left(1-z^2\right)^\lambda \mathrm{d} z,
\end{aligned}
\end{equation}
where $m=\rho u, \theta=\frac{\gamma-1}{2}, \lambda=\frac{3-\gamma}{2(\gamma-1)}, g(\xi)$ is any smooth function of $\xi$, and
$$
\chi(\xi ; \rho, u)=\left(\rho^{\gamma-1}-(\xi-u)^2\right)_{+}^\lambda .
$$
Due to the extensive use of $\eta(\rho, m)$ in the proof process, we use
$$
g=g(u+z\rho^{\theta}),\quad \bar{g}=\bar{g}(\bar{u}+z\bar{\rho}^{\theta})
$$
for convenience without confusions.

Next, we recall some basics about Barenblatt's solution (cf. \cite{ARO,BAR}).
Consider
\begin{eqnarray}\label{3.1}
\left\{\begin{array}{ll}
\bar{\rho}_t=\left(\bar{\rho}^\gamma\right)_{x x}, \\
\bar{\rho}(x,-1)=M \delta(x), \quad M>0,
\end{array} \right.
\end{eqnarray}
which admits a unique solution given by
\begin{equation}\label{barrho}
\bar{\rho}(x, t)=(1+t)^{-\frac{1}{\gamma+1}}\left\{\left(A_0-B_0 (1+t)^{-\frac{2}{\gamma+1}}x^2\right)_{+}\right\}^{\frac{1}{\gamma-1}},
\end{equation}
where $(f)_{+}=\max \{0, f\}$,
$$
B_0=\frac{\gamma-1}{2(\gamma+1)}, \quad \text { and } \quad A_{0}^{\frac{\gamma+1}{2(\gamma-1)}}=M \sqrt{B_0}\left(\int_{-1}^1\left(1-y^2\right)^{1 /(\gamma-1)} d y\right)^{-1} \text {. }
$$
At the same time, the velocity is given by
\begin{equation}\label{ux}
\bar{u}(x, t)=\frac{x}{(\gamma+1)(1+t)},\,|x| < \sqrt{\frac{A_0}{B_0}}(1+t)^{\frac{1}{\gamma+1}}.
\end{equation}
Consequently,
$$
\left(\frac{\bar{m}}{\bar{\rho}}\right)_x=\bar{u}_x=\frac{1}{(\gamma+1)(1+t)},\,|x| < \sqrt{\frac{A_0}{B_0}}(1+t)^{\frac{1}{\gamma+1}}.
$$

It is easy to see that the derivative of $\bar{\rho}$ is not continuous across the interface between the gas and vacuum, which is caused by the degeneracy at vacuum. Indeed, $\bar{\rho}$ is a weak solution to \eqref{3.1} satisfying
$
\int_{-\infty}^{+\infty} \bar{\rho} \mathrm{d} x=M,
$
and
$$
\bar{\rho}=0, \quad \text { if } \quad|x| \ge \sqrt{\frac{A_0}{B_0}}(1+t)^{\frac{1}{\gamma+1}}.
$$
Therefore, for a limited period of time $\bar{\rho}$ has compact support, which reflects the phenomenon of finite speed of propagation of porous medium equation.

Now, we present our main result.

\begin{theorem}\label{thm1}
Suppose that $\rho_0(x) \in L^1(\mathbf{R}) \cap L^{\infty}(\mathbf{R}), u_0(x) \in L^{\infty}(\mathbf{R})$ and
$$
M=\int_{-\infty}^{+\infty} \rho_0(x) d x>0, \quad \rho_0(x) \geq 0 .
$$
Let $1<\gamma<+\infty$, and $(\rho, m)$ be an $L^{\infty}$ entropy solution of the Cauchy problem \eqref{1.1}-\eqref{r0u0}. Let $\bar{\rho}$ be the Barenblatt solution of porous medium equation \eqref{1.3} with mass $M$ and $\bar{m}=-\left(\bar{\rho}^\gamma\right)_x$. Define
$$
y(x, t)=-\int_{-\infty}^x(\rho-\bar{\rho})(r, t) d r .
$$
If $y(x, 0) \in L^2(\mathbf{R})$, then for any $\varepsilon>0$ and $t>0$, it holds that\\
\noindent{(i)} If $2\le\gamma<+\infty$, then

\begin{equation}\label{1rr+1}
\begin{aligned}
& \|(\rho-\bar{\rho})(\cdot, t)\|_{L^\gamma}^\gamma \leq C_{\varepsilon}(1+t)^{-\frac{\gamma^2+\gamma-1}{(\gamma+1)^2}+\varepsilon}, \\
& \|(\rho-\bar{\rho})(\cdot, t)\|_{L^1} \leq C_{\varepsilon}(1+t)^{-\frac{\gamma}{2(\gamma+1)^2}+\varepsilon};
\end{aligned}
\end{equation}

\noindent{(ii)} If $1<\gamma< \frac{9}{7}$, then
\begin{equation}\label{1rr+1'}
\begin{aligned}
& \|(\rho-\bar{\rho})(\cdot, t)\|_{L^\gamma+1}^{\gamma+1} \leq C_{\varepsilon}(1+t)^{-\frac{\gamma^2+2\gamma}{(\gamma+1)^2}+\varepsilon}, \\
& \|(\rho-\bar{\rho})(\cdot, t)\|_{L^1} \leq C_{\varepsilon}(1+t)^{-\frac{\gamma}{2(\gamma+1)^2}+\varepsilon} .
\end{aligned}
\end{equation}

\end{theorem}

\begin{remark}\label{rm}
It is shown in \cite{GEN} for $\gamma\ge2$ and \cite{HUA} for $1<\gamma<3$ that $\|(\rho-\bar{\rho})(\cdot, t)\|_{L^1}\leq C(1+t)^{-\frac{1}{(\gamma+1)^2}+\varepsilon}$ and $\|(\rho-\bar{\rho})(\cdot, t)\|_{L^1}\leq C(1+t)^{-\frac{1}{4(\gamma+1)}+\varepsilon}$, respectively. It is easy to check that $\frac{\gamma}{2(\gamma+1)^2}>\frac{1}{4(\gamma+1)}$ when $\gamma>1$, and $\frac{\gamma}{2(\gamma+1)^2}>\frac{1}{(\gamma+1)^2}$ when $\gamma\ge2$. This means that the  $L^1$ convergence rates in Theorem \ref{thm1} are better. In addition, we note that the $L^1$ convergence rate estimates in \eqref{1rr+1} and \eqref{1rr+1'} are of the same form.

To prove \ref{thm1}, we will follow the strategy in \cite{HUA,GEN}, employing the entropy inequality (see Lemma \ref{lem3} below) associated with an appropriate entropy $\eta$. For $\gamma\ge2$, we choose (as in \cite{GEN})
\begin{equation}\label{etae}
\eta=\frac{m^2}{2 \rho}+\frac{\kappa}{\gamma-1} \rho^\gamma,
\end{equation}
which takes the form of mechanical energy. This $\eta$ measures the $L^{\gamma}$ norm while the entropy inequality measures the $L^{\gamma+1}$ norm in density, and the mismatch between the exponents poses an essential difficulty in obtaining sharp decay rates. In order to overcome this difficulty, our idea is to find the relationship between
$$
\rho^{\gamma+1}-\bar{\rho}^{\gamma+1}-(\gamma+1)\bar{\rho}^{\gamma}(\rho-\bar{\rho})
~
\mbox{and}
~
\rho^{\gamma}-\bar{\rho}^{\gamma}-\gamma\bar{\rho}^{\gamma-1}(\rho-\bar{\rho});
$$
see Lemma \ref{lem1}.

For $1<\gamma<\frac{9}{7}$, $\eta$ in \eqref{etae} is no longer suitable, since
$$
C|\rho-\bar{\rho}|^2 \leq \rho^\gamma-\bar{\rho}^\gamma-\gamma \bar{\rho}^{\gamma-1}(\rho-\bar{\rho}) \leq C'|\rho-\bar{\rho}|^\gamma,
$$
for $1<\gamma\le2$. That is, the index 2 of $|\rho-\bar{\rho}|^2$ is not precise enough. Hence, we should make use of Lemma \ref{lem2} and estimate
$
\rho^{\gamma+1}-\bar{\rho}^{\gamma+1}-(\gamma+1) \bar{\rho}^\gamma(\rho-\bar{\rho}).
$
Consequently, the index of $\rho$ in $\eta$ should be raised to $\gamma+1$ compared to $\gamma$ in \eqref{etae}(cf.\cite{HUA}). However, some new difficulties arise besides the lack of regularity and the resonance near vacuum. The main difficulty is to show that $A_*,D_*$ and $(\eta-C_0\rho^{\gamma+1})_*$ (defined respectively in \eqref{A*B*}, Lemma \ref{lem69} and Lemma \ref{lem610}) are nonnegative, and therefore justify the convexity of $A, D$ and $\eta-C_0\rho^{\gamma+1}$. We will employ Taylor's theorem and make in-depth analysis of the associated Hessian matrix to deal with the issue. This turns out to be a particularly challenging task. The restriction of $\gamma<\frac{9}{7}$ in Theorem \ref{thm1}-(ii) stems from $k\ge4$ in Lemma \ref{lem65}.
\end{remark}

\section{Some auxiliary results}

In this section, we state/prove some auxiliary results, which are important for proving our main result.

\begin{lemma}\label{lem1}
Let $1<\gamma<+\infty$, and $0\leq \rho,\bar{\rho}\leq C$. Then, there exists a constant $c>0$, such that
\begin{equation}\label{gam1}
\rho^{\gamma+1}-\bar{\rho}^{\gamma+1}-(\gamma+1)\bar{\rho}^{\gamma}(\rho-\bar{\rho})\ge c(\rho^{\gamma}-\bar{\rho}^{\gamma}-\gamma\bar{\rho}^{\gamma-1}(\rho-\bar{\rho}))^{\frac{\gamma+1}{\gamma}}.
\end{equation}
\end{lemma}

\begin{proof}
When $\bar{\rho}=0,$ \eqref{gam1} is true for any $0<c\le 1$. For $\rho=0$,
\eqref{gam1} becomes
$
\gamma\bar{\rho}^{\gamma+1}\ge c(\gamma-1)^{\frac{\gamma+1}{\gamma}}\bar{\rho}^{\gamma+1},
$
and it is equivalent to proving that
$
\displaystyle\frac{\gamma}{(\gamma-1)^{\frac{\gamma+1}{\gamma}}},
$
has a positive lower bound, which is guaranteed by $\gamma<+\infty$.

Now, assume $\bar{\rho}\neq0, \rho\neq0.$ Then $\bar{\rho}>0$.
Notice that \eqref{gam1} is equivalent to
$$
\left(\frac{\rho}{\bar{\rho}}\right)^{\gamma+1}-1-(\gamma+1)\left(\frac{\rho}{\bar{\rho}}-1\right)\ge c\left[\left(\frac{\rho}{\bar{\rho}}\right)^{\gamma}-1-\gamma\left(\frac{\rho}{\bar{\rho}}-1\right)\right]^{\frac{\gamma+1}{\gamma}}.
$$
Set $\frac{\rho}{\bar{\rho}}=x$ and
$$
f(x)=x^{\gamma+1}-1-(\gamma+1)(x-1)-c\left[x^{\gamma}-1-\gamma(x-1)\right]^{\frac{\gamma+1}{\gamma}};
$$
it is obvious that $f(1)=0$ and
\begin{align*}
\frac{f'(x)}{\gamma+1}&=x^{\gamma}-1-\frac{c}{\gamma}\left[x^{\gamma}-1-\gamma(x-1)\right]^{\frac{1}{\gamma}}
\left(\gamma x^{\gamma-1}-\gamma\right)\\
&=x^{\gamma}-1-c\left[x^{\gamma}-1-\gamma(x-1)\right]^{\frac{1}{\gamma}}\left( x^{\gamma-1}-1\right).
\end{align*}

Next, we claim that there exists $0<c\le 1$, such that $f'(x)\le0, \forall x\in (0,1)$, and $f'(x)\ge0, \forall x\in (1,+\infty)$; that is,
\begin{align*}
f(x)\ge f(1)=0,\quad \forall x\in[0,+\infty).
\end{align*}
 To this end, we set
$$
h(x)=\frac{x^{\gamma}-1}{\left[x^{\gamma}-1-\gamma(x-1)\right]^{\frac{1}{\gamma}}\left( x^{\gamma-1}-1\right)},
$$
and distinguish two cases.

\vspace{4pt}
\noindent\textbf{$Case~ I: ~x\in (1,+\infty).$}

Observing
$
\left[(1+t)^{k}-1\right]\sim k t~(t\rightarrow 0),
$
where $k>0$ is a constant, we find
\begin{align*}
\lim_{x\rightarrow 1^{+}}\frac{x^{\gamma}-1}{\left[x^{\gamma}-1-\gamma(x-1)\right]^{\frac{1}{\gamma}}\left( x^{\gamma-1}-1\right)}=\lim_{x\rightarrow 1^{+}}\frac{\gamma x}{\left[x^{\gamma}-1-\gamma(x-1)\right]^{\frac{1}{\gamma}}(\gamma-1) x}=+\infty,
\end{align*}
due to
$$
\left[x^{\gamma}-1-\gamma(x-1)\right]^{\frac{1}{\gamma}}\rightarrow 0, \quad \mbox{as}~\,x\rightarrow 1^{+}.
$$
In addition,
\begin{align*}
\lim_{x\rightarrow +\infty}h(x)=\lim_{x\rightarrow +\infty}\frac{1-\frac{1}{x^{\gamma}}}{\left[1-\frac{1}{x^{\gamma}}-\gamma\left (\frac{1}{x^{\gamma-1}}-\frac{1}{x^{\gamma}}\right)\right]^{\frac{1}
{\gamma}}{\left( 1-\frac{1}{x^{\gamma-1}}\right)}}=1.
\end{align*}
Hence, there exist $0<\delta<1$ and $X>2$, such that when $x\in(1,1+\delta)$ or $x>X$,
$
h(x)\ge\frac{1}{2}.
$
When $x\in[1+\delta, X]$,
$
\min_{x\in[1+\delta, X]}h(x)=h(x_0)>0,
$
since $h(x)$ is a continuous function.

In conclusion, for any $x\in(1,+\infty)$,
$
h(x)\ge\min\left\{h(x_0),\frac{1}{2}\right\}.
$

\vspace{4pt}
\noindent\textbf{$Case ~II:~ x\in (0,1).$}

In this case, $x^{\gamma-1}-1<0$,
\begin{align*}
\lim_{x\rightarrow 1^{-}}\frac{x^{\gamma}-1}{[x^{\gamma}-1-\gamma(x-1)]^{\frac{1}{\gamma}}( x^{\gamma-1}-1)}=\lim_{x\rightarrow 1^{-}}\frac{\gamma x}{[x^{\gamma}-1-\gamma(x-1)]^{\frac{1}{\gamma}}(\gamma-1) x}=+\infty,
\end{align*}
and
\begin{align*}
\lim_{x\rightarrow 0^+}\frac{x^{\gamma}-1}{[x^{\gamma}-1-\gamma(x-1)]^{\frac{1}{\gamma}}( x^{\gamma-1}-1)}=(\gamma-1)^{-\frac{1}{\gamma}}.
\end{align*}
Hence, there exists $0<\delta_1<\frac{1}{4}$, such that when $x\in(1-\delta_1, 1)$ or $x\in(0, \delta_1)$,
$$
h(x)\ge\frac{1}{2}(\gamma-1)^{-\frac{1}{\gamma}}.
$$
If $x\in[\delta_1, 1-\delta_1]$, we have, similarly as in $Case~ I$,
$$
\min_{x\in[\delta_1, 1-\delta_1]}h(x)=h(x_1)>0,
$$
and
$$
h(x)\ge\min\left\{h(x_1),\frac{1}{2}(\gamma-1)^{-\frac{1}{\gamma}}\right\}.
$$

Consequently, we deduce
$$
h(x)\ge\min\left\{h(x_0),h(x_1), \frac{1}{2},\frac{1}{2}(\gamma-1)^{-\frac{1}{\gamma}}\right\}>0.
$$
Hence, noting $
x^{\gamma}-1-\gamma(x-1)>0, \, \forall x\neq1,
$ and letting
$$
c=\frac{1}{2}\min\left\{\frac{\gamma}{(\gamma-1)^{\frac{\gamma+1}{\gamma}}},h(x_0),h(x_1), \frac{1}{2},\frac{1}{2}(\gamma-1)^{-\frac{1}{\gamma}}\right\}
$$
justifies the claim
and so completes the proof.
\end{proof}

Next, we recall some important inequalities giving sharp information on the pressure near vacuum, and an invariant region result for $L^{\infty}$ weak entropy solution to system \eqref{1.1}-\eqref{r0u0}.

 \begin{lemma}[\cite{HUA0, HUA}]\label{lem2}
Assume $0 \leq \rho, \bar{\rho} \leq C$ and $\gamma>1$, then there are two constants $d_1>0$ and $d_2>0$ such that
$$
\left\{\begin{aligned}
\left(\rho^\gamma-\bar{\rho}^\gamma\right)(\rho-\bar{\rho})&\ge |\rho-\bar{\rho}|^{\gamma+1},\\
d_1\left(\rho^{\gamma-1}+\bar{\rho}^{\gamma-1}\right)(\rho-\bar{\rho})^2 & \leqq \rho^{\gamma+1}-\bar{\rho}^{\gamma+1}-(\gamma+1) \bar{\rho}^\gamma(\rho-\bar{\rho}) \\
& \leqq d_2\left(\rho^{\gamma-1}+\bar{\rho}^{\gamma-1}\right)(\rho-\bar{\rho})^2, \\
d_1\left(\rho^{\gamma-1}+\bar{\rho}^{\gamma-1}\right)(\rho-\bar{\rho})^2 & \leqq\left(\rho^\gamma-\bar{\rho}^\gamma\right)(\rho-\bar{\rho}) \\
& \leqq d_2\left(\rho^{\gamma-1}+\bar{\rho}^{\gamma-1}\right)(\rho-\bar{\rho})^2 .
\end{aligned}\right.
$$
\end{lemma}

\begin{lemma}[\cite{HOU1}]\label{thm2}
Suppose that $\left(\rho_0, u_0\right)(x) \in L^{\infty}(\mathbf{R})$ satisfies
$$
0 \leq \rho_0(x) \leq C, \quad\left|m_0(x)\right| \leq C \rho_0(x) .
$$
Let $(\rho, u) \in L^{\infty}(\mathbf{R} \times[0, T])$ be an $L^{\infty}$ weak entropy solution of the system \eqref{1.1}-\eqref{r0u0} with $\gamma>1$. Then $(\rho, m)$ satisfies
$$
0 \leq \rho(x, t) \leq C, \quad|m(x, t)| \leq C \rho(x, t),
$$
where the constant $C$ depends solely on the initial data.
\end{lemma}

Suppose that $(\rho ,m)$ (with $m=\rho u$) be a weak entropy solution of system \eqref{1.1}-\eqref{r0u0} satisfying the conditions in Theorem 1.1, and that $\bar{\rho}$ is the Barenblatt solution of the porous medium equation (with the same total mass $M$ as $\rho$), $\bar{m}=-\left(\bar{\rho}^\gamma\right)_x$, and $\bar{R}=\bar{m}_t+\left(\frac{\bar{m}^2}{\bar{\rho}}\right)_x$. Then, it follows that
\begin{eqnarray}\label{4.2}
\left\{\begin{array}{l}
\rho_t+m_x=0, \\
m_t+\left(\frac{m^2}{\rho}+p(\rho)\right)_x=-\alpha m,
\end{array}\right.
\end{eqnarray}
and
\begin{eqnarray}\label{4.3}
\left\{\begin{array}{l}
\bar{\rho}_t+\bar{m}_x=0, \\
\bar{m}_t+\left(\frac{\bar{m}^2}{\bar{\rho}}+p(\bar{\rho})\right)_x=-\alpha \bar{m}+\bar{R}.
\end{array}\right.
\end{eqnarray}
Writing $w=\rho-\bar{\rho},~z=m-\bar{m}$, we have
\begin{eqnarray}\label{4.4}
\left\{\begin{array}{l}
w_t+z_x=0, \\
z_t+\left(\frac{m^2}{\rho}\right)_x+\left(p(\rho)-p(\bar{\rho})\right)_x+\alpha z=-\bar{m}_t .
\end{array}\right.
\end{eqnarray}
Set $y=-\int_{-\infty}^x w(r, t) \mathrm{d} r$; one has $y_x=-w,~ z=y_t.$
Hence, the equation \eqref{4.4} turns into a nonlinear wave equation with source terms, degenerate at vacuum:
$$
y_{tt}+\left(\frac{m^2}{\rho}\right)_x+\left(p(\rho)-p(\bar{\rho})\right)_x+\alpha y_t=-\bar{m}_t .
$$
The following is
an a priori estimate for the above equation.

\begin{lemma}[\cite{HUA}]\label{lem3}
Under the conditions of Theorem 1.1, for any $t>0$, it holds that
$$
\begin{gathered}
\int_{-\infty}^{+\infty}\left(y^2+y_t^2+\frac{m^2}{\rho}+p(\rho)\right) d x+\int_0^t \int_{-\infty}^{+\infty}\left(y_t^2+\frac{m^2}{\rho}\right) d x d \tau \\
+\int_0^t \int_{-\infty}^{+\infty}(\rho^\gamma-\bar{\rho}^{\gamma})(\rho-\bar{\rho}) d x d \tau \leq C .
\end{gathered}
$$
\end{lemma}

We now make estimates of $\bar{\rho}^{\beta_1}\bar{u}^{\beta_2}$ and $\bar{\rho}^\delta\left(\frac{\bar{R}}{\bar{\rho}}\right)$.
\begin{lemma}\label{lem4}
Let $1\leq p<+\infty$,
\begin{equation}\label{beta1}
\beta_1\ge-\frac{\gamma-1}{p},
\end{equation}
$\beta_2\ge0$, and $\delta>0$. Then, it holds that
\begin{equation}\label{rhobeta}
\left\|\bar{\rho}^{\beta_1}\bar{u}^{\beta_2}\right\|_{L^p}\leq C(1+t)^{-\frac{p(\beta_2\gamma+\beta_1)-1}{p(r+1)}},
\end{equation}
\begin{equation}\label{rhobeta'}
\left\|\bar{\rho}^{\beta_1}\bar{u}^{\beta_2}\right\|_{L^{\infty}}\leq C(1+t)^{-\frac{\beta_2\gamma+\beta_1}{r+1}},
\end{equation}
and
\begin{equation}\label{rhodelta}
\left\|\bar{\rho}^\delta\left(\frac{\bar{R}}{\bar{\rho}}\right)\right\|_{L^p} \leq C(1+t)^{-\frac{p(\delta+2 \gamma+1)-1}{p(\gamma+1)}} .
\end{equation}
\end{lemma}

\begin{proof}
On one hand, using \eqref{barrho}, \eqref{ux} and $\xi=x(t+1)^{-\frac{1}{\gamma+1}}$, we have
\begin{align*}
&\left\|\bar{\rho}^{\beta_1}\bar{u}^{\beta_2}\right\|^p_{L^p}\\
=&\int^{+\infty}_{-\infty}\left\{(1+t)^{-\frac{1}{\gamma+1}}\left[\left(A_0-B_0 (1+t)^{-\frac{2}{\gamma+1}}x^2\right)_{+}\right]^{\frac{1}{\gamma-1}}\right\}^{\beta_1 p}\left[\frac{ x}{(1+t)(\gamma+1)}\right]^{p\beta_2}dx\\
\leq& C\int^{\sqrt{\frac{A_0}{B_0}}}_{-\sqrt{\frac{A_0}{B_0}}}(1+t)^{-\frac{\beta_1 p}{\gamma+1}}(1+t)^{-\frac{p\beta_2\gamma}{\gamma+1}}
\left[\left(A_0-B_0\xi^2\right)_{+}\right]^{\frac{\beta_1 p}{\gamma-1}}\xi^{p\beta_2}(1+t)^{\frac{1}{\gamma+1}}d\xi\\
\leq& C(1+t)^{-\frac{p(\beta_2\gamma+\beta_1)-1}{r+1}},
\end{align*}
where the last inequality is from the estimation
\begin{align*}
\int^{\sqrt{\frac{A_0}{B_0}}}_{-\sqrt{\frac{A_0}{B_0}}}\left[\left(A_0-B_0\xi^2\right)_{+}\right]^{\frac{\beta_1 p}{\gamma-1}}\xi^{p\beta_2}d
\xi&\leq  C\int^{1}_{-1}\left[\left(A_0-A_0\tau^2\right)_{+}\right]^{\frac{\beta_1 p}{\gamma-1}}\sqrt{\frac{A_0}{B_0}}d\tau\\
&=C\int^{1}_{0}(1-\tau)^{\frac{\beta_1 p}{\gamma-1}}(1+\tau)^{\frac{\beta_1 p}{\gamma-1}}d\tau\\
&\leq C\int^{1}_{0}(1-\tau)^{\frac{\beta_1 p}{\gamma-1}}d\tau\leq C,
\end{align*}
by $p\beta_2\ge0$, $\sqrt{\frac{A_0}{B_0}}\tau=\xi$, and \eqref{beta1}. Hence, we obtain \eqref{rhobeta}. Particularly,
\begin{align*}
\left\|\bar{\rho}^{\beta_1}\bar{u}^{\beta_2}\right\|_{L^{\infty}}=ess \sup_{x\in R}\left (\bar{\rho}^{\beta_1}\bar{u}^{\beta_2}\right)
\leq C(1+t)^{-\frac{\beta_2\gamma+\beta_1}{r+1}},
\end{align*}
by $x=C(1+t)^{\frac{1}{\gamma+1}}$, which gives \eqref{rhobeta'}.

On the other hand, since $\bar{R}=\bar{m}_t+\left(\frac{\bar{m}^2}{\bar{\rho}}\right)_x$, and $\bar{m}=\bar{\rho}\bar{u}$, we have
\begin{align*}
\bar{R}&=(\bar{\rho}\bar{u})_t+\left(\bar{\rho}\bar{u}^2\right)_x=\bar{\rho}_t\bar{u}+\bar{\rho}\bar{u}_t+\bar{\rho}_x\bar{u}^2+
2\bar{\rho}\bar{u}\bar{u}_x=\bar{\rho}\bar{u}_t+\bar{\rho}\bar{u}\bar{u}_x,
\end{align*}
by $\bar{\rho}_t+(\bar{\rho}\bar{u})_x=0$. Combining the above equality and \eqref{ux}, we derive
$$
\frac{\bar{R}}{\bar{\rho}}=\bar{u}_t+\bar{u}\bar{u}_x=-\frac{\gamma x}{(1+t)^{2}(\gamma+1)^2},
$$
which implies that with $\xi=x(t+1)^{-\frac{1}{\gamma+1}}$,
\begin{align*}
&\left\|\bar{\rho}^\delta\left(\frac{\bar{R}}{\bar{\rho}}\right)\right\|^{p}_{L^p} \\ =&\int^{+\infty}_{-\infty}\left\{(1+t)^{-\frac{1}{\gamma+1}}\left[\left(A_0-B_0 (1+t)^{-\frac{2}{\gamma+1}}x^2\right)_{+}\right]^{\frac{1}{\gamma-1}}\right\}^{\delta p}\left[\frac{\gamma x}{(1+t)^{2}(\gamma+1)^2}\right]^{p}dx\\
\leq& C\int^{\sqrt{\frac{A_0}{B_0}}}_{-\sqrt{\frac{A_0}{B_0}}}(1+t)^{-\frac{\delta p}{\gamma+1}}(1+t)^{-\frac{p(2\gamma+1)}{\gamma+1}}
\left[\left(A_0-B_0\xi\right)_{+}\right]^{\frac{\delta p}{\gamma-1}}\xi^p(1+t)^{\frac{1}{\gamma+1}}d\xi\\
\leq& C(1+t)^{-\frac{p(\delta+2 \gamma+1)-1}{\gamma+1}},
\end{align*}
from which we complete the proof of \eqref{rhodelta}.
\end{proof}

\section{Some important properties of entropy-flux pair}

First, since the convexity of the entropy function is crucial, we recall the following result.
\begin{lemma}[\cite{LIO}]\label{lem61}
Weak entropy $\eta(\rho , m)$ defined in (3.1) is convex with respect to $\rho$ and $m$ if and only if $g(\xi)$ is a convex function.
\end{lemma}

Let us briefly repeat the proof of the convexity of $\eta(\rho , m)$ here. It is easy to compute the Hessian matrix of $\eta(\rho , m)$:
\begin{equation}\label{Hessiaeta}
\begin{aligned}
\eta_{\rho\rho}&=\int^{1}_{-1}(\theta^2+\theta) \rho^{\theta-1} g' z\left(1-z^2\right)^\lambda dz  +\rho\int^{1}_{-1}\left(-\frac{m}{\rho^2}+z\theta \rho^{\theta-1}\right)^2 g''\left(1-z^2\right)^\lambda d z\\
&=\int^{1}_{-1}\theta^2\rho^{2\theta-1} g'' \left(1-z^2\right)^{\lambda+1} dz  +\rho\int^{1}_{-1}\left(-\frac{m}{\rho^2}+z\theta \rho^{\theta-1}\right)^2 g''\left(1-z^2\right)^\lambda d z,\\
\eta_{mm}&=\frac{1}{\rho} \int^{1}_{-1} g''\left(1-z^2\right)^\lambda d z,\\
\eta_{\rho m}&=\int^{1}_{-1}\left(-\frac{m}{\rho^2}+z\theta\rho^{\theta-1}\right) g'' \left(1-z^2\right)^\lambda dz.
\end{aligned}
\end{equation}
We see that whenever $g$ is convex, it holds that $\eta_{\rho\rho}\ge0$ and $\eta_{mm}\ge0$,  and also
\begin{align*}
\eta_{\rho\rho}\eta_{mm}-\left(\eta_{\rho m}\right)^2\ge&\int^{1}_{-1}\left(-\frac{m}{\rho^2}+z\theta \rho^{\theta-1}\right)^2 g''\left(1-z^2\right)^\lambda d z\cdot \int^{1}_{-1} g''\left(1-z^2\right)^\lambda d z\\
&-\left[\int^{1}_{-1}\left(-\frac{m}{\rho^2}+z\theta\rho^{\theta-1}\right) g'' \left(1-z^2\right)^\lambda dz\right]^2,
\end{align*}
which is again nonnegative by the Cauchy-Schwarz inequality. Here we have used
$$
\int^{1}_{-1}\theta^2\rho^{2\theta-1} g'' \left(1-z^2\right)^{\lambda+1} dz\ge0,
$$
which is extremely important in our later arguments. Thus, we get the convexity of $\eta$.

Next, for sharp convergence rate, we define
$$
\begin{aligned}
& \eta_*=\eta(v)-\eta(\bar{v})-\nabla \eta(\bar{v})(v-\bar{v}), \\
& q_*=q(v)-q(\bar{v})-\nabla \eta(\bar{v})(f(v)-f(\bar{v})),
\end{aligned}
$$
where $\eta$ and $q$ are any weak convex entropy-flux pair defined in \eqref{etaq},
$$f(v)=\left(m, \frac{m^2}{\rho}\right)^t, f(\bar{v})=\left(\bar{m}, \frac{\bar{m}^2}{\bar{\rho}}\right)^t ,$$
and
\begin{equation}\label{pQ}
\begin{aligned}
P_*&=p(\rho)-p(\bar{\rho})-p'(\bar{\rho})(\rho-\bar{\rho})\ge0,\\
Q_*&=\frac{m^2}{\rho}-\frac{\bar{m}^2}{\bar{\rho}}+\frac{\bar{m}^2}{\bar{\rho}^2}(\rho-\bar{\rho})-\frac{2\bar{m}}{\bar{\rho}}(m-\bar{m})\ge0,
\end{aligned}
\end{equation}
due to the convexity of $p(\rho)$ and $\frac{m^2}{\rho}$. Then, the properties of $\eta_*$ and $q_*$ are stated in the following.

\begin{proposition}\label{lem51}
Assume $\eta_*$ and $q_*$ are defined as above. Let
\begin{align*}
Q_1(v)=m\int^{1}_{-1}g'(1-z^2)^{\lambda}dz.
\end{align*}
Then, under the conditions of Theorem 1.1, it holds that
\begin{equation}\label{detadt}
\begin{aligned}
& \frac{d}{dt}\int_{-\infty}^{+\infty}\eta_{* }dx +\alpha\int_{-\infty}^{+\infty}Q_{1*}dx\\
=&\int_{-\infty}^{+\infty}\left[\frac{\bar{m}\bar{R}}{\bar{\rho}^2}\left(\rho-\bar{\rho}\right)-\frac{\bar{R}}{\bar{\rho}}
\left(m-\bar{m}\right)\right]\int^{1}_{-1}\bar{g}''\left(1-z^2\right)^{\lambda}dz\\
&-\int_{-\infty}^{+\infty}(\rho-\bar{\rho})\frac{\bar{R}}{\bar{\rho}}\theta\bar{\rho}^{\theta}\int^{1}_{-1}\bar{g}''z
\left(1-z^2\right)^{\lambda}dz-\int_{-\infty}^{+\infty}(\eta_{\bar{m}})_x(P_*+Q_*)dx.
\end{aligned}
\end{equation}

\end{proposition}

\begin{proof}
By a direct computation, we have
\begin{equation}\label{etaqk}
\begin{aligned}
& \eta_{* t}+q_{* x} \\
&= {\left[\eta(v)_t+q(v)_x\right]-\left[\eta(\bar{v})_t+q(\bar{v})_x\right]-\left[\nabla \eta(\bar{v})(v-\bar{v})\right] }_t -[ \nabla \eta(\bar{v})(f(v)-f(\bar{v}))]_x\\
&=:k_1+k_2+k_3+k_4.
\end{aligned}
\end{equation}
First, for $k_1$, we deduce
\begin{align*}
\eta(v)_t&=\rho_t\int^{1}_{-1} g(1-z^2)^{\lambda}dz+\rho\int^{1}_{-1}g'\left(\frac{m_t}{\rho}-\frac{m}{\rho^2}\rho_t
+z\theta\rho^{\theta-1}\rho_t\right)(1-z^2)^{\lambda}dz\\
&=\rho_t\int^{1}_{-1}g(1-z^2)^{\lambda}dz+\int^{1}_{-1}g'\left(m_t-\frac{m}{\rho}\rho_t+z\theta\rho^{\theta}\rho_t\right)(1-z^2)^{\lambda}dz,
\end{align*}
and
\begin{align*}
&q(v)_x\\
=&\rho_x\int^{1}_{-1}g\left(\frac{m}{p}+z\theta\rho^{\theta}\right)(1-z^2)^{\lambda}dz+\rho\int^{1}_{-1}g\left(\frac{m_x}{\rho}-\frac{m}
{\rho^2}\rho_x+z\theta^2\rho^{\theta-1}\rho_x\right)(1-z^2)^{\lambda}dz\\
&+\rho\int^{1}_{-1}g'\left(\frac{m_x}{\rho}-\frac{m}{\rho^2}\rho_x+z\theta\rho^{\theta-1}\rho_x
\right)\left(\frac{m}{p}+z\theta\rho^{\theta}\right)(1-z^2)^{\lambda}dz\\
=&\theta\rho^{\theta}\rho_x\int^{1}_{-1}gz(1-z^2)^{\lambda}dz+\int^{1}_{-1}g\left(m_x+z\theta^2\rho^{\theta}\rho_x\right)(1-z^2)^{\lambda}dz\\
&+\int^{1}_{-1}g'\left(\frac{m}{\rho}m_x+m_xz\theta\rho^{\theta}-\frac{m^2}{\rho^2}\rho_x+
z^2\theta^2\rho^{2\theta}\rho_x\right)(1-z^2)^{\lambda}dz.
\end{align*}
Combining the above two equalities and $\rho_t+m_x=0$, we obtain
\begin{equation}\label{k10}
\begin{aligned}
\eta(v)_t+q(v)_x=&(\theta+\theta^2)\rho^{\theta}\rho_x\int^{1}_{-1}gz(1-z^2)^{\lambda}dz\\
&+\int^{1}_{-1}g'\left(m_t+\frac{2 m}{p}m_x-\frac{m^2}{\rho^2}\rho_x+
z^2\theta^2\rho^{2\theta}\rho_x\right)(1-z^2)^{\lambda}dz.
\end{aligned}
\end{equation}
Integrating by parts gives
\begin{align*}
\int^{1}_{-1}gz(1-z^2)^{\lambda}dz=&-\left[\frac{1}{2(\lambda+1)}g(1-z^2)^{\lambda+1}\right]^{1}_{-1}+\frac{1}{2(\lambda+1)}\int^{1}_{-1}
g'\rho^{\theta}(1-z^2)^{\lambda+1}dz\\
=&\frac{\rho^{\theta}}{2(\lambda+1)}\int^{1}_{-1}g'(1-z^2)^{\lambda+1}dz.
\end{align*}
As a result,
\begin{equation}\label{k11}
\begin{aligned}
(\theta+\theta^2)\rho^{\theta}\rho_x\int^{1}_{-1}gz(1-z^2)^{\lambda}dz&=(\theta+\theta^2)
\frac{\rho^{2\theta}\rho_x}{2(\lambda+1)}\int^{1}_{-1}g'(1-z^2)^{\lambda+1}dz\\
&=\theta^2\rho^{2\theta}\rho_x\int^{1}_{-1}g'(1-z^2)^{\lambda+1}dz,
\end{aligned}
\end{equation}
due to
\begin{equation}\label{theta2}
\frac{\theta+\theta^2}{{2(\lambda+1)}}=\frac{\gamma-1}{2}\cdot\frac{\gamma+1}{2}\cdot\frac{\gamma-1}{\gamma+1}=\left(\frac{\gamma-1}{2}\right)^2
=\theta^2.
\end{equation}
By \eqref{4.2}, we derive
$$
m_t+\frac{2 m}{p}m_x-\frac{m^2}{\rho^2}\rho_x=-\alpha m-\kappa\gamma\rho^{\gamma-1}\rho_x=-\alpha m-\theta^2\rho^{2\theta}\rho_x.
$$
Thus, it follows from \eqref{k10}, \eqref{k11} and the above equality that
\begin{equation}\label{k1}
\begin{aligned}
k_1=&\int^{1}_{-1}g'\left(-\alpha m-\theta^2\rho^{2\theta}\rho_x+z^2\theta^2\rho^{2\theta}\rho_x\right)
(1-z^2)^{\lambda}dz+\theta^2\rho^{2\theta}\rho_x\int^{1}_{-1}g'(1-z^2)^{\lambda+1}dz\\
=&-\alpha m\int^{1}_{-1}g'(1-z^2)^{\lambda}dz.
\end{aligned}
\end{equation}

For $k_2$, we note that \eqref{4.3} is similar to \eqref{4.2}, and so get, with $\alpha \bar{m}-\bar{R}$ instead of $-\alpha m$ in \eqref{k1},
\begin{equation}\label{k2}
k_2=-\eta(\bar{v})_t-q(\bar{v})_x=(\alpha \bar{m}-\bar{R})\int^{1}_{-1}\bar{g}'(1-z^2)^{\lambda}dz.
\end{equation}

As for $k_3$, we observe
\begin{equation}\label{k30}
\begin{aligned}
{\left[\nabla \eta(\bar{v})(v-\bar{v})\right] }_t&=[(\eta_{\bar{\rho}})(\rho-\bar{\rho})+(\eta_{\bar{m}})(m-\bar{m})]_t\\
&=(\eta_{\bar{\rho}})_t(\rho-\bar{\rho})+(\eta_{\bar{\rho}})(\rho-\bar{\rho})_t+(\eta_{\bar{m}})_t(m-\bar{m})+(\eta_{\bar{m}})(m-\bar{m})_t.
\end{aligned}
\end{equation}
Since $\rho_t+m_x=0$ and $\bar{\rho}_t+\bar{m}_x=0$, it follows that
\begin{equation}\label{k31}
\begin{aligned}
(\eta_{\bar{\rho}})(\rho-\bar{\rho})_t+(\eta_{\bar{m}})_t(m-\bar{m})&=-(\eta_{\bar{\rho}})(m-\bar{m})_x+(\eta_{\bar{m}})_t(m-\bar{m})\\
&=(\eta_{\bar{\rho}})_x(m-\bar{m})+(\eta_{\bar{m}})_t(m-\bar{m})-\left[(\eta_{\bar{\rho}})(m-\bar{m})\right]_x\\
&=\left[(\eta_{\bar{\rho}})_x+(\eta_{\bar{m}})_t\right](m-\bar{m})-\left[(\eta_{\bar{\rho}})(m-\bar{m})\right]_x.
\end{aligned}
\end{equation}
Using \eqref{4.2}, \eqref{4.3} and \eqref{pQ} yields
\begin{align*}
&-\eta_{\bar{m}}(m-\bar{m})_t\\
=&\eta_{\bar{m}}\left\{\left[\alpha m+\left(\frac{m^2}{\rho}+p(\rho)\right)_x\right]-\left[\alpha \bar{m}-\bar{R}+\left(\frac{\bar{m}^2}{\bar{\rho}}+p(\bar{\rho})\right)_x\right]\right\}\\
=&\eta_{\bar{m}}(\alpha m-\alpha\bar{m}+\bar{R})+\eta_{\bar{m}}\left[\left(\frac{m^2}{\rho}+p(\rho)\right)-\left
(\frac{\bar{m}^2}{\bar{\rho}}+p(\bar{\rho})\right)\right]_x\\
=&\eta_{\bar{m}}(\alpha m-\alpha\bar{m}+\bar{R})+\eta_{\bar{m}}\left(P_*+Q_*+p'(\bar{\rho})(\rho-\bar{\rho})
-\frac{\bar{m}^2}{\bar{\rho}^2}(\rho-\bar{\rho})+\frac{2\bar{m}}{\bar{\rho}}(m-\bar{m})\right)_x\\
=&\eta_{\bar{m}}(\alpha m-\alpha\bar{m}+\bar{R})-(\eta_{\bar{m}})_x\left(p'(\bar{\rho})(\rho-\bar{\rho})
-\frac{\bar{m}^2}{\bar{\rho}^2}(\rho-\bar{\rho})+\frac{2\bar{m}}{\bar{\rho}}(m-\bar{m})\right)\\
&-(\eta_{\bar{m}})_x\left(P_*+Q_*\right)+\left\{\eta_{\bar{m}}\left[\left(\frac{m^2}{\rho}+P(\rho)\right)-\left
(\frac{\bar{m}^2}{\bar{\rho}}+P(\bar{\rho})\right)\right]\right\}_x.
\end{align*}
From this equality, \eqref{k30} and \eqref{k31}, we derive
\begin{equation}\label{k3}
\begin{aligned}
k_3=&-(\eta_{\bar{\rho}})_t(\rho-\bar{\rho})-\left[(\eta_{\bar{\rho}})_x+(\eta_{\bar{m}})_t\right](m-\bar{m})+
\left[(\eta_{\bar{\rho}})(m-\bar{m})\right]_x\\
&+\eta_{\bar{m}}(\alpha m-\alpha\bar{m}+\bar{R})-(\eta_{\bar{m}})_x\left(p'(\bar{\rho})(\rho-\bar{\rho})
-\frac{\bar{m}^2}{\bar{\rho}^2}(\rho-\bar{\rho})+\frac{2\bar{m}}{\bar{\rho}}(m-\bar{m})\right)\\
&-(\eta_{\bar{m}})_x\left(P_*+Q_*\right)+\left\{\eta_{\bar{m}}\left[\left(\frac{m^2}{\rho}+P(\rho)\right)-\left
(\frac{\bar{m}^2}{\bar{\rho}}+P(\bar{\rho})\right)\right]\right\}_x\\
=&\left[-(\eta_{\bar{\rho}})_t-(\eta_{\bar{m}})_xp'(\bar{\rho})+(\eta_{\bar{m}})_x\frac{\bar{m}^2}{\bar{\rho}^2}\right](\rho-\bar{\rho})
+\left[(\eta_{\bar{\rho}})(m-\bar{m})\right]_x\\
&+\left[-(\eta_{\bar{\rho}})_x-(\eta_{\bar{m}})_t-\frac{2\bar{m}}{\bar{\rho}}(\eta_{\bar{m}})_x\right](m-\bar{m})+\eta_{\bar{m}}(\alpha m-\alpha\bar{m}+\bar{R})\\
&-(\eta_{\bar{m}})_x\left(P_*+Q_*\right)+\left\{\eta_{\bar{m}}\left[\left(\frac{m^2}{\rho}+P(\rho)\right)-\left
(\frac{\bar{m}^2}{\bar{\rho}}+P(\bar{\rho})\right)\right]\right\}_x.
\end{aligned}
\end{equation}
Observing $\eta(\bar{v})=\bar{\rho} \int_{-1}^1 \bar{g}\left(1-z^2\right)^\lambda dz$, we deduce
$$
\eta_{\bar{\rho}}=\int^{1}_{-1}\bar{g}\left(1-z^2\right)^{\lambda}dz+\int^{1}_{-1}\bar{g}'
\left(-\frac{\bar{m}}{\bar{\rho}}+z\theta\bar{\rho}^{\theta}\right)\left(1-z^2\right)^{\lambda}dz,
$$
and
$$
\eta_{\bar{m}}=\int^{1}_{-1}\bar{g}'\left(1-z^2\right)^{\lambda}dz.
$$
As a result,
\begin{equation}\label{4k31}
\begin{aligned}
&(\eta_{\bar{m}})_x\frac{\bar{m}^2}{\bar{\rho}^2}-(\eta_{\bar{m}})_xp'(\bar{\rho})-(\eta_{\bar{\rho}})_t\\
=&\left[\int^{1}_{-1}\bar{g}'\left(1-z^2\right)^{\lambda}dz\right]_x\frac{\bar{m}^2}{\bar{\rho}^2}-\left[\int^{1}_{-1}\bar{g}'
\left(1-z^2\right)^{\lambda}dz\right]_xp'(\bar{\rho})\\
&-\left[\int^{1}_{-1}\bar{g}\left(1-z^2\right)^{\lambda}dz+
\int^{1}_{-1}\bar{g}'\left(-\frac{\bar{m}}{\bar{\rho}}+z\theta\bar{\rho}^{\theta}\right)\left(1-z^2\right)^{\lambda}dz\right]_t\\
=&\frac{\bar{m}^2}{\bar{\rho}^2}\int^{1}_{-1}\bar{g}''\left(\frac{\bar{m}_x}{\rho}-\frac{\bar{m}}{\bar{\rho}^2}\bar{\rho}_x+z\theta\bar{\rho}
^{\theta-1}\bar{\rho}_x\right)\left(1-z^2\right)^{\lambda}dz-(\theta+\theta^2)\bar{\rho}^{\theta-1}\bar{\rho}_t
\int^{1}_{-1}\bar{g}'\left(1-z^2\right)^{\lambda}zdz\\
&-p'(\bar{\rho})\int^{1}_{-1}\bar{g}''\left(\frac{\bar{m}_x}{\rho}-\frac{\bar{m}}{\bar{\rho}^2}\bar{\rho}_x+
z\theta\bar{\rho}^{\theta-1}\bar{\rho}_x\right)\left(1-z^2\right)^{\lambda}dz\\
&-\int^{1}_{-1}\bar{g}''\left(\frac{\bar{m}_t}{\rho}-\frac{\bar{m}}{\bar{\rho}^2}\bar{\rho}_t+z\theta\bar{\rho}^{\theta-1}\bar{\rho}_t\right)
\left(-\frac{\bar{m}}{\bar{\rho}}+z\theta\bar{\rho}^{\theta}\right)\left(1-z^2\right)^{\lambda}dz\\
=&-(\theta+\theta^2)\bar{\rho}^{\theta-1}\bar{\rho}_t\int^{1}_{-1}\bar{g}'\left(1-z^2\right)^{\lambda}zdz+\left(-\alpha\bar{m}+\bar{R}\right)
\frac{\bar{m}}{\bar{\rho}^2}\int^{1}_{-1}\bar{g}''\left(1-z^2\right)^{\lambda}dz\\
&+\left(\alpha\bar{m}-\bar{R}\right)\theta\bar{\rho}^{\theta-1}\int^{1}_{-1}\bar{g}''z\left(1-z^2\right)^{\lambda}dz+
\theta^2\bar{\rho}^{2\theta-1}\bar{\rho}_t\int^{1}_{-1}\bar{g}''\left(1-z^2\right)^{\lambda+1}dz\\
=&\left(-\alpha\bar{m}+\bar{R}\right)\frac{\bar{m}}{\bar{\rho}^2}\int^{1}_{-1}\bar{g}''\left(1-z^2\right)^{\lambda}dz
+\left(\alpha\bar{m}-\bar{R}\right)\theta\bar{\rho}^{\theta-1}\int^{1}_{-1}\bar{g}''z\left(1-z^2\right)^{\lambda}dz,
\end{aligned}
\end{equation}
by \eqref{4.3} and
\begin{align*}
(\theta+\theta^2)\bar{\rho}^{\theta-1}\bar{\rho}_t\int^{1}_{-1}\bar{g}'\left(1-z^2\right)^{\lambda}zdz=&-\frac{(\theta+\theta^2)}{2(\lambda+1)}
\bar{\rho}^{\theta-1}\bar{\rho}_t\bar{g}'\left(1-z^2\right)^{\lambda+1}\big|^{1}_{-1}\\
&+\frac{(\theta+\theta^2)}{2(\lambda+1)}\bar{\rho}^{2\theta-1}\bar{\rho}_t\int^{1}_{-1}\bar{g}''\left(1-z^2\right)^{\lambda+1}dz\\
=&\theta^2\bar{\rho}^{2\theta-1}\bar{\rho}_t\int^{1}_{-1}\bar{g}''\left(1-z^2\right)^{\lambda+1}dz,
\end{align*}
due to \eqref{theta2}.

Similarly,
\begin{equation}\label{4k32}
\begin{aligned}
&\frac{2\bar{m}}{\bar{\rho}}(\eta_{\bar{m}})_x+(\eta_{\bar{m}})_t+(\eta_{\bar{\rho}})_x\\
=&\left[\int^{1}_{-1}\bar{g}'\left(1-z^2\right)^{\lambda}dz\right]_x\frac{2\bar{m}}{\bar{\rho}}+\left[\int^{1}_{-1}\bar{g}'
\left(1-z^2\right)^{\lambda}dz\right]_t\\
&+\left[\int^{1}_{-1}\bar{g}\left(1-z^2\right)^{\lambda}dz+
\int^{1}_{-1}\bar{g}'\left(-\frac{\bar{m}}{\bar{\rho}}+z\theta\bar{\rho}^{\theta}\right)\left(1-z^2\right)^{\lambda}dz\right]_x\\
=&(\theta+\theta^2)\bar{\rho}^{\theta-1}\bar{\rho}_x\int^{1}_{-1}
\bar{g}'\left(1-z^2\right)^{\lambda}zdz+\frac{-\alpha\bar{m}+\bar{R}}{\bar{\rho}}\int^{1}_{-1}\bar{g}''\left(1-z^2\right)^{\lambda}dz\\
&-\theta^2\bar{\rho}^{2\theta-1}\bar{\rho}_x\int^{1}_{-1}\bar{g}''\left(1-z^2\right)^{\lambda+1}dz\\
=&\frac{-\alpha\bar{m}+\bar{R}}{\bar{\rho}}\int^{1}_{-1}\bar{g}''\left(1-z^2\right)^{\lambda}dz.
\end{aligned}
\end{equation}
Hence, combining \eqref{4k31}, \eqref{4k32} and \eqref{k3}, we derive
\begin{equation}\label{k3'}
\begin{aligned}
k_3=&\left[\frac{\bar{m}\bar{R}}{\bar{\rho}^2}\left(\rho-\bar{\rho}\right)-\frac{\bar{R}}{\bar{\rho}}\left(m-\bar{m}\right)\right]
\int^{1}_{-1}\bar{g}''\left(1-z^2\right)^{\lambda}dz-(\eta_{\bar{m}})_x\left(P_*+Q_*\right)\\
&+\alpha\left[\frac{\bar{m}}{\bar{\rho}}\left(m-\bar{m}\right)-\frac{\bar{m}^2}{\bar{\rho}^2}\left(\rho-\bar{\rho}\right)\right]
\int^{1}_{-1}\bar{g}''\left(1-z^2\right)^{\lambda}dz+ \left[(\eta_{\bar{\rho}})(m-\bar{m})\right]_x\\
&+\left(\rho-\bar{\rho}\right)\left(\alpha\bar{m}-\bar{R}\right)\theta\bar{\rho}^{\theta-1}\int^{1}_{-1}\bar{g}''z\left(1-z^2\right)^{\lambda}dz
+\eta_{\bar{m}}(\alpha m-\alpha\bar{m}+\bar{R})\\
&+\left\{\eta_{\bar{m}}\left[\left(\frac{m^2}{\rho}+P(\rho)\right)-\left
(\frac{\bar{m}^2}{\bar{\rho}}+P(\bar{\rho})\right)\right]\right\}_x.
\end{aligned}
\end{equation}

Later, plugging \eqref{k1}, \eqref{k2} and \eqref{k3'} into \eqref{etaqk}, we obtain
\begin{align}\label{etaqk123}
& \eta_{* t}+q_{* x} \nonumber \\
=&\left[\frac{\bar{m}\bar{R}}{\bar{\rho}^2}\left(\rho-\bar{\rho}\right)-\frac{\bar{R}}{\bar{\rho}}\left(m-\bar{m}\right)\right]
\int^{1}_{-1}\bar{g}''\left(1-z^2\right)^{\lambda}dz-(\rho-\bar{\rho})\frac{\bar{R}}{\bar{\rho}}\theta\bar{\rho}^{\theta}\int^{1}_{-1}\bar{g}''
z\left(1-z^2\right)^{\lambda}dz \nonumber \\
&-(\eta_{\bar{m}})_x(P_*+Q_*)-\alpha m\int^{1}_{-1}g'(1-z^2)^{\lambda}dz+\alpha\bar{m}\int^{1}_{-1}\bar{g}'(1-z^2)^{\lambda}dz \nonumber \\
&+\alpha\bar{m}(\rho-\bar{\rho})\theta\bar{\rho}^{\theta-1}\int^{1}_{-1}\bar{g}''z(1-z^2)^{\lambda}dz
-\alpha\frac{\bar{m}^2}{\bar{\rho^2}}(\rho-\bar{\rho})\int^{1}_{-1}\bar{g}''(1-z^2)^{\lambda}dz  \nonumber \\
&+\alpha(m-\bar{m})\int^{1}_{-1}\bar{g}'(1-z^2)^{\lambda}dz+\alpha\frac{\bar{m}}{\bar{\rho}}(m-\bar{m})
\int^{1}_{-1}\bar{g}''z(1-z^2)^{\lambda}dz  \nonumber \\
&+\left[(\eta_{\bar{\rho}})(m-\bar{m})\right]_x+\left\{\eta_{\bar{m}}\left[\left(\frac{m^2}{\rho}+P(\rho)\right)-\left
(\frac{\bar{m}^2}{\bar{\rho}}+P(\bar{\rho})\right)\right]\right\}_x \nonumber \\
&-[ \nabla \eta(\bar{v})(f(v)-f(\bar{v}))]_x.
\end{align}
From the definition of $Q_{1*}$, we get
\begin{align*}
Q_{1*}=&Q_1(v)-Q_1(\bar{v})-\nabla Q_1(\bar{v})(v-\bar{v})\\
=&m\int^{1}_{-1}g'(1-z^2)^{\lambda}dz-\bar{m}\int^{1}_{-1}\bar{g}'(1-z^2)^{\lambda}dz\\
&-\bar{m}(\rho-\bar{\rho})\theta\bar{\rho}^{\theta-1}\int^{1}_{-1}\bar{g}''z(1-z^2)^{\lambda}dz
+\frac{\bar{m}^2}{\bar{\rho^2}}(\rho-\bar{\rho})\int^{1}_{-1}\bar{g}''(1-z^2)^{\lambda}dz\\
&-(m-\bar{m})\int^{1}_{-1}\bar{g}'(1-z^2)^{\lambda}dz-\frac{\bar{m}}{\bar{\rho}}(m-\bar{m})
\int^{1}_{-1}\bar{g}''z(1-z^2)^{\lambda}dz.
\end{align*}
Integrating \eqref{etaqk123} over $R$, and using the above equality yields
\begin{align*}
& \frac{d}{dt}\int_{-\infty}^{+\infty}\eta_{* }dx +\alpha\int_{-\infty}^{+\infty}Q_{1*}dx\\
=&\int_{-\infty}^{+\infty}\left[\frac{\bar{m}\bar{R}}{\bar{\rho}^2}\left(\rho-\bar{\rho}\right)-\frac{\bar{R}}{\bar{\rho}}
\left(m-\bar{m}\right)\right]\int^{1}_{-1}\bar{g}''\left(1-z^2\right)^{\lambda}dz\\
&-\int_{-\infty}^{+\infty}(\rho-\bar{\rho})\frac{\bar{R}}{\bar{\rho}}\theta\bar{\rho}^{\theta}\int^{1}_{-1}\bar{g}''z
\left(1-z^2\right)^{\lambda}dz-\int_{-\infty}^{+\infty}(\eta_{\bar{m}})_x(P_*+Q_*)dx,
\end{align*}
and this completes the proof.
\end{proof}

\section{Proofs}

In this section, we detail the proof of Theorems \ref{thm1}. For $\gamma\ge2$ and $1<\gamma<\frac{9}{7}$, different entropy-flux pairs $(\eta, q)$ with different functions $g$ will be used, respectively.

\subsection{Proof of Theorem \ref{thm1}-(i)}

In this subsection, we take $g(\xi)=\frac{1}{2}|\xi|^2$.

\begin{proposition}\label{thm3}
Under the conditions of Theorem \ref{thm1}-(i), it holds that for any $t>0$ and $1<\gamma<+\infty$,
\begin{equation}\label{PQC}
\int^{+\infty}_{-\infty}(P_*+Q_*)dx+\int^{t}_{0}\int^{+\infty}_{-\infty}Q_*dx\leq C,
\end{equation}
where $P_*$ and $Q_*$ are defined as in \eqref{pQ}.

Particularly, if $2\le\gamma<+\infty$, then
\begin{equation}\label{gamma1/2}
\begin{aligned}
&\|(\rho-\bar{\rho})(\cdot, t)\|_{L^\gamma}^\gamma+\|(m-\bar{m})(\cdot, t)\|^2 \\ &+\int_{-\infty}^{+\infty}\left(\rho^{\gamma-2}+\bar{\rho}^{\gamma-2}\right)(\rho-\bar{\rho})^2 d x \leq C(1+t)^{-\frac{\gamma^2+\gamma-1}{(\gamma+1)^2}+2 \varepsilon},
\end{aligned}
\end{equation}
where the positive constant $\varepsilon$ can be arbitrarily small.
\end{proposition}

\begin{proof}
Since $g(\xi)=\frac{1}{2}|\xi|^2$ and $\int^{1}_{-1}z(1-z^2)^{\lambda}dz=0$, we have
\begin{equation}\label{etag2}
\eta=\rho\int^{1}_{-1} g\left(1-z^2\right)^\lambda dz=\frac{1}{2}\rho\int^{1}_{-1} \left(u^2+z^2 \rho^{2\theta}\right)\left(1-z^2\right)^\lambda dz
=C_1\frac{m^2}{\rho}+C_2\rho^{\gamma},
\end{equation}
where
$$
C_1=\frac{1}{2}\int_{-1}^1\left(1-z^2\right)^\lambda dz,
$$
and
$$
C_2=\frac{1}{2}\int_{-1}^1z^2\left(1-z^2\right)^\lambda dz.
$$
Therefore, using \eqref{etag2} and \eqref{pQ} gives
\begin{equation}\label{etaPQ}
\eta_*=C_1 Q_*+\frac{C_2}{\kappa}P_*.
\end{equation}

In addition, noting $g'(\xi)|_{\xi=u+z\rho^{\theta}}=u+z\rho^{\theta}$ and $g''(\xi)=1$,we infer
$$
\left(\eta_{\bar{m}}\right)_x=2C_1\left(\frac{\bar{m}}{\bar{\rho}}\right)_x=2C_1\bar{u}_x=\frac{2C_1}{(1+t)(1+\gamma)}, \, |x|<{\sqrt{\frac{A_0}{B_0}}(1+t)^{\frac{1}{\gamma+1}}}
$$
by \eqref{ux}, and
$$
Q_1=m\int^{1}_{-1} g'(1-z^2)^{\lambda}dz=\frac{m^2}{\rho}\int^{1}_{-1} (1-z^2)^{\lambda}dz=2C_1 \frac{m^2}{\rho}.
$$
This and \eqref{pQ} together show
$$
Q_{1^*}=2C_1Q_*,
$$
which implies that \eqref{detadt} becomes
\begin{equation}\label{detadt'}
\begin{aligned}
& \frac{d}{dt}\int_{-\infty}^{+\infty}\eta_{* }dx +2C_1\alpha\int_{-\infty}^{+\infty}Q_{*}dx+\frac{2C_1}{(1+t)(1+\gamma)}\int_{-\sqrt{\frac{A_0}{B_0}}(1+t)^{\frac{1}{\gamma+1}}}
^{\sqrt{\frac{A_0}{B_0}}(1+t)^{\frac{1}{\gamma+1}}}(P_*+Q_*)dx\\
=&2C_1 \int_{-\infty}^{+\infty}\left\{\bar{u}\left[\bar{\rho}^{\delta}\frac{\bar{R}}{\bar{\rho}}\right]
\bar{\rho}^{-\delta}\left(\rho-\bar{\rho}\right)-\left[\bar{\rho}^{\delta}\frac{\bar{R}}{\bar{\rho}}\right]
\bar{\rho}^{-\delta}\left(m-\bar{m}\right)\right\}dx\\
=:&I_1+I_2.
\end{aligned}
\end{equation}
Here, the representation $\bar{\rho}^\delta\left(\frac{\bar{R}}{\bar{\rho}}\right)$ is used to avoid $\delta$ function for $\frac{\bar{R}}{\bar{\rho}}$ on the interface between the gas and vacuum, and will be helpful in the later reasoning.

Then, it follows from \eqref{rhobeta} and \eqref{rhodelta} that
\begin{equation}\label{I1}
\begin{aligned}
I_1&=2C_1\int_{-\infty}^{+\infty}  \bar{u}\left[\bar{\rho}^\delta\left(\frac{\bar{R}}{\bar{\rho}}\right)\right] \bar{\rho}^{-\delta}(\rho-\bar{\rho})dx\\
&\leq C\left\|\bar{\rho}^\delta \frac{\bar{R}}{\bar{\rho}}\right\|\left\|\bar{\rho}^{-\delta} \bar{u}\right\|\|\rho-\bar{\rho}\|_{L^{\infty}}\\
&\leq C(1+t)^{-\frac{3\gamma}{\gamma+1}},
\end{aligned}
\end{equation}
and
\begin{equation}\label{I2}
\begin{aligned}
I_2&=-2C_1\int_{-\infty}^{+\infty}  \left[\bar{\rho}^\delta\left(\frac{\bar{R}}{\bar{\rho}}\right)\right] \bar{\rho}^{-\delta}(m-\bar{m})dx\\
&\leq \left\|\bar{\rho}^\delta \frac{\bar{R}}{\bar{\rho}}\right\|_{L^4}\left\|\bar{\rho}^{-\delta}\right\|_{L^4}\left\|y_t\right\|\\
&\leq C(1+t)^{-\frac{4\gamma+1}{2(\gamma+1)}}\|y_t\|\\
&\leq C(1+t)^{-1}\|y_t\|^2+C(1+t)^{-\frac{3\gamma}{\gamma+1}}.
\end{aligned}
\end{equation}
Substituting the above two inequalities into \eqref{detadt'}, then integrating over $(0,t)$, and using Lemma \ref{lem3}, we obtain
\begin{align*}
&\int_{-\infty}^{+\infty}\eta_{* }dx +C_1\int^{t}_{0}\int_{-\infty}^{+\infty}Q_{*}dxd\tau\\
\leq &C+C\int^{t}_{0}(1+\tau)^{-\frac{3\gamma}{\gamma+1}}d\tau+C\int^{t}_{0}(1+\tau)^{-1}\|y_{\tau}\|^2d\tau\leq C.
\end{align*}
This, in combination with \eqref{etaPQ} and $P_*\ge0, Q_*\ge0$, gives \eqref{PQC}.

As for the second conclusion, setting
\begin{equation}\label{kvar}
k(\varepsilon)=\frac{\gamma^2+\gamma-1}{(\gamma+1)^2}-2\varepsilon,
\end{equation}
and multiplying \eqref{detadt'} with $(1+t)^{k(\varepsilon)}$, we have
\begin{align*}
&\frac{d}{d t}\left[(1+t)^{k(\varepsilon)} \int_{-\infty}^{+\infty} \eta_*(x, t) d x\right] +2C_1\alpha(1+t)^{k(\varepsilon)}\int_{-\infty}^{+\infty}Q_* dx\\
\leq&(1+t)^{k(\varepsilon)}(I_1+I_2)+k(\varepsilon)(1+t)^{k(\varepsilon)-1} \int_{-\infty}^{+\infty} \eta_*(x, t) d x,\\
\leq&(1+t)^{k(\varepsilon)-\frac{3\gamma}{\gamma+1}}+C(1+t)^{k(\varepsilon)-1}\|y_t\|^2+k(\varepsilon)(1+t)^{k(\varepsilon)-1} \int_{-\infty}^{+\infty} \eta_*(x, t) d x,
\end{align*}
by \eqref{I1} and \eqref{I2}. We then integrate over $(0,t)$ and use Lemma \ref{lem3}, giving
\begin{equation}\label{etakv}
\begin{aligned}
&(1+t)^{k(\varepsilon)} \int_{-\infty}^{+\infty} \eta_*(x, t)dx+C\int^{t}_{0}\int_{-\infty}^{+\infty}(1+\tau)^{k(\varepsilon)}Q_*dxd\tau\\
\leq& C+C\int^{t}_{0}\int_{-\infty}^{+\infty}(1+\tau)^{k(\varepsilon)-1}(P_*+Q_*)dxd\tau\\
\leq& C+C\int^{t}_{0}\int_{-\infty}^{+\infty}(1+\tau)^{k(\varepsilon)-1}\bar{\rho}^{\delta}P_*\bar{\rho}^{-\delta}dxd\tau,
\end{aligned}
\end{equation}
due to \begin{align*}
k(\varepsilon)-\frac{3\gamma}{\gamma+1}=-\frac{2\gamma^2+2\gamma+1}{(\gamma+1)^2}<-1,
\end{align*}
\eqref{PQC}, and $P_*,Q_*\ge0$.

Next, from Lemmas \ref{lem1} and \ref{lem2}, and \eqref{pQ}, we derive
\begin{align*}
P_*=p(\rho)-p(\bar{\rho})-p'(\bar{\rho})(\rho-\bar{\rho})&\leq C\left[\rho^{\gamma+1}-\bar{\rho}^{\gamma+1}-(\gamma+1)\bar{\rho}^{\gamma}\left(\rho-\bar{\rho}\right)\right]^{\frac{\gamma}{\gamma+1}}\\
&\leq C\left[\left(\rho^{\gamma}-\bar{\rho}^{\gamma}\right)\left(\rho-\bar{\rho}\right)\right]^{\frac{\gamma}{\gamma+1}}.
\end{align*}
Hence, it follows from Lemma \ref{lem4} that
\begin{align*}
\int_{-\infty}^{+\infty}\bar{\rho}^{\delta}P_*\bar{\rho}^{-\delta}dx&\leq C\int_{-\infty}^{+\infty}\bar{\rho}^{\delta}\left[\left(\rho^{\gamma}-\bar{\rho}^{\gamma}\right)
\left(\rho-\bar{\rho}\right)\right]^{\frac{\gamma}{\gamma+1}}\bar{\rho}^{-\delta}dx\\
&\leq C\left(\int_{-\infty}^{+\infty}\bar{\rho}^{\frac{\gamma+1}{\gamma}\delta}
\left[\left(\rho^{\gamma}-\bar{\rho}^{\gamma}\right)\left(\rho-\bar{\rho}\right)\right]  dx\right)^{\frac{\gamma}{\gamma+1}}    \left(\int_{-\infty}^{+\infty}\bar{\rho}^{-\delta(\gamma+1)}dx\right)^{\frac{1}{\gamma+1}}\\
&\leq C||\bar{\rho}^{\frac{\gamma+1}{\gamma}\delta}||^{\frac{\gamma}{\gamma+1}}_{L^{\infty}}\left(\int_{-\infty}^{+\infty}\left(\rho^{\gamma}-
\bar{\rho}^{\gamma}\right)\left(\rho-\bar{\rho}\right)  dx\right)^{\frac{\gamma}{\gamma+1}}\cdot (1+t)^{\frac{\delta+1}{(\gamma+1)^2}}\\
&\leq C(1+t)^{\frac{1}{(\gamma+1)^2}}\left(\int_{-\infty}^{+\infty}\left(\rho^{\gamma}-\bar{\rho}^{\gamma}\right)\left(\rho-\bar{\rho}\right)  dx\right)^{\frac{\gamma}{\gamma+1}},
\end{align*}
where the last inequality is from \eqref{rhobeta'}. As a result, we deduce that
\begin{align*}
&\int^{t}_{0}\int_{-\infty}^{+\infty}(1+\tau)^{k(\varepsilon)-1}\bar{\rho}^{\delta}P_*\bar{\rho}^{-\delta}dxd\tau\\
\leq &C\int^{t}_{0}(1+\tau)^{k(\varepsilon)-1+\frac{1}{(\gamma+1)^2}}
\left(\int_{-\infty}^{+\infty}\left(\rho^{\gamma}-\bar{\rho}^{\gamma}\right)\left(\rho-\bar{\rho}\right)    dx\right)^{\frac{\gamma}{\gamma+1}}d\tau\\
\leq &C\int^{t}_{0}(1+\tau)^{\left[k(\varepsilon)-1+\frac{1}{(\gamma+1)^2}\right]\cdot(\gamma+1)}d\tau+
\int^{t}_{0}\int_{-\infty}^{+\infty}\left(\rho^{\gamma}-\bar{\rho}^{\gamma}\right)\left(\rho-\bar{\rho}\right)   dxd\tau\\
\leq& C\int^{t}_{0}(1+\tau)^{-1-2(\gamma+1)\varepsilon}d\tau+C\leq C,
\end{align*}
by Lemma \ref{lem3} again.

Thus, combining the above estimate and \eqref{etakv}, we see that
\begin{equation*}
\int_{-\infty}^{+\infty}\eta_* dx\leq C(1+t)^{-k(\varepsilon)}.
\end{equation*}
On one hand, this and \eqref{etaPQ} together gives, by Lemma \ref{lem2},
\begin{equation}\label{rhogamma}
\|(\rho-\bar{\rho})(\cdot, t)\|_{L^\gamma}^\gamma+\int_{-\infty}^{+\infty}\left(\rho^{\gamma-2}+\bar{\rho}^{\gamma-2}\right)\left(\rho-\bar{\rho}\right)^2 d x \leq C(1+t)^{-\frac{\gamma^2+\gamma-1}{(\gamma+1)^2}+2 \varepsilon}.
\end{equation}
On the other hand, it follows from Lemma \ref{thm2} that, with a constant $\bar{C}_3>0$,
\begin{equation}\label{barC3}
\begin{aligned}
\int_{-\infty}^{+\infty}(m-\bar{m})^2 \mathrm{~d} x & \leq C \int_{-\infty}^{+\infty}\left(m^2+\bar{m}^2\right) d x \\
& \leq \bar{C}_3 \int_{-\infty}^{+\infty} \frac{m^2}{\rho}dx+\int_{-\infty}^{+\infty}\bar{u}^2\bar{\rho}^2 d x \\
& \leq \bar{C}_3 \int_{-\infty}^{+\infty} \frac{m^2}{\rho}dx+C(1+t)^{-\frac{2\gamma+1}{\gamma+1}},
\end{aligned}
\end{equation}
where the last inequality is from Lemma \ref{lem4} with $\beta_1=\beta_2=2$ and $p=1$.

Then, by virtue of \eqref{PQC}, Young's inequality and Lemma \ref{lem4} again, it holds that
\begin{align*}
\int_{-\infty}^{+\infty} \frac{m^2}{\rho}dx&=\int_{-\infty}^{+\infty}\left[Q_*+\frac{\bar{m}^2}{\bar{\rho}}-\frac{\bar{m}^2}{\bar{\rho}^2}(\rho-\bar{\rho})+
\frac{2\bar{m}}{\bar{\rho}}(m-\bar{m})\right]dx\\
&\leq \int_{-\infty}^{+\infty}Q_*dx+\left\|\bar{u}^{2}\bar{\rho}\right\|_{L^{1}}+
\left\|\bar{u}^2\right\|_{L^{\infty}}\|\rho-\bar{\rho}\|_{L^1}+2\left\|\bar{u}\right\|\|m-\bar{m}\|\\
&\leq C(1+t)^{-k(\varepsilon)}+C(1+t)^{-\frac{2\gamma-1}{\gamma+1}}+\frac{1}{2\bar{C}_3}\|m-\bar{m}\|^2\\
&\leq C(1+t)^{-k(\varepsilon)}+\frac{1}{2\bar{C}_3}\|m-\bar{m}\|^2.
\end{align*}
Thus, by \eqref{barC3} we obtain
\begin{align*}
\int_{-\infty}^{+\infty}(m-\bar{m})^2 dx &\leq \bar{C}_3 \int_{-\infty}^{+\infty} \frac{m^2}{\rho}dx+C(1+t)^{-\frac{2\gamma+1}{\gamma+1}}\\
&\leq C(1+t)^{-k(\varepsilon)}+\frac{1}{2}\|m-\bar{m}\|^2,
\end{align*}
which means that
$$
\int_{-\infty}^{+\infty}(m-\bar{m})^2 dx \leq C(1+t)^{-k(\varepsilon)}.
$$
Therefore, using \eqref{kvar} and \eqref{rhogamma} together justifies Proposition \ref{thm3}.

\end{proof}

In view of Lemma \ref{lem2}, Lemma \ref{lem3}, \eqref{rhobeta} and \eqref{gamma1/2}, we obtain the $L^1$ convergence rate on density as follows
\begin{align*}
\int_{-\infty}^{+\infty}|\rho-\bar{\rho}| d x &\leq\left(\int_{-\infty}^{+\infty} \bar{\rho}^{\gamma-2}|\rho-\bar{\rho}|^2 d x\right)^{\frac{1}{2}}\left(\int_{-\infty}^{+\infty} \bar{\rho}^{2-\gamma} d x\right)^{\frac{1}{2}}\\
&\leq C(1+t)^{-\frac{\gamma^2+\gamma-1}{2(\gamma+1)^2}+ \varepsilon}(1+t)^{\frac{\gamma-1}{2(\gamma+1)}}\\
&\leq C(1+t)^{-\frac{\gamma}{2(\gamma+1)^2}+\varepsilon}.
\end{align*}
This combined with \eqref{gamma1/2} completes the proof of Theorem \ref{thm1}-(i).

\subsection{Proof of Theorem \ref{thm1}-(ii)}

In this subsection, we take $g(\xi)=\frac{(\gamma-1)^2}{2\gamma(\gamma+1)}|\xi|^{\frac{2\gamma}{\gamma-1}}$. Then
$
g'(\xi)=\frac{2\gamma}{\gamma-1}\frac{g(\xi)}{\xi},
$
and hence
\begin{align*}
\eta(\rho, m)&=\rho \int^{1}_{-1}  g\left(1-z^2\right)^\lambda \mathrm{d} z\\
&=\frac{\gamma-1}{2\gamma}\int^{1}_{-1} g'\rho\left(u+z \rho^\theta\right)\left(1-z^2\right)^\lambda \mathrm{d}z \\ &=\frac{\gamma-1}{2\gamma}\left(m\int^{1}_{-1}g'\left(1-z^2\right)^\lambda \mathrm{d} z+\rho^{\theta+1}\int^{1}_{-1}g'z\left(1-z^2\right)^\lambda \mathrm{d} z\right)\\
&=:\frac{\gamma-1}{2\gamma}(A+B).
\end{align*}
Thus
\begin{equation}\label{A*B*}
\eta_*=\frac{\gamma-1}{2\gamma}(A_*+B_*).
\end{equation}
From the definition of $Q_1$ (in Proposition \ref{lem51}), we find that $Q_1=A$ and \eqref{detadt} becomes
\begin{equation}\label{J123}
\begin{aligned}
& \frac{d}{dt}\int_{-\infty}^{+\infty}\eta_{* }dx +\alpha\int_{-\infty}^{+\infty}A_*dx\\
=&\int_{-\infty}^{+\infty}\left\{\bar{u}\left[\bar{\rho}^\delta\left(\frac{\bar{R}}{\bar{\rho}}\right)\right] \bar{\rho}^{-\delta}(\rho-\bar{\rho})-\left[\bar{\rho}^\delta\left(\frac{\bar{R}}{\bar{\rho}}\right)\right] \bar{\rho}^{-\delta}(m-\bar{m})\right\}\int^{1}_{-1}\bar{g}''\left(1-z^2\right)^{\lambda}dzdx\\
&-\theta\int_{-\infty}^{+\infty}\left[\bar{\rho}^\delta\left(\frac{\bar{R}}{\bar{\rho}}\right)\right] \bar{\rho}^{-\delta}(\rho-\bar{\rho})\int^{1}_{-1}
\bar{g}''z\bar{\rho}^{\theta}(1-z^2)^{\lambda}dzdx\\
&-\int_{-\infty}^{+\infty}
\left[\bar{\rho}^{\delta}(\eta_{\bar{m}})_x\right]\bar{\rho}^{-\delta}(P_*+Q_*)dx\\
=:&J_1+J_2+J_3.
\end{aligned}
\end{equation}
\subsubsection{Estimates of $J_1$, $J_2$, and $J_3$}
First, since $g''(\xi)=|\xi|^{\frac{2}{\gamma-1}}$ and $\lambda=\frac{3-\gamma}{2(\gamma-1)}>0$,  we have
\begin{align*}
\left|\int^{1}_{-1}\bar{g}''\left(1-z^2\right)^{\lambda}dz\right|\leq \int^{1}_{-1} \left|\bar{g}''\left(1-z^2\right)^{\lambda}\right|dz\leq C\int^{1}_{-1} |\bar{u}+z\bar{\rho}^\theta|^{\frac{2}{\gamma-1}}dz\leq C(|\bar{u}|^{\frac{2}{\gamma-1}}+|\bar{\rho}|),
\end{align*}
and so
\begin{equation}\label{J10}
\begin{aligned}
J_1\leq&\bigg|\int_{-\infty}^{+\infty}\left\{\bar{u}\left[\bar{\rho}^\delta\left(\frac{\bar{R}}{\bar{\rho}}\right)\right] \bar{\rho}^{-\delta}(\rho-\bar{\rho})-\left[\bar{\rho}^\delta\left(\frac{\bar{R}}{\bar{\rho}}\right)\right] \bar{\rho}^{-\delta}(m-\bar{m})\right\}\int^{1}_{-1}\bar{g}''\left(1-z^2\right)^{\lambda}dzdx\bigg|\\
\leq &\int_{-\infty}^{+\infty}\left\{\left|\bar{u}\left[\bar{\rho}^\delta\left(\frac{\bar{R}}{\bar{\rho}}\right)\right] \bar{\rho}^{-\delta}(\rho-\bar{\rho})\right|+\left|\left[\bar{\rho}^\delta\left(\frac{\bar{R}}{\bar{\rho}}\right)\right] \bar{\rho}^{-\delta}(m-\bar{m})\right|\right\}\left|\int^{1}_{-1}\bar{g}''\left(1-z^2\right)^{\lambda}\right|dzdx\\
\leq &C\int_{-\infty}^{+\infty}\left\{\left|\bar{u}\left[\bar{\rho}^\delta\left(\frac{\bar{R}}{\bar{\rho}}\right)\right] \bar{\rho}^{-\delta}(\rho-\bar{\rho})\right|+\left|\left[\bar{\rho}^\delta\left(\frac{\bar{R}}{\bar{\rho}}\right)\right] \bar{\rho}^{-\delta}(m-\bar{m})\right|\right\}(|\bar{u}|^{\frac{2}{\gamma-1}}+|\bar{\rho}|)dx\\
=&:J^{1}_{1}+J^{2}_{1}+J^{3}_{1}+J^{4}_{1}.
\end{aligned}
\end{equation}
Applying $\rm{H\ddot{o}lder}$'s inequality with exponents 2, 2, $\infty$, and using \eqref{rhobeta} and \eqref{rhodelta}, we obtain
\begin{equation}\label{J11}
\begin{aligned}
J^{1}_{1}&=C\int_{-\infty}^{+\infty}\left\{\left|\bar{u}\left[\bar{\rho}^\delta\left(\frac{\bar{R}}{\bar{\rho}}\right)\right] \bar{\rho}^{-\delta}(\rho-\bar{\rho})\right|\right\}|\bar{u}|^{\frac{2}{\gamma-1}}dx\\
&\leq C\left\|\bar{\rho}^\delta \frac{\bar{R}}{\bar{\rho}}\right\|\left\|\bar{\rho}^{-\delta} \bar{u}^{\frac{\gamma+1}{\gamma-1}}\right\|\|\rho-\bar{\rho}\|_{L^{\infty}}\\
&\leq C (1+t)^{-\frac{2(\delta+2 \gamma+1)-1}{2(\gamma+1)}}\cdot(1+t)^{-\frac{2\left(\frac{\gamma+1}{\gamma-1}\gamma-\delta\right)-1}{2(r+1)}}\\
&\leq C(1+t)^{-\frac{(3\gamma-1)\gamma}{\gamma^2-1}},
\end{aligned}
\end{equation}
and
\begin{equation}\label{J12}
\begin{aligned}
J^{2}_{1}&=C\int_{-\infty}^{+\infty}\left\{\left|\bar{u}\left[\bar{\rho}^\delta\left(\frac{\bar{R}}{\bar{\rho}}\right)\right] \bar{\rho}^{-\delta}(\rho-\bar{\rho})\right|\right\}|\bar{\rho}|dx\\
&\leq C\left\|\bar{\rho}^\delta \frac{\bar{R}}{\bar{\rho}}\right\|\left\|\bar{\rho}^{1-\delta} \bar{u}\right\|\|\rho-\bar{\rho}\|_{L^{\infty}}\\
&\leq C (1+t)^{-\frac{2(\delta+2 \gamma+1)-1}{2(\gamma+1)}}\cdot(1+t)^{-\frac{2(\gamma+1-\delta)-1}{2(r+1)}}\\
&\leq C (1+t)^{-\frac{3\gamma+1}{\gamma+1}}.
\end{aligned}
\end{equation}
Again, applying $\rm{H\ddot{o}lder}$'s inequality with exponents 4, 4, 2, and Lemma \ref{lem4}, we infer
\begin{equation}\label{J13}
\begin{aligned}
J^{3}_{1}&=C\int_{-\infty}^{+\infty}\left[\bar{\rho}^\delta\left(\frac{\bar{R}}{\bar{\rho}}\right)\right] \bar{\rho}^{-\delta}|m-\bar{m}||\bar{u}|^{\frac{2}{\gamma-1}}dx\\
&\leq C\left\|\bar{\rho}^\delta \frac{\bar{R}}{\bar{\rho}}\right\|_{L^4}\left\|\bar{\rho}^{-\delta}
\bar{u}^{\frac{2}{\gamma-1}}\right\|_{L^4}\left\|y_t\right\|\\
&\leq C(1+t)^{-\frac{4(\delta+2 \gamma+1)-1}{4(\gamma+1)}}\cdot(1+t)^{-\frac{4\left(\frac{2\gamma}{\gamma-1}-\delta\right)-1}{4(r+1)}}||y_t||\\
&\leq C(1+t)^{-\frac{4\gamma^2+\gamma-1}{2(\gamma^2-1)}}||y_t||\\
&\leq C(1+t)^{-\frac{3\gamma^2+\gamma}{\gamma^2-1}}+ C(1+t)^{-1}\|y_t\|^2,
\end{aligned}
\end{equation}
where the last inequality is from Young's inequality, and
\begin{equation}\label{J14}
\begin{aligned}
J^{4}_{1}&=C\int_{-\infty}^{+\infty}\left[\bar{\rho}^\delta\left(\frac{\bar{R}}{\bar{\rho}}\right)\right] \bar{\rho}^{-\delta}|m-\bar{m}|\bar{\rho}dx\\
&\leq C\left\|\bar{\rho}^\delta \frac{\bar{R}}{\bar{\rho}}\right\|_{L^4}\left\|\bar{\rho}^{1-\delta}\right\|_{L^4}\left\|y_t\right\|\\
&\leq C(1+t)^{-\frac{4(\delta+2 \gamma+1)-1}{4(\gamma+1)}}\cdot(1+t)^{-\frac{4(1-\delta)-1}{4(r+1)}}\left\|y_t\right\|\\
&\leq C(1+t)^{-\frac{4\gamma+3}{2(\gamma+1)}}||y_t||\\
&\leq C(1+t)^{-\frac{3\gamma+2}{\gamma+1}}+ C(1+t)^{-1}\|y_t\|^2.
\end{aligned}
\end{equation}
With $\gamma\in(1,2]$, one has
\begin{align*}
\frac{(3\gamma-1)\gamma}{\gamma^2-1}-\frac{3\gamma+1}{\gamma+1}=\frac{(3\gamma-1)\gamma-(3\gamma+1)(\gamma-1)}{\gamma^2-1}=\frac{1}{\gamma-1}>0.
\end{align*}
Therefore, it follows from \eqref{J10}-\eqref{J14} that
\begin{equation}\label{J1}
J_1\leq  C (1+t)^{-\frac{3\gamma+1}{\gamma+1}}+C(1+t)^{-1}\|y_t\|^2.
\end{equation}

As for $J_2$, from
$$
g''(\xi)=\frac{\gamma+1}{\gamma-1}\frac{g'(\xi)}{\xi},
$$
we have
\begin{align*}
\int^{1}_{-1}\bar{g}''z\bar{\rho}^{\theta}\left(1-z^2\right)^{\lambda}dz&=\int^{1}_{-1}\bar{g}''(\bar{u}+z\bar{\rho}^{\theta}-\bar{u})
\left(1-z^2\right)^{\lambda}dz\\
&=\frac{\gamma+1}{\gamma-1}\int^{1}_{-1}\bar{g}'\left(1-z^2\right)^{\lambda}dz-\bar{u}\int^{1}_{-1}\bar{g}''\left(1-z^2\right)^{\lambda}dz.
\end{align*}
Consequently,
\begin{align*}
\bigg|\int^{1}_{-1}\bar{g}''z\bar{\rho}^{\theta}\left(1-z^2\right)^{\lambda}dz\bigg|&\leq C\int^{1}_{-1}\bigg|\bar{g}'\left(1-z^2\right)^{\lambda}\bigg|dz+C\bigg|\bar{u}\int^{1}_{-1}\bar{g}''\left(1-z^2\right)^{\lambda}\bigg|dz\\
&\leq C\left(|\bar{u}|^{\frac{\gamma+1}{\gamma-1}}+\bar{\rho}^{\frac{\gamma+1}{2}}\right)
+C|\bar{u}|\left(|\bar{u}|^{\frac{2}{\gamma-1}}+\bar{\rho}\right)\\
&\leq C\left(|\bar{u}|^{\frac{\gamma+1}{\gamma-1}}+\bar{\rho}^{\frac{\gamma+1}{2}}\right),
\end{align*}
and
\begin{equation}\label{J20}
\begin{aligned}
J_2\leq&\theta\bigg|\int_{-\infty}^{+\infty}\left[\bar{\rho}^\delta\left(\frac{\bar{R}}{\bar{\rho}}\right)\right] \bar{\rho}^{-\delta}(\rho-\bar{\rho})\int^{1}_{-1}\bar{g}''z\bar{\rho}^{\theta}\left(1-z^2\right)^{\lambda}dzdx\bigg|\\
\leq& \int_{-\infty}^{+\infty}\left|\left[\bar{\rho}^\delta\left(\frac{\bar{R}}{\bar{\rho}}\right)\right] \bar{\rho}^{-\delta}(\rho-\bar{\rho})\right|\left|\int^{1}_{-1}\bar{g}''z\bar{\rho}^{\theta}\left(1-z^2\right)^{\lambda}dz\right|dx\\
\leq& \int_{-\infty}^{+\infty}\left|\left[\bar{\rho}^\delta\left(\frac{\bar{R}}{\bar{\rho}}\right)\right] \bar{\rho}^{-\delta}(\rho-\bar{\rho})\right|\left(|\bar{u}|^{\frac{\gamma+1}{\gamma-1}}+\bar{\rho}^{\frac{\gamma+1}{2}}\right)dx\\
\leq& C(J^{1}_{2}+J_{2}^{1}).
\end{aligned}
\end{equation}
Then, we apply Lemma \ref{lem4} and $\rm{H\ddot{o}lder}$'s inequality (with exponents $\frac{2(\gamma+1)}{\gamma}, \frac{2(\gamma+1)}{\gamma}$ and $\gamma+1$) to show
\begin{equation*}
\begin{aligned}
J^{1}_{2}=&\int_{-\infty}^{+\infty}\left||\bar{u}|^{\frac{\gamma+1}{\gamma-1}}\bar{\rho}^{-\delta}\left[\bar{\rho}^\delta\left(\frac{\bar{R}}
{\bar{\rho}}\right)\right] (\rho-\bar{\rho})\right|dx\\
\leq&C||\bar{u}^{\frac{\gamma+1}{\gamma-1}}\bar{\rho}^{-\delta}||_{L^{\frac{2(\gamma+1)}{\gamma}}}\left\|\bar{\rho}^\delta \frac{\bar{R}}{\bar{\rho}}\right\|_{L^{\frac{2(\gamma+1)}{\gamma}}}||\rho-\bar{\rho}||_{L^{\gamma+1}}\\
\leq& C(1+t)^{-\frac{2(\gamma+1)\left(\frac{\gamma+1}{\gamma-1}\gamma-\delta\right)-\gamma}{2(\gamma+1)^2}}
\cdot(1+t)^{-\frac{2(\gamma+1)(\delta+2\gamma+1)-\gamma}{2(\gamma+1)^2}}||\rho-\bar{\rho}||_{L^{\gamma+1}}\\
\leq &C(1+t)^{-\frac{3\gamma^3+2\gamma^2-1}{(\gamma+1)^2(\gamma-1)}}||\rho-\bar{\rho}||_{L^{\gamma+1}},
\end{aligned}
\end{equation*}
and
\begin{equation*}
\begin{aligned}
J^{2}_{2}=&\int_{-\infty}^{+\infty}\left|\bar{\rho}^{\theta+1-\delta}\left[\bar{\rho}^\delta\left(\frac{\bar{R}}
{\bar{\rho}}\right)\right] (\rho-\bar{\rho})\right|dx\\
\leq&C||\bar{\rho}^{\theta+1-\delta}||_{L^{\frac{2(\gamma+1)}{\gamma}}}\left\|\bar{\rho}^\delta \frac{\bar{R}}{\bar{\rho}}\right\|_{L^{\frac{2(\gamma+1)}{\gamma}}}||\rho-\bar{\rho}||_{L^{\gamma+1}}\\
\leq &C(1+t)^{-\frac{2(\gamma+1)(\theta+1-\delta)-\gamma}{2(\gamma+1)^2}}
\cdot(1+t)^{-\frac{2(\gamma+1)(\delta+2\gamma+1)-\gamma}{2(\gamma+1)^2}}||\rho-\bar{\rho}||_{L^{\gamma+1}}\\
\leq &C(1+t)^{-\frac{5\gamma^2+6\gamma+3}{2(\gamma+1)^2}}||\rho-\bar{\rho}||_{L^{\gamma+1}}.
\end{aligned}
\end{equation*}
This and \eqref{J20} enable us to conclude
\begin{equation}\label{J2}
J_2\leq C(1+t)^{-\frac{5\gamma^2+6\gamma+3}{2(\gamma+1)^2}}||\rho-\bar{\rho}||_{L^{\gamma+1}},
\end{equation}
due to
\begin{align*}
\frac{3\gamma^3+2\gamma^2-1}{(\gamma+1)^2(\gamma-1)}-\frac{5\gamma^2+6\gamma+3}{2(\gamma+1)^2}=&
\frac{2(3\gamma^3+2\gamma^2-1)-(5\gamma^2+6\gamma+3)(\gamma-1)}{2(\gamma+1)^2(\gamma-1)}\\
=&\frac{\gamma^3+3\gamma^2+3\gamma+1}{2(\gamma+1)^2(\gamma-1)}>0.
\end{align*}

As for $J_3$, since $\bar{g}''=|\bar{u}+z\bar{\rho}^{\theta}|^{\frac{2}{\gamma-1}}\ge0$, we first calculate $-(\eta_{\bar{m}})_x$ as follows
\begin{equation}\label{etamx}
\begin{aligned}
-(\eta_{\bar{m}})_x&=- \int_{-1}^1 \bar{g}''\left[\left(\frac{\bar{m}}{\bar{\rho}}\right)_x+
z\theta\bar{\rho}^{\theta-1}\bar{\rho}_x\right]\left(1-z^2\right)^\lambda \mathrm{d}z\\
&=-\int_{-1}^1 \bar{g}''\left[\frac{1}{(1+t)(\gamma+1)}+z\theta\bar{\rho}^{\theta-1}\bar{\rho}_x\right]\left(1-z^2\right)^\lambda \mathrm{d}z\\
&\leq -\theta\bar{\rho}^{\theta-1}\bar{\rho}_x\int_{-1}^1 \bar{g}''z\left(1-z^2\right)^\lambda \mathrm{d}z\\
&=\frac{\theta}{2(\lambda+1)}\bar{\rho}^{\theta-1}\bar{\rho}_x\bar{g}''\left(1-z^2\right)^{\lambda+1}\big|^{1}_{-1} -\frac{\theta}{2(\lambda+1)}\bar{\rho}^{2\theta-1}\bar{\rho}_x\int_{-1}^1 \bar{g}^{(3)}\left(1-z^2\right)^{\lambda+1}dz\\
&=\frac{\theta}{2\gamma(\lambda+1)}\frac{\bar{m}}{\bar{\rho}}\int_{-1}^1 \bar{g}^{(3)}\left(1-z^2\right)^{\lambda+1}dz,
\end{aligned}
\end{equation}
by $\bar{m}=-(\bar{\rho}^{\gamma})_x$. Then, using
$
\bar{g}^{(3)}(\xi)=\frac{2}{\gamma-1}\frac{\bar{g}''(\xi)}{\xi},
$
we derive that
\begin{align*}
&\frac{\bar{m}}{\bar{\rho}}\int_{-1}^1 \bar{g}^{(3)}\left(1-z^2\right)^{\lambda+1}dz\\
=&\int_{-1}^1 \bar{g}^{(3)}\left(\frac{\bar{m}}{\bar{\rho}}+
z\bar{\rho}^{\theta}\right)\left(1-z^2\right)^{\lambda+1}dz-\int_{-1}^1 \bar{g}^{(3)}z\bar{\rho}^{\theta}\left(1-z^2\right)^{\lambda+1}dz\\
=&\frac{2}{\gamma-1}\int_{-1}^1 \bar{g}''\left(1-z^2\right)^{\lambda+1}dz-\bar{g}''z\left(1-z^2\right)^{\lambda+1}\big|^{1}_{-1}\\
&+\int_{-1}^1\bar{g}''\left(1-z^2\right)^{\lambda+1}dz-\int_{-1}^1\bar{g}''2(\lambda+1)\left(1-z^2\right)^{\lambda}z^2dz\\
=&\frac{\gamma+1}{\gamma-1}\int_{-1}^1 \bar{g}''\left[\left(1-z^2\right)^{\lambda+1}-\left(1-z^2\right)^{\lambda}z^2\right]dz.
\end{align*}
Hence, it follows from \eqref{etamx} that
\begin{align*}
-(\eta_{\bar{m}})_x\leq&\frac{\theta}{2\gamma(\lambda+1)}\frac{\bar{m}}{\bar{\rho}}\int_{-1}^1 \bar{g}^{(3)}\left(1-z^2\right)^{\lambda+1}dz\\
=&\frac{\gamma-1}{2\gamma}\int_{-1}^1 \bar{g}''\left[\left(1-z^2\right)^{\lambda+1}-\left(1-z^2\right)^{\lambda}z^2\right]dz,
\end{align*}
which implies
\begin{equation}\label{J30}
J_3=-\int_{-\infty}^{+\infty}(\eta_{\bar{m}})_x(P_*+Q_*)dx\leq C\int_{-\infty}^{+\infty}(P_*+Q_*)\int_{-1}^1 \bar{g}''[\left(1-z^2\right)^{\lambda+1}-\left(1-z^2\right)^{\lambda}z^2]dzdx,
\end{equation}
by $P_*\ge0$ and $Q_*\ge0$.

By Taylor's Theorem \ref{lem64}, it holds that
$$
\bar{g}''(\bar{u}+z\bar{\rho}^{\theta})=\bar{g}''(z\bar{\rho}^{\theta})+\bar{g}^{(3)}(z\bar{\rho}^{\theta})\bar{u}+\int^{1}_{0}(1-s)\bar{g}^{(4)}
(s\bar{u}+z\bar{\rho}^{\theta})\bar{u}^2ds.
$$
Based on the fact that
\begin{align*}
&\int_{-1}^1 \bar{g}^{(3)}(z\bar{\rho}^{\theta})u\left[\left(1-z^2\right)^{\lambda+1}-\left(1-z^2\right)^{\lambda}z^2\right]dz\\
=&\frac{2}{\gamma-1}\bar{\rho}^{\theta}\bar{u}\int_{-1}^1 |z\bar{\rho}|^{\frac{4-2\gamma}{\gamma-1}}z\left[\left(1-z^2\right)^{\lambda+1}-\left(1-z^2\right)^{\lambda}z^2\right]dz=0,
\end{align*}
and
\begin{align*}
&\int_{-1}^1 \bar{g}''(z\bar{\rho}^{\theta})\left[\left(1-z^2\right)^{\lambda+1}-\left(1-z^2\right)^{\lambda}z^2\right]dz\\
=&\bar{\rho}\int_{-1}^1 |z|^{\frac{2}{\gamma-1}}\left[\left(1-z^2\right)^{\lambda+1}-\left(1-z^2\right)^{\lambda}z^2\right]dz\\
=&\bar{\rho}\int_{-1}^1 \left[|z|^{\frac{2}{\gamma-1}}\left(1-z^2\right)^{\lambda+1}-\frac{1}{2(\lambda+1)}\frac{\gamma+1}
{\gamma-1}|z|^{\frac{2}{\gamma-1}}\left(1-z^2\right)^{\lambda+1}\right]dz=0,
\end{align*}
where the last equality is from $\frac{1}{2(\lambda+1)}\frac{\gamma+1}{\gamma-1}=1$, we obtain
\begin{align*}
&\int_{-1}^1 \bar{g}''(\bar{u}+z\bar{\rho}^{\theta})\left[\left(1-z^2\right)^{\lambda+1}-\left(1-z^2\right)^{\lambda}z^2\right]dz\\
=&\int_{-1}^1\int^{1}_{0}(1-s)\bar{g}^{(4)}
(s\bar{u}+z\bar{\rho}^{\theta})\bar{u}^2\left[\left(1-z^2\right)^{\lambda+1}-\left(1-z^2\right)^{\lambda}z^2\right]dsdz\\
\leq& C\bar{u}^2\left(|\bar{u}|^{\frac{4-2\gamma}{\gamma-1}}+|\bar{\rho}|^{\frac{\gamma-1}{2}\cdot\frac{4-2\gamma}{\gamma-1}}\right)\\
=& C\left(|\bar{u}|^{\frac{2}{\gamma-1}}+|\bar{u}|^2|\bar{\rho}|^{2-\gamma}\right).
\end{align*}
This combined with \eqref{J30} gives
\begin{equation}\label{J3}
\begin{aligned}
J_3&\leq C\int_{-\infty}^{+\infty}(P_*+Q_*)\int_{-1}^1 \bar{g}''\left[\left(1-z^2\right)^{\lambda+1}-\left(1-z^2\right)^{\lambda}z^2\right]dzdx\\
&\leq C\int_{-\infty}^{+\infty}(P_*+Q_*)\left(|\bar{u}|^{\frac{2}{\gamma-1}}+|\bar{u}|^2|\bar{\rho}|^{2-\gamma}\right)dx\\
&\leq C\left(\left\|\bar{u}^{\frac{2}{\gamma-1}}\right\|_{L^{\infty}}+
\left\|\bar{u}^2\bar{\rho}^{2-\gamma}\right\|_{L^{\infty}}\right)\int_{-\infty}^{+\infty}(P_*+Q_*)dx\\
&\leq C\left[(1+t)^{-\frac{2\gamma}{\gamma^2-1}}+(1+t)^{-\frac{\gamma+2}{\gamma+1}}\right]\int_{-\infty}^{+\infty}(P_*+Q_*)dx\\
&\leq C(1+t)^{-\frac{\gamma+2}{\gamma+1}},
\end{aligned}
\end{equation}
by \eqref{PQC}, $\gamma\in(1,2)$ and
$
\frac{2\gamma}{\gamma^2-1}\ge \frac{\gamma+2}{\gamma+1}.
$

\subsubsection{Integral estimate about $\eta$ and $A$ in terms of $J_1$, $J_2$,$J_3$, and $B$}
Plugging \eqref{J1}, \eqref{J2} and \eqref{J3} into \eqref{J123}, and integrating over $[0,t]$, we obtain
\begin{equation}\label{eta*ga+1}
\begin{aligned}
&\int_{-\infty}^{+\infty}\eta_{* }dx +\alpha\int^{t}_{0}\int_{-\infty}^{+\infty}A_*dxd\tau\\
\leq& C \int^{t}_{0}(1+\tau)^{-\frac{3\gamma+1}{\gamma+1}}d\tau+C\int^{t}_{0}(1+\tau)^{-1}\|y_t\|^2d\tau
+C\int^{t}_{0}(1+\tau)^{-\frac{2\gamma}{\gamma^2-1}}d\tau\\
+&C\int^{t}_{0}(1+\tau)^{-\frac{5\gamma^2+6\gamma+3}{2(\gamma+1)^2}\cdot\frac{\gamma+1}{\gamma}}d\tau+
C\int^{t}_{0}||\rho-\bar{\rho}||^{\gamma+1}_{L^{\gamma+1}}d\tau+C\\
\leq &C,
\end{aligned}
\end{equation}
by Lemma \ref{lem2}, Lemma \ref{lem3},
and
$$
\gamma^2-1-2\gamma,~ 2\gamma(\gamma+1)-(5\gamma^2+6\gamma+3)<0,\quad \mbox{with}~\gamma\in(1,2].
$$

Next, for any small positive constant $\varepsilon>0$, we define
\begin{equation}\label{nu}
\nu(\varepsilon)=\frac{\gamma^2+2\gamma}{(\gamma+1)^2}-2\varepsilon.
\end{equation}
Multiply \eqref{J123} with $(1+t)^{\nu(\varepsilon)}$, and then integrate over $[0, t]$; then
\begin{equation}\label{nuvar}
\begin{aligned}
&(1+t)^{\nu(\varepsilon)} \int_{-\infty}^{+\infty} \eta_*(x, t)dx+\alpha\int^{t}_{0}\int_{-\infty}^{+\infty}
(1+t)^{\nu(\varepsilon)}A_*dxd\tau\\
\leq& C+C\int^{t}_{0}\int_{-\infty}^{+\infty}(1+\tau)^{\nu(\varepsilon)-1}
(A_*+B_*)dxd\tau+C\int^{t}_{0}(1+\tau)^{\nu(\varepsilon)}(J_1+J_2+J_3)d\tau\\
\leq& C+C\int^{t}_{0}\int_{-\infty}^{+\infty}(1+\tau)^{\nu(\varepsilon)-1}B_*dxd\tau+C\int^{t}_{0}(1+\tau)^{\nu(\varepsilon)}(J_1+J_2+J_3)d\tau,
\end{aligned}
\end{equation}
where the last inequality is from \eqref{eta*ga+1} and $\eta_*\ge0$.
\subsubsection{Relationship between $A$ and $\eta$}
The relationship is stated in the following lemma.
\begin{lemma}\label{lem67}
Let $A=m\int^{1}_{-1}g'\left(1-z^2\right)^\lambda \mathrm{d} z$ and $\eta=\rho \int_{-1}^1 g\left(1-z^2\right)^\lambda \mathrm{d} z$. Let $c_0, c_1$ be two positive constants and
$$
\psi(\rho, m)=c_0A-c_1\eta.
$$
Then,
\begin{align*}
&\theta^{-2}\rho^{-2\theta}\left(\psi_{\rho\rho}\psi_{mm}-\psi^2_{\rho m}\right)\\
=&\left[\left(\frac{3\gamma-1}{\gamma-1}c_0-c_1\right)\int_{-1}^1h(a+z)\left(1-z^2\right)^\lambda d z-2\lambda c_0\int_{-1}^1  h(a+z)\left(1-z^2\right)^{\lambda-1}z^2 d z\right]\\
&\times\left[\left(\frac{\gamma+1}{\gamma-1}c_0-c_1\right)\int_{-1}^1 h(a+z)\left(1-z^2\right)^\lambda d z-2\lambda c_0\int_{-1}^1 h(a+z)\left(1-z^2\right)^{\lambda-1}z^2 d z\right]\\
&-\left[\left(\frac{3\gamma-1}{\gamma-1}c_0-c_1\right)\int_{-1}^1h(a+z)\left(1-z^2\right)^{\lambda} zdz-2\lambda c_0\int_{-1}^1h(a+z)\left(1-z^2\right)^{\lambda-1}z^3 dz\right]^2,
\end{align*}
where $\theta=\frac{\gamma-1}{2}$, $a=\frac{u}{\rho^{\theta}}$ and $h(\xi)=|\xi|^{\frac{2}{\gamma-1}}$, and
\begin{equation}\label{psimm}
\begin{aligned}
\psi_{mm}=&\left(\frac{3\gamma-1}{\gamma-1}c_0-c_1\right)\int_{-1}^1h(a+z)\left(1-z^2\right)^\lambda d z-2\lambda c_0\int_{-1}^1  h(a+z)\left(1-z^2\right)^{\lambda-1}z^2 d z\\
=&\left(2c_0-c_1\right)\int_{-1}^1|a+z|^\frac{2}{\gamma-1}\left(1-z^2\right)^\lambda d z+\frac{2a}{\gamma-1}c_0\int_{-1}^1\frac{|a+z|^\frac{2}{\gamma-1}}{a+z}\left(1-z^2\right)^{\lambda}dz.
\end{aligned}
\end{equation}
\end{lemma}

\begin{proof}
From the definition of $B$, we have
\begin{align*}
B=&\rho^{\theta+1}\int_{-1}^1g'z\left(1-z^2\right)^\lambda dz\\
=&\rho g z\left(1-z^2\right)^\lambda\big|^{1}_{-1}-\rho\int_{-1}^1g\left[\left(1-z^2\right)^\lambda-2\lambda\left(1-z^2\right)^{\lambda-1} z^2\right]dz\\
=&-\rho\int_{-1}^1g\left(1-z^2\right)^\lambda dz+2\lambda\rho\int_{-1}^1g\left(1-z^2\right)^{\lambda-1} z^2dz\\
=&-\eta+2\lambda\rho\int_{-1}^1g\left(1-z^2\right)^{\lambda-1} z^2dz.
\end{align*}
Since
$$
\eta(\rho, m)=\rho \int_{-1}^1 g\left(1-z^2\right)^\lambda \mathrm{d} z=\frac{\gamma-1}{2\gamma}(A+B),
$$
it turns out that
\begin{align*}
A=\frac{2\gamma}{\gamma-1}\eta-B=\frac{3\gamma-1}{\gamma-1}\eta-2\lambda\rho\int_{-1}^1g\left(1-z^2\right)^{\lambda-1} z^2dz,
\end{align*}
which indicates that
$$
\psi=c_0A-c_1\eta=\left(\frac{3\gamma-1}{\gamma-1}c_0-c_1\right)\eta-2\lambda c_0\rho\int_{-1}^1g\left(1-z^2\right)^{\lambda-1} z^2dz.
$$

From \eqref{Hessiaeta}, we can compute the Hessian matrix of $\psi$ similarly. Indeed, $\psi_{\rho\rho}$ is equal to
\begin{align*}
&\frac{3\gamma-1}{\gamma-1}c_0\bigg[\int_{-1}^1(\theta^2+\theta) \rho^{\theta-1} g' z\left(1-z^2\right)^\lambda dz  +\int_{-1}^1\left(-u+z\theta \rho^{\theta}\right)^2 g''\left(1-z^2\right)^\lambda d z \bigg]\\
&-2\lambda c_0\left[\int_{-1}^1(\theta^2+\theta) \rho^{\theta-1} g' z^3\left(1-z^2\right)^{\lambda-1} dz  +\int_{-1}^1\left(-u+z\theta \rho^{\theta}\right)^2 g''\left(1-z^2\right)^{\lambda-1}z^2 d z \right]\\
&-c_1\bigg[\int_{-1}^1(\theta^2+\theta) \rho^{\theta-1} g' z\left(1-z^2\right)^\lambda dz  +\int_{-1}^1\left(-u+z\theta \rho^{\theta}\right)^2 g''\left(1-z^2\right)^\lambda d z \bigg]\\
=&(\theta^2+\theta)\rho^{\theta-1}\left[\left(\frac{3\gamma-1}{\gamma-1}c_0-c_1\right)\int_{-1}^1 g' z\left(1-z^2\right)^\lambda dz-2\lambda c_0\int_{-1}^1 g' z^3\left(1-z^2\right)^{\lambda-1} dz\right]\\
&+\frac{1}{\rho}\left(\frac{3\gamma-1}{\gamma-1}c_0-c_1\right)\int_{-1}^1\left(-u+z\theta \rho^{\theta}\right)^2 g''\left(1-z^2\right)^\lambda d z\\
&-\frac{1}{\rho}2\lambda c_0\int_{-1}^1\left(-u+z\theta \rho^{\theta}\right)^2 g''\left(1-z^2\right)^{\lambda-1}z^2 d z \\
=&:(\theta^2+\theta)\rho^{\theta-1}\left(\psi_{\rho\rho1}+\psi_{\rho\rho2}\right)+\frac{1}{\rho}\left(\psi_{\rho\rho3}+\psi_{\rho\rho4}\right),
\end{align*}
\begin{align*}
\psi_{mm}=&\frac{1}{\rho}\left[ \left(\frac{3\gamma-1}{\gamma-1}c_0-c_1\right)\int_{-1}^1 g''\left(1-z^2\right)^\lambda d z-2\lambda c_0\int_{-1}^1 g''\left(1-z^2\right)^{\lambda-1}z^2 d z\right]\\
=&:\frac{1}{\rho}\left( \psi_{mm1}+\psi_{mm2}\right),
\end{align*}
and
\begin{align*}
\psi_{\rho m}=&\left(\frac{3\gamma-1}{\gamma-1}c_0-c_1\right)\int_{-1}^1\left(-\frac{m}{\rho^2}+z\theta\rho^{\theta-1}\right) g'' \left(1-z^2\right)^\lambda dz\\
&-2\lambda c_0\int_{-1}^1\left(-\frac{m}{\rho^2}+z\theta\rho^{\theta-1}\right) g'' \left(1-z^2\right)^{\lambda-1}z^2 dz\\
=&\frac{1}{\rho}\left(\frac{3\gamma-1}{\gamma-1}c_0-c_1\right)\int_{-1}^1\left(-u+z\theta\rho^{\theta}\right) g'' \left(1-z^2\right)^\lambda dz\\
&-\frac{1}{\rho}2\lambda c_0\int_{-1}^1\left(-u+z\theta\rho^{\theta}\right) g'' \left(1-z^2\right)^{\lambda-1}z^2 dz\\
=&:\frac{1}{\rho}\left(\psi_{\rho m1}+\psi_{\rho m2}\right),
\end{align*}
which gives that
\begin{equation}\label{Hpsi}
\begin{aligned}
\rho^{2}\left(\psi_{\rho\rho}\psi_{mm}-\psi^{2}_{\rho m}\right)=&(\theta^2+\theta)\rho^{\theta}\left(\psi_{\rho\rho1}
+\psi_{\rho\rho2}\right)\left( \psi_{mm1}+\psi_{mm2}\right)\\
&+\left(\psi_{\rho\rho3}+\psi_{\rho\rho4}\right)\left( \psi_{mm1}+\psi_{mm2}\right)
-\left(\psi_{\rho m1}+\psi_{\rho m2}\right)^2.
\end{aligned}
\end{equation}

First, we estimate $\psi_{\rho\rho1}+\psi_{\rho\rho2}$. Integrating $\psi_{\rho\rho2}$ by parts of, we see
\begin{align*}
-2\lambda\int_{-1}^1 g' z^3\left(1-z^2\right)^{\lambda-1} dz=-\rho^{\theta}\int _{-1}^1 g''\left(1-z^2\right)^{\lambda}z^2dz-2\int_{-1}^1 g'z\left(1-z^2\right)^{\lambda}dz,
\end{align*}
which implies that $\psi_{\rho\rho1}+\psi_{\rho\rho2}$ becomes
\begin{align*}
\psi_{\rho\rho1}+\psi_{\rho\rho2}=&\left[\frac{\gamma+1}{\gamma-1}c_0-c_1\right]\int_{-1}^1 g' z\left(1-z^2\right)^\lambda dz-c_0\rho^{\theta}\int_{-1}^1 g''\left(1-z^2\right)^{\lambda}z^2dz\\
=& \left[c_0-c_1\frac{\gamma-1}{\gamma+1}\right]\rho^{\theta}\int_{-1}^1 g''\left(1-z^2\right)^{\lambda+1}dz-c_0\rho^{\theta}\int_{-1}^1 g''\left(1-z^2\right)^{\lambda}z^2dz\\
=&\rho^{\theta}\left[\left(c_0-c_1\frac{\gamma-1}{\gamma+1}\right)\int_{-1}^1 g''\left(1-z^2\right)^{\lambda+1}dz-c_0\int_{-1}^1 g''\left(1-z^2\right)^{\lambda}z^2dz\right].
\end{align*}

Next, we simplify
$$
\left(\psi_{\rho\rho3}+\psi_{\rho\rho4}\right)\cdot\left( \psi_{mm1}+\psi_{mm2}\right)-\left(\psi_{\rho m1}+\psi_{\rho m2}\right)^2
$$
by merging terms of the same kind. The coefficient for $\left(\frac{3\gamma-1}{\gamma-1}c_0-c_1\right)^2$ is
\begin{align*}
&\int_{-1}^1\left(-u+z\theta \rho^{\theta}\right)^2 g''\left(1-z^2\right)^\lambda d z\cdot\int_{-1}^1 g''\left(1-z^2\right)^\lambda d z\\
&-\left[\int_{-1}^1 g'' \left(-u+z\theta\rho^{\theta}\right)\left(1-z^2\right)^\lambda  dz\right]^2\\
=&\theta^2\rho^{2\theta}\left\{\int_{-1}^1z^2g''\left(1-z^2\right)^\lambda d z\cdot\int_{-1}^1 g''\left(1-z^2\right)^\lambda d z-\left[\int_{-1}^1 g'' \left(1-z^2\right)^\lambda z dz\right]^2\right\}.
\end{align*}
Similarly, the coefficient for $(2\lambda c_0)^2$ is
\begin{align*}
&\int_{-1}^1\left(-u+z\theta \rho^{\theta}\right)^2 g''\left(1-z^2\right)^{\lambda-1}z^2 d z\cdot\int_{-1}^1 g''\left(1-z^2\right)^{\lambda-1}z^2 d z\\
&-\left[\int_{-1}^1\left(-u+z\theta\rho^{\theta}\right) g'' \left(1-z^2\right)^{\lambda-1}z^2 dz\right]^2\\
=&\theta^2\rho^{2\theta}\int_{-1}^1g''\left(1-z^2\right)^{\lambda-1}z^4 d z\cdot\int_{-1}^1 g''\left(1-z^2\right)^{\lambda-1}z^2 d z\\
&-\theta^2\rho^{2\theta}\left[\int_{-1}^1 g'' \left(1-z^2\right)^{\lambda-1}z^3 dz\right]^2,
\end{align*}
and the coefficient for $\left(\frac{3\gamma-1}{\gamma-1}c_0-c_1\right)2\lambda c_0$ is
\begin{align*}
&-\int_{-1}^1\left(-u+z\theta \rho^{\theta}\right)^2 g''\left(1-z^2\right)^\lambda d z\cdot\int_{-1}^1 g''\left(1-z^2\right)^{\lambda-1}z^2 d z\\
&-\int_{-1}^1\left(-u+z\theta \rho^{\theta}\right)^2 g''\left(1-z^2\right)^{\lambda-1}z^2 d z\cdot\int_{-1}^1 g''\left(1-z^2\right)^\lambda d z\\
&+2\int_{-1}^1\left(-u+z\theta\rho^{\theta}\right) g'' \left(1-z^2\right)^\lambda dz\cdot\int_{-1}^1\left(-u+z\theta\rho^{\theta}\right) g'' \left(1-z^2\right)^{\lambda-1}z^2 dz\\
=&-\theta^2 \rho^{2\theta}\int_{-1}^1 g''\left(1-z^2\right)^\lambda z^2d z\cdot\int_{-1}^1 g''\left(1-z^2\right)^{\lambda-1}z^2 d z\\
&-\theta^2 \rho^{2\theta}\int_{-1}^1 g''\left(1-z^2\right)^{\lambda-1}z^4 d z\cdot\int_{-1}^1 g''\left(1-z^2\right)^\lambda d z\\
&+2\theta^2 \rho^{2\theta}\int_{-1}^1\ g'' \left(1-z^2\right)^\lambda zdz\cdot\int_{-1}^1 g'' \left(1-z^2\right)^{\lambda-1}z^3 dz.
\end{align*}
Hence,
\begin{align*}
&\theta^{-2}\rho^{-2\theta}\left[\left(\psi_{\rho\rho3}+\psi_{\rho\rho4}\right)\cdot\left( \psi_{mm1}+\psi_{mm2}\right)-\left(\psi_{\rho m1}+\psi_{\rho m2}\right)^2\right]\\
=&\left[\left(\frac{3\gamma-1}{\gamma-1}c_0-c_1\right)\int_{-1}^1 g'' \left(1-z^2\right)^\lambda z^2 dz-2\lambda c_0\int_{-1}^1 g'' \left(1-z^2\right)^{\lambda-1}z^4 dz\right]\rho\psi_{mm}\\
-&\left[\left(\frac{3\gamma-1}{\gamma-1}c_0-c_1\right)\int_{-1}^1 g'' \left(1-z^2\right)^\lambda z dz-2\lambda c_0\int_{-1}^1 g'' \left(1-z^2\right)^{\lambda-1}z^3 dz\right]^2.
\end{align*}
Using the above two equalities and \eqref{Hpsi}, we deduce
\begin{align*}
&\theta^{-2}\rho^{2-2\theta}\left(\psi_{\rho\rho}\psi_{mm}-\psi^{2}_{\rho m}\right)\\
=&
\left[\left(\frac{\gamma+1}{\gamma-1}c_0-c_1\right)\int_{-1}^1 g''\left(1-z^2\right)^{\lambda+1}dz-\frac{\gamma+1}{\gamma-1}c_0\int_{-1}^1 g''\left(1-z^2\right)^{\lambda}z^2dz\right]\rho\psi_{mm}\\
&+\left[\left(\frac{3\gamma-1}{\gamma-1}c_0-c_1\right)\int_{-1}^1 g'' \left(1-z^2\right)^\lambda z^2 dz-2\lambda c_0\int_{-1}^1 g'' \left(1-z^2\right)^{\lambda-1}z^4 dz\right]\rho\psi_{mm}\\
&-\left[\left(\frac{3\gamma-1}{\gamma-1}c_0-c_1\right)\int_{-1}^1 g'' \left(1-z^2\right)^\lambda z dz-2\lambda c_0\int_{-1}^1 g'' \left(1-z^2\right)^{\lambda-1}z^3 dz\right]^2\\
=&\left[\left(\frac{\gamma+1}{\gamma-1}c_0-c_1\right)\int_{-1}^1 g'' \left(1-z^2\right)^\lambda  dz-2\lambda c_0\int_{-1}^1 g'' \left(1-z^2\right)^{\lambda-1}z^2 dz\right]\\
&\times\left[ \left(\frac{3\gamma-1}{\gamma-1}c_0-c_1\right)\int_{-1}^1 g''\left(1-z^2\right)^\lambda d z-2\lambda c_0\int_{-1}^1 g''\left(1-z^2\right)^{\lambda-1}z^2 d z\right]\\
&-\left[\left(\frac{3\gamma-1}{\gamma-1}c_0-c_1\right)\int_{-1}^1 g'' \left(1-z^2\right)^\lambda z dz-2\lambda c_0\int_{-1}^1 g'' \left(1-z^2\right)^{\lambda-1}z^3 dz\right]^2.
\end{align*}
Also, notice
$$
g''=|u+z\rho^{\theta}|^{\frac{2}{\gamma-1}}=\rho|a+z|^{\frac{2}{\gamma-1}}=\rho h(a+z),
$$
where $a=\frac{u}{\rho^{\theta}}$ and $h(\xi)=|\xi|^{\frac{2}{\gamma-1}}$. Therefore, we obtain the first conclusion.

In the sequel, we will write, for simplicity,
$$
h=h(a+z)=|a+z|^{\frac{2}{\gamma-1}},\,h'=h'(a+z).
$$

About $\psi_{mm}$, we have
\begin{align*}
\psi_{mm}=&\frac{1}{\rho}\left[ \left(\frac{3\gamma-1}{\gamma-1}c_0-c_1\right)\int_{-1}^1 g''\left(1-z^2\right)^\lambda d z-2\lambda c_0\int_{-1}^1 g''\left(1-z^2\right)^{\lambda-1}z^2 d z\right]\\
=&\left(\frac{3\gamma-1}{\gamma-1}c_0-c_1\right)\int_{-1}^1 h\left(1-z^2\right)^\lambda d z-2\lambda c_0\int_{-1}^1 h\left(1-z^2\right)^{\lambda-1}z^2 d z,
\end{align*}
by using $g''=\rho h$ again. Then, integrating by parts and noting
$(a+z)h'(a+z)=\frac{2}{\gamma-1}h(a+z),$ yields
\begin{align*}
2\lambda\int_{-1}^1  h\left(1-z^2\right)^{\lambda-1}z^2 d z&=\int_{-1}^1  h\left(1-z^2\right)^{\lambda}dz+\int_{-1}^1  h'\left(1-z^2\right)^{\lambda}zdz\\
&=\frac{\gamma+1}{\gamma-1}\int_{-1}^1  h\left(1-z^2\right)^{\lambda}dz-a\int_{-1}^1h'\left(1-z^2\right)^{\lambda}dz.
\end{align*}
Hence,
\begin{align*}
\psi_{mm}=(2c_0-c_1)\int_{-1}^1|a+z|^\frac{2}{\gamma-1}\left(1-z^2\right)^\lambda d z+\frac{2a}{\gamma-1}c_0\int_{-1}^1\frac{|a+z|^\frac{2}{\gamma-1}}{a+z}\left(1-z^2\right)^{\lambda}dz,
\end{align*}
and this ends the proof.
\end{proof}
\subsubsection{Convexity of two functions involving in A}
Further advantage of $A_*$ is summarized in the following lemma.

\begin{lemma}\label{lem7}
Let $g(\xi)=\frac{(\gamma-1)^2}{2\gamma(\gamma+1)}|\xi|^{\frac{2\gamma}{\gamma-1}}$ with $\gamma\in(1,2]$ and $A=m\int^{1}_{-1} g'\left(1-z^2\right)^\lambda \mathrm{d} z$. Then, $A$ is convex in $(\rho, m)$ for all $\rho, m$, and $A_*\ge0$.
\end{lemma}

\begin{proof}

From Lemma \ref{lem67}, it holds with $h=|a+z|^{\frac{2}{\gamma-1}}$, $c_0=1$ and $c_1=0$ that
\begin{equation}\label{hzeta}
\begin{aligned}
&\rho^{-2\theta}\theta^{-2}\left[A_{\rho\rho}A_{mm}-A_{\rho m}^2\right]\\
=&\left[\frac{3\gamma-1}{\gamma-1}\int_{-1}^1h\left(1-z^2\right)^\lambda d z-2\lambda\int_{-1}^1  h\left(1-z^2\right)^{\lambda-1}z^2 d z\right]\\
&\times\left[\frac{\gamma+1}{\gamma-1}\int_{-1}^1 h\left(1-z^2\right)^\lambda d z-2\lambda\int_{-1}^1 h\left(1-z^2\right)^{\lambda-1}z^2 d z\right]\\
&-\left[\frac{3\gamma-1}{\gamma-1}\int_{-1}^1h\left(1-z^2\right)^{\lambda} zdz-2\lambda\int_{-1}^1h\left(1-z^2\right)^{\lambda-1}z^3 dz\right]^2\\
=&:h_1\times h_2-h^2_3.
\end{aligned}
\end{equation}
First, we show that $A_{mm}\ge0$ for all $\rho, m$. From \eqref{psimm}, it follows that
\begin{align*}
A_{mm}=\psi_{mm}=2\int_{-1}^1|a+z|^\frac{2}{\gamma-1}\left(1-z^2\right)^\lambda d z+\frac{2a}{\gamma-1}\int_{-1}^1\frac{|a+z|^\frac{2}{\gamma-1}}{a+z}\left(1-z^2\right)^{\lambda}dz>0
\end{align*}
by \eqref{ka}. Also, each of $h_1, h_2$ and $h_3$ is even in the variable $a$ by Lemma \ref{lem5}.
Hence, we need only to illustrate that the right side of \eqref{hzeta} is greater than or equal to 0 with $a\ge0$, for which we distinguish three cases.

\vspace{4pt}
\noindent {\bf Case I}: $a=0$.
Since
$$
h_3=\frac{3\gamma-1}{\gamma-1}\int_{-1}^1|z|^{\frac{2}{\gamma-1}}\left(1-z^2\right)^{\lambda} zdz-2\lambda\int_{-1}^1|z|^{\frac{2}{\gamma-1}}\left(1-z^2\right)^{\lambda-1}z^3 dz=0,
$$
and
\begin{align*}
h_2=\frac{\gamma+1}{\gamma-1}\int_{-1}^1 |z|^{\frac{2}{\gamma-1}}\left(1-z^2\right)^\lambda dz-2\lambda\int_{-1}^1 |z|^{\frac{2}{\gamma-1}}\left(1-z^2\right)^{\lambda-1}z^2dz=0,
\end{align*}
one sees that $A_{\rho\rho}A_{mm}-A_{\rho m}^2=0$.

\vspace{4pt}
\noindent {\bf Case II}: $a>1$.
Let $\forall z\in[-1,1]$. We have
\begin{equation}\label{h'}
h'(a+z)=\frac{2}{\gamma-1}\frac{h}{a+z}=\frac{2}{\gamma-1}(a+z)^{\frac{2}{\gamma-1}-1}\ge0.
\end{equation}
Integrating by parts gives
\begin{align*}
&-2\lambda\int_{-1}^1  h\left(1-z^2\right)^{\lambda-1}z^2 d z\\
=&h\left(1-z^2\right)^{\lambda}z\bigg|^{1}_{-1}-\int_{-1}^1h\left(1-z^2\right)^{\lambda}dz-\int_{-1}^1
h'\left(1-z^2\right)^{\lambda}zdz\\
=&-\frac{\gamma+1}{\gamma-1}\int_{-1}^1h\left(1-z^2\right)^\lambda d z+a\int_{-1}^1h'\left(1-z^2\right)^\lambda d z,
\end{align*}
by \eqref{h'}. Hence,
\begin{equation}\label{h1}
h_1=2\int_{-1}^1h\left(1-z^2\right)^\lambda d z+a\int_{-1}^1h'\left(1-z^2\right)^\lambda d z,
\end{equation}
and
\begin{equation}\label{h2}
h_2=a\int_{-1}^1h'\left(1-z^2\right)^\lambda d z.
\end{equation}
As for $h_3$, by $h'(\xi)\xi=\frac{2}{\gamma-1}h(\xi)$ we deduce
\begin{align*}
&-2\lambda\int_{-1}^1  h\left(1-z^2\right)^{\lambda-1}z^3 d z\\
=&h\left(1-z^2\right)^{\lambda}z^2\bigg|^{1}_{-1}-2\int_{-1}^1h\left(1-z^2\right)^{\lambda}zdz-\int_{-1}^1
h'\left(1-z^2\right)^{\lambda}z^2dz\\
=&-\frac{2\gamma}{\gamma-1}\int_{-1}^1h\left(1-z^2\right)^\lambda zdz+a\int_{-1}^1h'\left(1-z^2\right)^\lambda z d z,
\end{align*}
which infers
\begin{equation*}
h_3=\int_{-1}^1h\left(1-z^2\right)^\lambda zd z+a\int_{-1}^1h'\left(1-z^2\right)^\lambda zd z.
\end{equation*}

Also, we have, by the Cauchy-Schwarz inequality,
\begin{equation}\label{h23}
\begin{aligned}
h^2_3=&\left[\int_{-1}^1h\left(1-z^2\right)^\lambda zd z+a\int_{-1}^1h'\left(1-z^2\right)^\lambda zd z\right]^2\\
=&\left\{\int_{-1}^1\left[(a+z)^{\frac{2}{\gamma-1}}+\frac{2a}{\gamma-1}(a+z)^{\frac{2}{\gamma-1}-1}\right]\left(1-z^2\right)^\lambda zd z\right\}^2\\
\leq& \int_{-1}^1(a+z)^{\frac{2}{\gamma-1}-1}\left(1-z^2\right)^\lambda dz\\
&\times\int_{-1}^1\left[(a+z)^{\frac{1}{\gamma-1}+\frac{1}{2}}+\frac{2a}{\gamma-1}(a+z)^{\frac{1}{\gamma-1}-\frac{1}{2}}\right]^2
\left(1-z^2\right)^\lambda z^2dz\\
=&\int_{-1}^1(a+z)^{\frac{2}{\gamma-1}-1}\left(1-z^2\right)^\lambda dz\\
&\times\int_{-1}^1\left[(a+z)^{\frac{\gamma+1}{\gamma-1}}+\frac{4a}{\gamma-1}(a+z)^{\frac{2}{\gamma-1}}+\left(\frac{2a}{\gamma-1}\right)^2
(a+z)^{\frac{2}{\gamma-1}-1}\right]\left(1-z^2\right)^\lambda z^2dz\\
=&:h_{31}h_{32}.
\end{aligned}
\end{equation}
Having in mind that $\lambda=\frac{3-\gamma}{2(\gamma-1)}$ and $2\lambda+2=\frac{\gamma+1}{\gamma-1}$, we see
\begin{equation}\label{h231}
\begin{aligned}
&\int_{-1}^1(a+z)^{\frac{\gamma+1}{\gamma-1}}\left(1-z^2\right)^\lambda z^2dz\\
=&-\frac{1}{2(\lambda+1)}(a+z)^{\frac{\gamma+1}{\gamma-1}}\left(1-z^2\right)^{\lambda+1}z\bigg|^{1}_{-1}+\int_{-1}^1(a+z)^{\frac{2}{\gamma-1}}
\left(1-z^2\right)^{\lambda+1}zdz\\
&+\frac{1}{2(\lambda+1)}\int_{-1}^1(a+z)^{\frac{\gamma+1}{\gamma-1}}\left(1-z^2\right)^{\lambda+1}dz\\
\leq& a\int_{-1}^1(a+z)^{\frac{2}{\gamma-1}}\left(1-z^2\right)^{\lambda+1}dz+\frac{2\gamma-2}{\gamma+1}a\int_{-1}^1(a+z)^{\frac{2}{\gamma-1}}
\left(1-z^2\right)^{\lambda+1}dz,
\end{aligned}
\end{equation}
where the last inequality is from $z\le a$, since $a>1$ and $z\in[-1,1]$. Then,
\begin{align*}
\frac{4}{\gamma-1}-\left(1+\frac{2\gamma-2}{\gamma+1}\right)=\frac{4(\gamma+1)-(3\gamma-1)(\gamma-1)}{\gamma^2-1}=
\frac{(3\gamma+1)(3-\gamma3)}{\gamma^2-1}\ge0,
\end{align*}
by $1<\gamma\le2$. Combining this and \eqref{h231} yields
\begin{align*}
\int_{-1}^1(a+z)^{\frac{\gamma+1}{\gamma-1}}\left(1-z^2\right)^\lambda z^2\leq \frac{4a}{\gamma-1}\int_{-1}^1(a+z)^{\frac{2}{\gamma-1}}\left(1-z^2\right)^{\lambda+1}dz,
\end{align*}
which infers
\begin{align*}
h_{32}&=\int_{-1}^1\left[(a+z)^{\frac{\gamma+1}{\gamma-1}}+\frac{4a}{\gamma-1}(a+z)^{\frac{2}{\gamma-1}}+\left(\frac{2a}{\gamma-1}\right)^2
(a+z)^{\frac{2}{\gamma-1}-1}\right]\left(1-z^2\right)^\lambda z^2dz\\
&\leq \frac{4a}{\gamma-1}\int_{-1}^1(a+z)^{\frac{2}{\gamma-1}}\left(1-z^2\right)^{\lambda}dz+
\left(\frac{2a}{\gamma-1}\right)^2\int_{-1}^1(a+z)^{\frac{2}{\gamma-1}-1}\left(1-z^2\right)^\lambda dz.
\end{align*}
Therefore, it follows from \eqref{h23} that
\begin{align*}
h^2_3\leq&\frac{2a}{\gamma-1}\int_{-1}^1(a+z)^{\frac{2}{\gamma-1}-1}\left(1-z^2\right)^\lambda dz\\
&\times \int_{-1}^1\left[2(a+z)^{\frac{2}{\gamma-1}}+
\frac{2a}{\gamma-1}(a+z)^{\frac{2}{\gamma-1}-1}\right]\left(1-z^2\right)^{\lambda}dz\\
=&a\int_{-1}^1h'\left(1-z^2\right)^\lambda d z\cdot\int_{-1}^1\left[2h+a h'\right]\left(1-z^2\right)^\lambda d z\\
=&h_1h_2,
\end{align*}
by \eqref{h1} and \eqref{h2}. Thus, $A_{\rho\rho}A_{mm}-A^2_{\rho m}\ge0$.

\vspace{4pt}
\noindent {\bf Case III}: $0<a<1$. For simplicity, we write
$$
n=\frac{2}{\gamma-1}>7.
$$
Then, \eqref{hzeta} becomes
\begin{equation}\label{h1h2}
\begin{aligned}
h_1h_2-h^2_3=&\left[(n+3)\int_{-1}^1h\left(1-z^2\right)^{\frac{n-1}{2}} d z-(n-1)\int_{-1}^1  h\left(1-z^2\right)^{\frac{n-3}{2}}z^2 d z\right]\\
&\times\left[(n+1)\int_{-1}^1 h\left(1-z^2\right)^{\frac{n-1}{2}} d z-(n-1)\int_{-1}^1 h\left(1-z^2\right)^{\frac{n-3}{2}}z^2 d z\right]\\
&-\left[(n+3)\int_{-1}^1h\left(1-z^2\right)^{\frac{n-1}{2}} zdz-(n-1)\int_{-1}^1h\left(1-z^2\right)^{\frac{n-3}{2}}z^3 dz\right]^2,
\end{aligned}
\end{equation}
with $h=|a+z|^n$. Next, we consider separately two subcases.

\vspace{4pt}
\noindent\textbf{$Subcase~1: n\in(2k-1,2k]$}.
Notice $k\ge4$ from $n>7$. We first estimate $h_1$ and $h_2$. By Taylor's theorem (Lemma \ref{lem64}), it holds that for $h(\xi)=|\xi|^{n}$,
\begin{equation}\label{htaylor}
\begin{aligned}
h
=&h(z)+h'(z)a+\cdot\cdot\cdot+\frac{h^{(2k-3)}(a+z)}{(2k-3)!}a^{2k-3}+\frac{h^{(2k-2)}(a+z)}{(2k-2)!}a^{2k-2}\\
&+a^{2k-1}\int^{1}_{0}\frac{(1-s)^{2k-2}}{(2k-2)!}h^{(2k-1)}(sa+z)ds\\
=&|z|^{n}+n\frac{|z|^{n}}{z}a+\cdot\cdot\cdot+C^{2k-3}_{n}\frac{|z|^{n-2k+4}}{z}a^{2k-3}+C^{2k-2}_{n}|z|^{n-2k+2}a^{2k-2}\\
&+a^{2k-1}\int^{1}_{0}\frac{(1-s)^{2k-2}}{(2k-2)!}h^{(2k-1)}(sa+z)ds.
\end{aligned}
\end{equation}
Consequently,
\begin{align*}
(n-1)\int_{-1}^1|z|^{n-2k+2}\left(1-z^2\right)^{\frac{n-3}{2}}z^2 d z=&2(n-1)\int_{0}^1z^{n-2k+4}\left(1-z^2\right)^{\frac{n-3}{2}}dz\\
=&2(n-2k+3)\int_{0}^1z^{n-2k+2}\left(1-z^2\right)^{\frac{n-1}{2}}dz,
\end{align*}
and
\begin{align*}
&(n+3)\int^1_{-1}|z|^{n-2k+2}\left(1-z^2\right)^{\frac{n-1}{2}} d z-(n-1)\int_{-1}^1|z|^{n-2k+2}\left(1-z^2\right)^{\frac{n-3}{2}}z^2 d z\\
=&2(n+3)\int^1_{0}z^{n-2k+2}\left(1-z^2\right)^{\frac{n-1}{2}} d z-2(n-2k+3)\int^1_{0}z^{n-2k+2}\left(1-z^2\right)^{\frac{n-1}{2}} d z\\
=&2k\int^1_{0}\left(1-t\right)^{\frac{n-1}{2}}t^{\frac{n-2k+1}{2}}dt\\
=&2kB\left(\frac{n+1}{2},\frac{n-2k+3}{2}\right).
\end{align*}
Therefore,
it follows from \eqref{h1h2} and \eqref{htaylor} that
\begin{equation}\label{h1k-1}
\begin{aligned}
h_1=&\Sigma^{k-1}_{j=0}(2j+2)C^{2j}_{n}B\left(\frac{n+1}{2},\frac{n-2j+1}{2}\right)a^{2j}\\
&+(n+3)a^{2k-1}\int^{1}_{-1}\int^{1}_{0}\frac{(1-s)^{2k-2}}{(2k-2)!}h^{(2k-1)}(sa+z)\left(1-z^2\right)^{\frac{n-1}{2}} dsdz\\
&-(n-1)a^{2k-1}\int^{1}_{-1}\int^{1}_{0}\frac{(1-s)^{2k-2}}{(2k-2)!}h^{(2k-1)}(sa+z)\left(1-z^2\right)^{\frac{n-3}{2}}z^2 dsdz\\
=&\Sigma^{k-1}_{j=0}(2j+2)C^{2j}_{n}B\left(\frac{n+1}{2},\frac{n-2j+1}{2}\right)a^{2j}+f_2(a),
\end{aligned}
\end{equation}
and
\begin{equation}\label{h2k-1}
\begin{aligned}
h_2=&\Sigma^{k-1}_{j=0}(2j)C^{2j}_{n}B\left(\frac{n+1}{2},\frac{n-2j+1}{2}\right)a^{2j}\\
&+(n+1)a^{2k-1}\int^{1}_{-1}\int^{1}_{0}\frac{(1-s)^{2k-2}}{(2k-2)!}h^{(2k-1)}(sa+z)\left(1-z^2\right)^{\frac{n-1}{2}} dsdz\\
&-(n-1)a^{2k-1}\int^{1}_{-1}\int^{1}_{0}\frac{(1-s)^{2k-2}}{(2k-2)!}h^{(2k-1)}(sa+z)\left(1-z^2\right)^{\frac{n-3}{2}}z^2 dsdz\\
=&\Sigma^{k-1}_{j=0}(2j)C^{2j}_{n}B\left(\frac{n+1}{2},\frac{n-2j+1}{2}\right)a^{2j}+f_3(a).
\end{aligned}
\end{equation}
Through
integration by parts, it is easy to see
\begin{align*}
&(n-1)\int^{1}_{-1}\int^{1}_{0}\frac{(1-s)^{2k-2}}{(2k-2)!}h^{(2k-1)}(sa+z)\left(1-z^2\right)^{\frac{n-3}{2}}z^2 dsdz\\
=&\int^{1}_{-1}\int^{1}_{0}\frac{(1-s)^{2k-2}}{(2k-2)!}h^{(2k-1)}(sa+z)\left(1-z^2\right)^{\frac{n-1}{2}} dsdz\\
&+\int^{1}_{-1}\int^{1}_{0}\frac{(1-s)^{2k-2}}{(2k-2)!}h^{(2k)}(sa+z)\left(1-z^2\right)^{\frac{n-1}{2}}z dsdz\\
=&(n-2k+2)\int^{1}_{-1}\int^{1}_{0}\frac{(1-s)^{2k-2}}{(2k-2)!}h^{(2k-1)}(sa+z)\left(1-z^2\right)^{\frac{n-1}{2}} dsdz\\
&-a\int^{1}_{-1}\int^{1}_{0}\frac{s(1-s)^{2k-2}}{(2k-2)!}h^{(2k)}(sa+z)\left(1-z^2\right)^{\frac{n-1}{2}} dsdz,
\end{align*}
due to
\begin{equation*}
(sa+z)h^{(2k)}(sa+z)=(n-2k+1)h^{(2k-1)}(sa+z).
\end{equation*}
Hence,
\begin{equation}\label{f2f3}
\begin{aligned}
f_3(a)=&(2k-1)\int^{1}_{-1}\int^{1}_{0}\frac{(1-s)^{2k-2}}{(2k-2)!}h^{(2k-1)}(sa+z)\left(1-z^2\right)^{\frac{n-1}{2}} dsdz\\
&+a\int^{1}_{-1}\int^{1}_{0}\frac{s(1-s)^{2k-2}}{(2k-2)!}h^{(2k)}(sa+z)\left(1-z^2\right)^{\frac{n-1}{2}} dsdz\\
\ge&(2k-1)\int^{1}_{-1}\int^{1}_{0}\frac{(1-s)^{2k-2}}{(2k-2)!}h^{(2k-1)}(sa+z)\left(1-z^2\right)^{\frac{n-1}{2}} dsdz\\
=&:(2k-1)f_4(a).
\end{aligned}
\end{equation}
Here we have used $h^{(2k)}\ge0, a>0$ and $s\ge0$. It is obvious that $f_4(a)=0$ and
\begin{equation*}
f'_4(a)=\int^{1}_{-1}\int^{1}_{0}\frac{s(1-s)^{2k-2}}{(2k-2)!}h^{(2k)}(sa+z)\left(1-z^2\right)^{\frac{n-1}{2}} dsdz\ge0.
\end{equation*}
In conclusion,
$$
f_3(a)\ge(2k-1)f_4(a)\ge(2k-1)f_4(0)=0.
$$
Similarly, we know that
\begin{align*}
f_2(a)=&(n+3)\int^{1}_{-1}\int^{1}_{0}\frac{(1-s)^{2k-2}}{(2k-2)!}h^{(2k-1)}(sa+z)\left(1-z^2\right)^{\frac{n-1}{2}} dsdz\\
&-(n-1)\int^{1}_{-1}\int^{1}_{0}\frac{(1-s)^{2k-2}}{(2k-2)!}h^{(2k-1)}(sa+z)\left(1-z^2\right)^{\frac{n-3}{2}}z^2 dsdz\ge0.
\end{align*}
As a result, using the above two equalities, \eqref{h1k-1} and \eqref{h2k-1}, we obtain
\begin{equation}\label{h1h2leq}
\begin{aligned}
h_1h_2\ge&\left[\Sigma^{k-1}_{j=0}(2j+2)C^{2j}_{n}B\left(\frac{n+1}{2},\frac{n-2j+1}{2}\right)a^{2j}\right] \\ &\times\left[\Sigma^{k-1}_{j=0}2jC^{2j}_{n}B\left(\frac{n+1}{2},\frac{n-2j+1}{2}\right)a^{2j}\right].
\end{aligned}
\end{equation}

As for $h_3$, we see from \eqref{htaylor} that
\begin{align*}
&(n+3)\int^{1}_{-1}\frac{|z|^{n-2k+4}}{z}\left(1-z^2\right)^{\frac{n-1}{2}}zdz-(n-1)\int^{1}_{-1}\frac{|z|^{n-2k+4}}{z}\left(1-z^2\right)^{\frac{n-3}{2}}z^3dz\\
=&(n+3)\int^{1}_{-1}|z|^{n-2k+4}\left(1-z^2\right)^{\frac{n-1}{2}}dz-(n-1)\int^{1}_{-1}|z|^{n-2k+4}\left(1-z^2\right)^{\frac{n-3}{2}}z^2dz\\
=&(2k-2)\int^{1}_{-1}|z|^{n-2k+4}\left(1-z^2\right)^{\frac{n-1}{2}}dz\\
=&(2k-2)B\left(\frac{n+1}{2},\frac{n-2k+5}{2}\right),
\end{align*}
which infers
\begin{equation}\label{h3f5}
\begin{aligned}
h_3=&\Sigma^{k-2}_{j=0}(2j+2)C^{2j+1}_{n}B\left(\frac{n+1}{2},\frac{n-2j+1}{2}\right)a^{2j+1}\\
&+(n+3)a^{2k-1}\int^{1}_{-1}\int^{1}_{0}\frac{(1-s)^{2k-2}}{(2k-2)!}h^{(2k-1)}(sa+z)\left(1-z^2\right)^{\frac{n-1}{2}}z dsdz\\
&-(n-1)a^{2k-1}\int^{1}_{-1}\int^{1}_{0}\frac{(1-s)^{2k-2}}{(2k-2)!}h^{(2k-1)}(sa+z)\left(1-z^2\right)^{\frac{n-3}{2}}z^3 dsdz\\
=&\Sigma^{k-2}_{j=0}(2j+2)C^{2j+1}_{n}B\left(\frac{n+1}{2},\frac{n-2j+1}{2}\right)a^{2j+1}\\
&+a^{2k-1}\left[(n+3)f_5(a)-(n-1)f_6(a)\right].
\end{aligned}
\end{equation}
Then, the derivative of $f_5(a)$ is
\begin{align*}
f'_5(a)&=\int^{1}_{-1}\int^{1}_{0}\frac{s(1-s)^{2k-2}}{(2k-2)!}h^{(2k)}(sa+z)\left(1-z^2\right)^{\frac{n-1}{2}}z dsdz\\
&=\int^{1}_{0}\frac{s(1-s)^{2k-2}}{(2k-2)!}\prod^{2k-1}_{j=0}(n-j)\left[\int^{1}_{-1}|sa+z|^{n-2k}\left(1-z^2\right)^{\frac{n-1}{2}}zdz\right]ds
\le0,
\end{align*}
by
\begin{equation}\label{f5n<2k}
\begin{aligned}
&\int^{1}_{-1}|sa+z|^{n-2k}\left(1-z^2\right)^{\frac{n-1}{2}}zdz\\
=&\int^{1}_{0}|sa+z|^{n-2k}\left(1-z^2\right)^{\frac{n-1}{2}}zdz+\int^{0}_{-1}|sa+z|^{n-2k}\left(1-z^2\right)^{\frac{n-1}{2}}zdz\\
=&\int^{1}_{0}|sa+z|^{n-2k}\left(1-z^2\right)^{\frac{n-1}{2}}zdz-\int^{1}_{0}|sa-z|^{n-2k}\left(1-z^2\right)^{\frac{n-1}{2}}zdz\le0,
\end{aligned}
\end{equation}
with $s,a\ge0$ and $n-2k\le0$. Similarly, we have $f'_6(a)\le0$. Hence,
\begin{align*}
f_5(a)\leq f_5(0)&=\int^{1}_{-1}\int^{1}_{0}\frac{(1-s)^{2k-2}}{(2k-2)!}h^{(2k-1)}(z)\left(1-z^2\right)^{\frac{n-1}{2}}z dsdz\\
&=\int^{1}_{-1}\frac{1}{(2k-1)!}h^{(2k-1)}(z)\left(1-z^2\right)^{\frac{n-1}{2}}z dz\\
&=C^{2k-1}_{n}\int^{1}_{-1}\frac{|z|^{n-2k+2}}{z}\left(1-z^2\right)^{\frac{n-1}{2}}z dz\\
&=C^{2k-1}_{n}\int^{1}_{0}t^{\frac{n-2k+1}{2}}(1-t)^{\frac{n-1}{2}}dt\\
&=C^{2k-1}_{n}B\left(\frac{n+1}{2},\frac{n-2k+3}{2}\right).
\end{align*}
Since $\int^{1}_{-1}h^{(2k-1)}(s+z)\left(1-z^2\right)^{\frac{n-3}{2}}z^3dz$ is a non increasing in $s$, we deduce (from \eqref{f5n<2k} again)
\begin{align*}
f_6(a)\ge f_6(1)&=\int^{1}_{-1}\int^{1}_{0}\frac{(1-s)^{2k-2}}{(2k-2)!}h^{(2k-1)}(s+z)\left(1-z^2\right)^{\frac{n-3}{2}}z^3 dsdz\\
&\ge\int^{1}_{-1}\int^{1}_{0}\frac{(1-s)^{2k-2}}{(2k-2)!}h^{(2k-1)}(1+z)\left(1-z^2\right)^{\frac{n-3}{2}}z^3 dsdz\\
&=\int^{1}_{0}\frac{(1-s)^{2k-2}}{(2k-2)!}\prod^{2k-2}_{j=0}(n-j)\left[\int^{1}_{-1}(1+z)^{n-2k+1}
\left(1-z^2\right)^{\frac{n-3}{2}}z^3dz\right]ds\ge0,
\end{align*}
owing to $n-2k+1\ge0$ and
\begin{align*}
&\int^{1}_{-1}(1+z)^{n-2k+1}\left(1-z^2\right)^{\frac{n-3}{2}}z^3dz\\
=&\int^{1}_{0}(1+z)^{n-2k+1}\left(1-z^2\right)^{\frac{n-3}{2}}z^3dz-\int^{1}_{0}(1-z)^{n-2k+1}\left(1-z^2\right)^{\frac{n-3}{2}}z^3dz\ge0.
\end{align*}
Then, it follows from \eqref{h3f5} that
\begin{align*}
h_3\leq&\Sigma^{k-2}_{j=0}(2j+2)C^{2j+1}_{n}B\left(\frac{n+1}{2},\frac{n-2j+1}{2}\right)a^{2j+1}\\
&+a^{2k-1}(n+3)C^{2k-1}_{n}B\left(\frac{n+1}{2},\frac{n-2k+3}{2}\right),
\end{align*}
and so $h_1h_2-h^{2}_3\ge0$, in view of the first result of Lemma \ref{lem65} and \eqref{h1h2leq}.

\vspace{4pt}
\noindent\textbf{$Subcase~2: n\in(2k,2k+1]$}.
As in Subcase 1, by Taylor's theorem we have
\begin{equation}\label{htaylor1}
\begin{aligned}
h=&h(z)+h'(z)a+\cdot\cdot\cdot+\frac{h^{(2k-2)}(a+z)}{(2k-2)!}a^{2k-2}+\frac{h^{(2k-1)}(a+z)}{(2k-1)!}a^{2k-1}\\
&+a^{2k-1}\int^{1}_{0}\frac{(1-s)^{2k-2}}{(2k-2)!}h^{(2k-1)}(sa+z)ds\\
=&|z|^{n}+n\frac{|z|^{n}}{z}a+\cdot\cdot\cdot+C^{2k-2}_{n}|z|^{n-2k+2}a^{2k-2}+C^{2k-1}_{n}\frac{|z|^{n-2k+2}}{z}a^{2k-1}\\
&+a^{2k}\int^{1}_{0}\frac{(1-s)^{2k-1}}{(2k-1)!}h^{(2k)}(sa+z)ds.
\end{aligned}
\end{equation}
Since
$$
\int^{1}_{-1}\frac{|z|^{n-2k+2}}{z}\left(1-z^2\right)^{\frac{n-1}{2}}dz=0,
$$
and
$$
\int^{1}_{-1}\frac{|z|^{n-2k+2}}{z}\left(1-z^2\right)^{\frac{n-3}{2}}z^2dz=0,
$$
we obtain (instead of \eqref{h1k-1} and \eqref{h2k-1})
\begin{align*}
h_1=&\Sigma^{k-1}_{j=0}(2j+2)C^{2j}_{n}B\left(\frac{n+1}{2},\frac{n-2j+1}{2}\right)a^{2j}\\
&+(n+3)a^{2k}\int^{1}_{-1}\int^{1}_{0}\frac{(1-s)^{2k-1}}{(2k-1)!}h^{(2k)}(sa+z)\left(1-z^2\right)^{\frac{n-1}{2}} dsdz\\
&-(n-1)a^{2k}\int^{1}_{-1}\int^{1}_{0}\frac{(1-s)^{2k-1}}{(2k-1)!}h^{(2k)}(sa+z)\left(1-z^2\right)^{\frac{n-3}{2}}z^2 dsdz\\
=&\Sigma^{k-1}_{j=0}(2j+2)C^{2j}_{n}B\left(\frac{n+1}{2},\frac{n-2j+1}{2}\right)a^{2j}+a^{2k}f_7(a),
\end{align*}
and
\begin{align*}
h_2=&\Sigma^{k-1}_{j=0}(2j)C^{2j}_{n}B\left(\frac{n+1}{2},\frac{n-2j+1}{2}\right)a^{2j}\\
&+(n+1)a^{2k}\int^{1}_{-1}\int^{1}_{0}\frac{(1-s)^{2k-1}}{(2k-1)!}h^{(2k)}(sa+z)\left(1-z^2\right)^{\frac{n-1}{2}} dsdz\\
&-(n-1)a^{2k}\int^{1}_{-1}\int^{1}_{0}\frac{(1-s)^{2k-1}}{(2k-1)!}h^{(2k)}(sa+z)\left(1-z^2\right)^{\frac{n-3}{2}}z^2 dsdz\\
=&\Sigma^{k-1}_{j=0}(2j)C^{2j}_{n}B\left(\frac{n+1}{2},\frac{n-2j+1}{2}\right)a^{2j}+a^{2k}f_8(a).
\end{align*}

We first estimate $f_8(a)$, since
$$
f_7(a)=f_8(a)+2\int^{1}_{-1}\int^{1}_{0}\frac{(1-s)^{2k-1}}{(2k-1)!}h^{(2k)}(sa+z)\left(1-z^2\right)^{\frac{n-1}{2}} dsdz.
$$
Integrate by parts (just like with $f_2(a)$) and replace $k$ with $k+\frac{1}{2}$ in \eqref{f2f3}; so we get
\begin{align*}
f_7(a)=&2k\int^{1}_{-1}\int^{1}_{0}\frac{(1-s)^{2k-1}}{(2k-1)!}h^{(2k)}(sa+z)\left(1-z^2\right)^{\frac{n-1}{2}} dsdz\\
&+a\int^{1}_{-1}\int^{1}_{0}\frac{s(1-s)^{2k-1}}{(2k-1)!}h^{(2k+1)}(sa+z)\left(1-z^2\right)^{\frac{n-1}{2}} dsdz\\
\ge&2k\int^{1}_{-1}\int^{1}_{0}\frac{(1-s)^{2k-1}}{(2k-1)!}h^{(2k)}(sa+z)\left(1-z^2\right)^{\frac{n-1}{2}} dsdz\\
=&:2k\int^{1}_{0}\frac{(1-s)^{2k-1}}{(2k-1)!}f_9(a)ds.
\end{align*}
Then,
\begin{align*}
&\frac{f'_9(a)}{s\prod^{2k}_{j=0}(n-j)}\\
=&\int^{1}_{-1}\frac{|sa+z|^{n-2k}}{sa+z}\left(1-z^2\right)^{\frac{n-1}{2}}dz\\
=&\left[\int^{-sa}_{-1}\frac{|sa+z|^{n-2k}}{sa+z}\left(1-z^2\right)^{\frac{n-1}{2}}dz+
\int^{1}_{-sa}\frac{|sa+z|^{n-2k}}{sa+z}\left(1-z^2\right)^{\frac{n-1}{2}}dz\right]\\
=&\left[\int^{0}_{-1+sa}\frac{|t|^{n-2k}}{t}\left(1-(t-sa)^2\right)^{\frac{n-1}{2}}dt+
\int^{1+sa}_{0}t^{n-2k-1}\left(1-(t-sa)^2\right)^{\frac{n-1}{2}}dt\right]\\
\ge&\int^{1-sa}_{0}t^{n-2k-1}\left[\left(1-(t-sa)^2\right)^{\frac{n-1}{2}}-
\left(1-(t+sa)^2\right)^{\frac{n-1}{2}}\right]dt
\ge0,
\end{align*}
by
$$
1-(t-sa)^2\ge 1-(t+sa)^2, \,\forall t\in [0,1-sa].
$$
As a result,
\begin{align*}
f_8(a)\ge 2k\int^{1}_{0}\frac{(1-s)^{2k-1}}{(2k-1)!}f_9(0)ds=&2kC^{2k}_{n}\int^{1}_{-1}|z|^{n-2k}\left(1-z^2\right)^{\frac{n-1}{2}} dz\\
=&2kC^{2k}_{n}B\left(\frac{n+1}{2},\frac{n-2k+1}{2}\right),
\end{align*}
and
\begin{align*}
f_7(a)=f_8(a)+2\int^{1}_{-1}\int^{1}_{0}\frac{(1-s)^{2k-1}}{(2k-1)!}f_9(a) ds\ge(2k+2)C^{2k}_{n}B\left(\frac{n+1}{2},\frac{n-2k+1}{2}\right).
\end{align*}
Hence,
\begin{equation}\label{h1h2leq'}
\begin{aligned}
h_1h_2\ge&\left[\Sigma^{k}_{j=0}(2j+2)C^{2j}_{n}B\left(\frac{n+1}{2},\frac{n-2j+1}{2}\right)a^{2j}\right] \\
&\times\left[\Sigma^{k}_{j=0}2jC^{2j}_{n}B\left(\frac{n+1}{2},\frac{n-2j+1}{2}\right)a^{2j}\right].
\end{aligned}
\end{equation}

As for $h_3$, taking the additional term $C^{2k-1}_{n}\frac{|z|^{n-2k+2}}{z}a^{2k-1}$ in \eqref{htaylor1} into account, we deduce
\begin{equation}\label{h3f51}
\begin{aligned}
h_3=&\Sigma^{k-1}_{j=0}(2j+2)C^{2j+1}_{n}B\left(\frac{n+1}{2},\frac{n-2j+1}{2}\right)a^{2j+1}\\
&+(n+3)a^{2k}\int^{1}_{-1}\int^{1}_{0}\frac{(1-s)^{2k-1}}{(2k-1)!}h^{(2k)}(sa+z)\left(1-z^2\right)^{\frac{n-1}{2}}z dsdz\\
&-(n-1)a^{2k}\int^{1}_{-1}\int^{1}_{0}\frac{(1-s)^{2k-1}}{(2k-1)!}h^{(2k)}(sa+z)\left(1-z^2\right)^{\frac{n-3}{2}}z^3 dsdz\\
\leq&\Sigma^{k-1}_{j=0}(2j+2)C^{2j+1}_{n}B\left(\frac{n+1}{2},\frac{n-2j+1}{2}\right)a^{2j+1}+(n+3)a^{2k}f_{10}(a),
\end{aligned}
\end{equation}
where the last inequality is from
\begin{align*}
&\int^{1}_{-1}\int^{1}_{0}\frac{(1-s)^{2k-1}}{(2k-1)!}h^{(2k)}(sa+z)\left(1-z^2\right)^{\frac{n-3}{2}}z^3 dsdz\\
=&2kC^{2k}_{n}\int^{1}_{-1}\int^{1}_{0}(1-s)^{2k-1}|sa+z|^{n-2k}\left(1-z^2\right)^{\frac{n-3}{2}}z^3 dsdz\\
=&2kC^{2k}_{n}\int^{1}_{0}(1-s)^{2k-1}ds\bigg\{\int^{1}_{0}\left[(sa+z)^{n-2k}-|sa-z|^{n-2k}\right]\left(1-z^2\right)^{\frac{n-3}{2}}z^3dz\bigg\}
\ge0,
\end{align*}
by $n-2k\ge0$. Then, in view of Lemma \ref{lem62}, $0<a<1$ and $n-2k\in(0,1]$, it holds that
\begin{equation}\label{f9}
\begin{aligned}
f_{10}(a)\le&\int^{1}_{0}\int^{1}_{0}\frac{(1-s)^{2k-1}}{(2k-1)!}h^{(2k)}(sa+z)\left(1-z^2\right)^{\frac{n-1}{2}}z dsdz\\
=&C^{2k-1}_{n}(n-2k+1)\int^{1}_{0}\int^{1}_{0}(1-s)^{2k-1}(sa+z)^{n-2k}\left(1-z^2\right)^{\frac{n-1}{2}}zdsdz\\
\le&C^{2k-1}_{n}(n-2k+1)\int^{1}_{0}\int^{1}_{0}(1-s)^{2k-1}\left(s^{n-2k}+z^{n-2k}\right)\left(1-z^2\right)^{\frac{n-1}{2}}zdsdz\\
=:&C^{2k-1}_{n}(n-2k+1)\left(K_1+K_2\right).
\end{aligned}
\end{equation}
On one hand, it follows from $n\in(2k,2k+1]$ that
\begin{equation}\label{K1}
\begin{aligned}
K_1=&\int^{1}_{0}\int^{1}_{0}(1-s)^{2k-1}s^{n-2k}\left(1-z^2\right)^{\frac{n-1}{2}}zdsdz\\
=&\frac{1}{2}\int^{1}_{0}\left(1-t\right)^{\frac{n-1}{2}}dt\cdot\int^{1}_{0}(1-s)^{2k-1}s^{n-2k}ds\\
=&\frac{1}{2}B\left(\frac{n+1}{2},1\right)\int^{1}_{0}(1-s)^{2k-1}s^{n-2k}ds\\
\leq& \frac{1}{4k}B\left(\frac{n+1}{2},\frac{n-2k+1}{2}\right),
\end{aligned}
\end{equation}
where the last inequality is from
\begin{align*}
\int^{1}_{0}(1-s)^{2k-1}s^{n-2k}ds=&\frac{2k-1}{n-2k+1}\int^{1}_{0}(1-s)^{2k-2}s^{n-2k+1}ds\\
=&\frac{2k-1}{n-2k+1}\cdot\cdot\cdot\frac{2}{n-2}\int^{1}_{0}(1-s)s^{n-2}ds\\
=&\frac{2k-1}{n-2k+1}\cdot\cdot\cdot\frac{2}{n-2}\cdot\frac{1}{n-1}\int^{1}_{0}s^{n-1}ds\\
=&\frac{2k-1}{n-2k+1}\cdot\cdot\cdot\frac{2}{n-2}\cdot\frac{1}{n-1}\cdot\frac{1}{n}\\
=&\frac{1}{n}\cdot\frac{2k-1}{n-1}\cdot\frac{2k-2}{n-2}\cdot\cdot\cdot\frac{1}{n-2k+1}\
\leq \frac{1}{2k},
\end{align*}
by $n\in(2k,2k+1]$.
On the other hand,
\begin{equation}\label{K2}
\begin{aligned}
K_2=&\int^{1}_{0}\int^{1}_{0}(1-s)^{2k-1}z^{n-2k}\left(1-z^2\right)^{\frac{n-1}{2}}zdz\\
=&\frac{1}{4k}\int^{1}_{0}t^{\frac{n-2k}{2}}\left(1-t\right)^{\frac{n-1}{2}}dt\\
=&\frac{1}{4k}B\left(\frac{n+1}{2},\frac{n-2k+2}{2}\right)\\
\leq&\frac{1}{4k}B\left(\frac{n+1}{2},\frac{n-2k+1}{2}\right).
\end{aligned}
\end{equation}
Plugging \eqref{K1} and \eqref{K2} into \eqref{f9} yields
\begin{align*}
f_{10}(a)\leq C^{2k}_{n}B\left(\frac{n+1}{2},\frac{n-2k+1}{2}\right),
\end{align*}
which infers from \eqref{h3f51} that
\begin{align*}
h_3\leq&\Sigma^{k-1}_{j=0}(2j+2)C^{2j+1}_{n}B\left(\frac{n+1}{2},\frac{n-2j+1}{2}\right)a^{2j+1}\\
&+(n+3)a^{2k}C^{2k}_{n}B\left(\frac{n+1}{2},\frac{n-2k+1}{2}\right).
\end{align*}
This together with the second result of Lemma \ref{lem65} and \eqref{h1h2leq'} justifies $h_1h_2-h^{2}_3\ge0.$

In conclusion, we have shown $A_{mm}\ge0$ and $A_{\rho\rho}A_{mm}-A^{2}_{\rho m}\ge0$, which completes the proof.

\end{proof}

In the following, we address the convexity of another function, based on the relationship between $A, \rho^{\gamma+1}, m^2$ and $B$.
\begin{lemma}\label{lem69}
Let $g(\xi)=\frac{(\gamma-1)^2}{2\gamma(\gamma+1)}|\xi|^{\frac{2\gamma}{\gamma-1}}$ with $\gamma\in(1,2]$ and $B=\rho^{\theta+1}\int g'z\left(1-z^2\right)^\lambda \mathrm{d} z$. Then, there exist three positive constants $M_0$, $M_1$ and $M_2$, such that
$$
D:=M_0A+M_1\rho^{\gamma+1}+M_2m^2-B
$$
is convex in $(\rho, m)$ for all $\rho, m$, and $D_*\ge0$.
\end{lemma}

\begin{proof}
First we observe
\begin{align*}
D
=&(M_0+1)A+M_1\rho^{\gamma+1}+M_2m^2-\frac{2\gamma}{\gamma-1}\eta\\
=&M_3A+\frac{M_4}{\gamma(\gamma+1)}\rho^{\gamma+1}+M_2m^2-\frac{2\gamma}{\gamma-1}\eta,
\end{align*}
with $M_3=M_0+1$ and $M_4=M_1(\gamma+1)\gamma$. Then, we distinguish two cases.

\vspace{4pt}
\noindent\textbf{$Case~ I: a>3$}. It suffices to prove that
$$
\tilde{D}=M_3A-\frac{2\gamma}{\gamma-1}\eta
$$
is convex in $(\rho, m)$ for all $\rho, m$, since the convexity of $\rho^{\gamma+1}$ and $m^2$ is obvious. From Lemma \ref{lem67}, with $c_0=M_3$ and $c_1=\frac{2\gamma}{\gamma-1}$, we get
\begin{align*}
&\theta^{-2}\rho^{-2\theta}\left(\tilde{D}_{\rho\rho}\tilde{D}_{mm}-\tilde{D}^2_{\rho m}\right)\\
=&\left[\left(\frac{3\gamma-1}{\gamma-1}M_3-\frac{2\gamma}{\gamma-1}\right)\int_{-1}^1h\left(1-z^2\right)^\lambda d z-2\lambda M_3\int_{-1}^1  h\left(1-z^2\right)^{\lambda-1}z^2 d z\right]\\
&\times\left[\left(\frac{\gamma+1}{\gamma-1}M_3-\frac{2\gamma}{\gamma-1}\right)\int_{-1}^1 h\left(1-z^2\right)^\lambda d z-2\lambda M_3\int_{-1}^1 h\left(1-z^2\right)^{\lambda-1}z^2 d z\right]\\
&-\left[\left(\frac{3\gamma-1}{\gamma-1}M_3-\frac{2\gamma}{\gamma-1}\right)\int_{-1}^1h\left(1-z^2\right)^{\lambda} zdz-2\lambda M_3\int_{-1}^1h\left(1-z^2\right)^{\lambda-1}z^3 dz\right]^2\\
=&\tilde{D}_1\tilde{D}_2-\tilde{D}^2_3.
\end{align*}
Integrating by parts and choosing
$
M_3=\frac{2\gamma}{\gamma-1},
$
we see
\begin{align*}
\tilde{D}_1&=\left(2M_3-\frac{2\gamma}{\gamma-1}\right)\int_{-1}^1h\left(1-z^2\right)^\lambda d z+M_3a\int_{-1}^1h'\left(1-z^2\right)^\lambda d z\\
&=M_3\int_{-1}^1h\left(1-z^2\right)^\lambda d z+M_3a\int_{-1}^1h'\left(1-z^2\right)^\lambda d z,
\end{align*}
$$
\tilde{D}_2=M_3a\int_{-1}^1h'\left(1-z^2\right)^\lambda d z-M_3\int_{-1}^1h\left(1-z^2\right)^\lambda d z,
$$
and
\begin{align*}
\tilde{D}_3=&\left(M_3-\frac{2\gamma}{\gamma-1}\right)\int_{-1}^1h\left(1-z^2\right)^\lambda zd z+M_3a\int_{-1}^1h'\left(1-z^2\right)^\lambda zd z\\
=&M_3a\int_{-1}^1h'\left(1-z^2\right)^\lambda zd z.
\end{align*}
As a result, by the Cauchy-Schwarz inequality and
$$
h'(a+z)=\frac{2}{\gamma-1}(a+z)^{\frac{2}{\gamma-1}-1},
$$
we derive
\begin{align*}
&M^{-2}_{3}\theta^{-2}\rho^{-2\theta}\left(\tilde{D}_{\rho\rho}\tilde{D}_{mm}-\tilde{D}^2_{\rho m}\right)\\
=&\left(a\int_{-1}^1h'\left(1-z^2\right)^\lambda d z\right)^2-\left(\int_{-1}^1h\left(1-z^2\right)^\lambda d z\right)^2-\left(a\int_{-1}^1h'\left(1-z^2\right)^\lambda zd z\right)^2\\
=&\left(\frac{2a}{\gamma-1}\int_{-1}^1(a+z)^{\frac{2}{\gamma-1}-1}\left(1-z^2\right)^\lambda d z\right)^2-\left(\int_{-1}^1(a+z)(a+z)^{\frac{2}{\gamma-1}-1}\left(1-z^2\right)^\lambda d z\right)^2\\
&-\left(\frac{2a}{\gamma-1}\int_{-1}^1(a+z)^{\frac{2}{\gamma-1}-1}\left(1-z^2\right)^\lambda zd z\right)^2\\
\ge& \left[\left(\frac{2a}{\gamma-1}\right)^2-2a^2\right]\left(\int_{-1}^1(a+z)^{\frac{3-\gamma}{\gamma-1}}\left(1-z^2\right)^\lambda d z\right)^2\\
&-\left[\left(\frac{2a}{\gamma-1}\right)^2+2\right]
\left(\int_{-1}^1(a+z)^{\frac{3-\gamma}{\gamma-1}}\left(1-z^2\right)^\lambda zd z\right)^2,
\end{align*}
Moreover,
\begin{align*}
\int_{-1}^1(a+z)^{\frac{3-\gamma}{\gamma-1}}\left(1-z^2\right)^\lambda zd z&=\frac{1}{2(\lambda+1)}\frac{3-\gamma}{\gamma-1}
\int_{-1}^1(a+z)^{\frac{4-2\gamma}{\gamma-1}}\left(1-z^2\right)^{\lambda+1} d z\\
&=\frac{3-\gamma}{\gamma+1}\int_{-1}^1(a+z)^{\frac{4-2\gamma}{\gamma-1}}\left(1-z^2\right)^{\lambda+1} d z\\
&\le\int_{-1}^1(a+z)^{\frac{4-2\gamma}{\gamma-1}}\left(1-z^2\right)^{\lambda} d z,
\end{align*}
by $\lambda=\frac{3-\gamma}{2(\gamma+1)}$ and $\gamma>1$. Consequently, we obtain
\begin{align*}
&M^{-2}_{3}\theta^{-2}\rho^{-2\theta}\left(\tilde{D}_{\rho\rho}\tilde{D}_{mm}-\tilde{D}^2_{\rho m}\right)\\
\ge& \left[\left(\frac{2a}{\gamma-1}\right)^2-2a^2\right]\left(\int_{-1}^1(a+z)^{\frac{4-2\gamma}{\gamma-1}}(a+z)\left(1-z^2\right)^\lambda d z\right)^2\\
&-\left[\left(\frac{2a}{\gamma-1}\right)^2+2\right]\left(\int_{-1}^1(a+z)^{\frac{4-2\gamma}{\gamma-1}}
\left(1-z^2\right)^{\lambda} d z\right)^2\\
=&\left\{\left[\left(\frac{2a}{\gamma-1}\right)^2-2a^2\right](a-1)^2-\left[\left(\frac{2a}{\gamma-1}\right)^2+2\right]\right\}
\left(\int_{-1}^1(a+z)^{\frac{4-2\gamma}{\gamma-1}}\left(1-z^2\right)^{\lambda} d z\right)^2\\
>&0,
\end{align*}
where the last inequality is from
\begin{align*}
&\left[\left(\frac{2a}{\gamma-1}\right)^2-2a^2\right](a-1)^2-\left[\left(\frac{2a}{\gamma-1}\right)^2+2\right]\\
=&\left(\frac{2a}{\gamma-1}\right)^2(a^2-2a)-2a^2(a-1)^2-2\\
=&2a^2(a^2-2a-1)-2>0,
\end{align*}
by $a>3$ and $\gamma\le2$. In addition, it follows from
$$
c_0=M_3=c_1=\frac{2\gamma}{\gamma-1}
$$
and \eqref{psimm} that
\begin{align*}
\tilde{D}_{mm}=&\frac{2\gamma}{\gamma-1}\int_{-1}^1|a+z|^\frac{2}{\gamma-1}\left(1-z^2\right)^\lambda d z+\frac{4\gamma^2a}{(\gamma-1)^2}\int_{-1}^1\frac{|a+z|^\frac{2}{\gamma-1}}{a+z}\left(1-z^2\right)^{\lambda}dz>0,
\end{align*}
by \eqref{ka}. Thus we show the convexity for this case.

\vspace{4pt}
\noindent\textbf{$Case~ II: 0\leq a\le3$}.
We are going to prove that
$$
\hat{D}=\frac{M_4}{\gamma(\gamma+1)}\rho^{\gamma+1}+M_2m^2-\frac{2\gamma}{\gamma-1}\eta,
$$
is convex in $(\rho, m)$ for all $\rho, m$, which is relatively easy to prove compared to the previous one, since $\forall z\in[-1,1]$,
$$
h=\left|a+z\right|^{\frac{2}{\gamma-1}}\leq |1+a|^{\frac{2}{\gamma-1}}\leq 4^{\frac{2}{\gamma-1}},
$$
and
$$
(-a+z\theta)^{2}\le(3+\theta)^2\leq 16,
$$
by $\theta=\frac{\gamma-1}{2}$ and $\gamma\in(1,2]$.

From \eqref{Hessiaeta} and with $h=h$, we compute the Hessian matrix of $\hat{D}$ as follows
\begin{align*}
\hat{D}_{\rho\rho}=&M_4\rho^{\gamma-1}-\frac{2\gamma}{\gamma-1}\eta_{\rho\rho}\\
=&\rho^{2\theta}\left\{M_4-\frac{2\gamma}{\gamma-1}\left[\int^{1}_{-1}\theta^2h \left(1-z^2\right)^{\lambda+1} dz  +\int^{1}_{-1}\left(-a+z\theta \right)^2 h\left(1-z^2\right)^\lambda d z\right]\right\}\\
\ge&\rho^{2\theta}\left(M_4-\frac{2\gamma}{\gamma-1}4^{\frac{2}{\gamma-1}+1}-\frac{2\gamma}{\gamma-1}4^{\frac{2}{\gamma-1}+3}\right)\\
\ge &\rho^{2\theta}\frac{M_4}{2},
\end{align*}
provided that
\begin{equation}\label{M4}
M_4>\frac{2\gamma}{\gamma-1}4^{\frac{2}{\gamma-1}+4},
\end{equation}
\begin{align*}
\hat{D}_{mm}&=2M_2-\frac{2\gamma}{\gamma-1}\eta_{m m}=2M_2- \frac{2\gamma}{\gamma-1}\int^{1}_{-1} h\left(1-z^2\right)^\lambda d z\\
&\ge2M_2-\frac{2\gamma}{\gamma-1}4^{\frac{2}{\gamma-1}+1}\ge M_2,
\end{align*}
provided that
\begin{equation}\label{M2}
M_2>\frac{2\gamma}{\gamma-1}4^{\frac{2}{\gamma-1}+2},
\end{equation}
and
\begin{align*}
\big|\hat{D}_{\rho m}\big|&=\bigg|\frac{2\gamma}{\gamma-1}\eta_{\rho m}\bigg|\\
&=\bigg|\frac{2\gamma}{\gamma-1}\int^{1}_{-1}\left(-\frac{m}{\rho^2}+z\theta\rho^{\theta-1}\right) g'' \left(1-z^2\right)^\lambda dz\bigg|\\
&=\bigg|\frac{2\gamma}{\gamma-1}\rho^{\theta}\int^{1}_{-1}\left(-a+z\theta\right) h \left(1-z^2\right)^\lambda dz\bigg|\\
&\le\frac{2\gamma}{\gamma-1}\rho^{\theta}4^{\frac{2}{\gamma-1}+2},
\end{align*}
which infers that $D_{mm}>0$ and
$$
\hat{D}_{\rho \rho}\hat{D}_{m m}-\hat{D}^2_{\rho m}\ge 7\rho^{2\theta}\left(\frac{2\gamma}{\gamma-1}4^{\frac{2}{\gamma-1}+2}\right)^2>0,
$$
by \eqref{M4} and \eqref{M2}. Therefore, we get the convexity of $\hat{D}$.

In conclusion, we complete the proof.
\end{proof}
\subsubsection{Final stage of the proof}
Making use of Lemma \ref{lem2}, Lemma \ref{lem69} and \eqref{eta*ga+1}, we deduce
\begin{align*}
&\int^{t}_{0}\int_{-\infty}^{+\infty}(1+\tau)^{\nu(\varepsilon)-1}B_*dxd\tau\\
=&\int^{t}_{0}\int_{-\infty}^{+\infty}(1+\tau)^{\nu(\varepsilon)-1}
\left[B_*-M_0A_*-M_1\left(\rho^{\gamma+1}\right)_*-M_2(m^2)_*\right]dxd\tau\\
&+\int^{t}_{0}\int_{-\infty}^{+\infty}(1+\tau)^{\nu(\varepsilon)-1}\left[M_0A_*+M_1\left(\rho^{\gamma+1}\right)_*+M_2(m^2)_*\right]dxd\tau\\
\leq& C+M_1\int^{t}_{0}\int_{-\infty}^{+\infty}(1+\tau)^{\nu(\varepsilon)-1}\left[\rho^{\gamma+1}-\bar{\rho}^{\gamma+1}-
(\gamma+1)\bar{\rho}^{\gamma}(\rho-\bar{\rho})\right]dxd\tau\\
&+M_2\int^{t}_{0}\int_{-\infty}^{+\infty}(1+\tau)^{\nu(\varepsilon)-1}\left(m^2-\bar{m}^2-2\bar{m}(m-\bar{m})\right)dxd\tau\\
\leq& C+C\int^{t}_{0}\int_{-\infty}^{+\infty}(1+\tau)^{\nu(\varepsilon)-1}
\left[\left(\rho^\gamma-\bar{\rho}^\gamma\right)(\rho-\bar{\rho})+\left(m-\bar{m}\right)^2\right]dxd\tau\\
\leq& C,
\end{align*}
where the last inequality is from Lemma \ref{lem3} and $y_t=m-\bar{m}$. Then, it follows from Lemma \ref{lem7} and \eqref{nuvar} that
\begin{equation}\label{vargamma+1}
\begin{aligned}
&(1+t)^{\nu(\varepsilon)} \int_{-\infty}^{+\infty} \eta_*(x, t)dx\\
\leq& C+C\int^{t}_{0}\int_{-\infty}^{+\infty}(1+\tau)^{\nu(\varepsilon)-1}B_*dxd\tau+C\int^{t}_{0}(1+\tau)^{\nu(\varepsilon)}(J_1+J_2+J_3)d\tau\\
\leq& C+C\int^{t}_{0}(1+\tau)^{\nu(\varepsilon)}(J_1+J_2+J_3)d\tau.
\end{aligned}
\end{equation}
Hence, it suffices to estimate
$$\int^{t}_{0}(1+\tau)^{\nu(\varepsilon)}\left(J_1+J_2+J_3\right)d\tau.$$

Using \eqref{J1} and
\begin{align*}
\frac{3\gamma+1}{\gamma+1}-\nu(\varepsilon)-1=\frac{2\gamma}{\gamma+1}-\frac{\gamma^2+2\gamma}{(\gamma+1)^2}+2\varepsilon=
\frac{\gamma^2}{(\gamma+1)^2}+2\varepsilon>0,
\end{align*}
we have
\begin{equation}\label{J1'}
\int^{t}_{0}(1+\tau)^{\nu(\varepsilon)}J_1d\tau\leq C \int^{t}_{0}\left[(1+t)^{\nu(\varepsilon)-\frac{3\gamma+1}{\gamma+1}}+C(1+t)^{\nu(\varepsilon)-1}\|y_t\|^2\right]d\tau \leq C,
\end{equation}
by Lemma \ref{lem3}. As for $J_2$, it follows from \eqref{J2} that
\begin{equation}\label{J2'}
\begin{aligned}
\int^{t}_{0}(1+\tau)^{\nu(\varepsilon)}J_2d\tau&\leq \int^{t}_{0}(1+\tau)^{\nu(\varepsilon)-\frac{5\gamma^2+6\gamma+3}{2(\gamma+1)^2}}||\rho-\bar{\rho}||_{L^{\gamma+1}}d\tau\\
&\leq \int^{t}_{0}(1+\tau)^{\left[\nu(\varepsilon)-\frac{5\gamma^2+6\gamma+3}{2(\gamma+1)^2}\right]\cdot\frac{\gamma+1}{\gamma}}d\tau+
\int^{t}_{0}||\rho-\bar{\rho}||^{\gamma+1}_{L^{\gamma+1}}d\tau
\leq C,
\end{aligned}
\end{equation}
due to Lemma \ref{lem2}, Lemma \ref{lem3} and
\begin{align*}
\frac{5\gamma^2+6\gamma+3}{2(\gamma+1)^2}-\frac{\gamma}{\gamma+1}-\nu(\varepsilon)&=\frac{3\gamma^2+4\gamma+3}{2(\gamma+1)^2}-
\frac{\gamma^2+2\gamma}{(\gamma+1)^2}+2\varepsilon\\
&=\frac{\gamma^2+3}{2(\gamma+1)^2}+2\varepsilon>0.
\end{align*}

For $J_3$, we note
$$
J_3=-\int_{-\infty}^{+\infty}(\eta_{\bar{m}})_x(P_*+Q_*)dx\leq \int_{-\infty}^{+\infty}(P_*+Q_*)\left(|\bar{u}|^{\frac{2}{\gamma-1}}+|\bar{u}|^2|\right),
$$
which implies that
\begin{equation}\label{J3'}
\begin{aligned}
\int^{t}_{0}(1+\tau)^{\nu(\varepsilon)}J_3d\tau&\leq C\int^{t}_{0}(1+\tau)^{\nu(\varepsilon)}\int_{-\infty}^{+\infty}(P_*+Q_*)\left(|\bar{u}|^{\frac{2}{\gamma-1}}+|\bar{u}|^2|
\bar{\rho}|^{2-\gamma}\right)dxd\tau\\
&=:J^{1}_{3}+J^{2}_{3}+J^{3}_{3}+J^{4}_{3}.
\end{aligned}
\end{equation}

Next, we estimate $J^{1}_{3},\dots,J^{4}_{3}$. By lemmas \ref{lem2} and \ref{lem3}, it holds that
\begin{equation}\label{J31}
\begin{aligned}
J^{1}_{3}&=C\int^{t}_{0}(1+\tau)^{\nu(\varepsilon)}\int_{-\infty}^{+\infty}P_*|\bar{u}|^{\frac{2}{\gamma-1}}dxd\tau\\
&\leq C\int^{t}_{0}(1+\tau)^{\nu(\varepsilon)}\int_{-\infty}^{+\infty}[(\rho^{\gamma}-\bar{\rho}^{\gamma})
(\rho-\bar{\rho})]^{\frac{\gamma}{\gamma+1}}|\bar{u}|^{\frac{2}{\gamma-1}}dxd\tau\\
&\leq C\int^{t}_{0}(1+\tau)^{\nu(\varepsilon)}\left(\int_{-\infty}^{+\infty}(\rho^{\gamma}-\bar{\rho}^{\gamma})
(\rho-\bar{\rho})dx\right)^{\frac{\gamma}{\gamma+1}}  \left(\int_{-\infty}^{+\infty}|\bar{u}|^{\frac{2\gamma+2}{\gamma-1}}dx\right)^{\frac{1}{\gamma+1}} d\tau\\
&\leq C\int^{t}_{0}(1+\tau)^{\left(\nu(\varepsilon)-\frac{2\gamma^2+\gamma+1}{(\gamma+1)^2(\gamma-1)}\right)(\gamma+1)}d\tau+ C\int^{t}_{0}\int_{-\infty}^{+\infty}(\rho^{\gamma}-\bar{\rho}^{\gamma})(\rho-\bar{\rho})  dxd\tau\\
&\leq C,
\end{aligned}
\end{equation}
due to
\begin{align*}
\frac{2\gamma^2+\gamma+1}{(\gamma+1)^2(\gamma-1)}-\nu(\varepsilon)-\frac{1}{\gamma+1}&=\frac{\gamma^2+\gamma+2}{(\gamma+1)^2(\gamma-1)}-
\frac{\gamma^2+2\gamma}{(\gamma+1)^2}+2\varepsilon\\
&=-\frac{\gamma^3-3\gamma-2}{(\gamma+1)^2(\gamma-1)}+2\varepsilon\\
&=-\frac{\gamma-2}{\gamma-1}+2\varepsilon>0,\quad \gamma\in(1,2].
\end{align*}
Similarly,
\begin{equation}\label{J32}
\begin{aligned}
J^{2}_{3}&=C\int^{t}_{0}(1+\tau)^{\nu(\varepsilon)}\int_{-\infty}^{+\infty}P_*|\bar{u}|^2|\bar{\rho}|^{2-\gamma}dxd\tau\\
&\leq C\int^{t}_{0}(1+\tau)^{\nu(\varepsilon)}\int_{-\infty}^{+\infty}[(\rho^{\gamma}-\bar{\rho}^{\gamma})
(\rho-\bar{\rho})]^{\frac{\gamma}{\gamma+1}}|\bar{u}|^2|\bar{\rho}|^{2-\gamma}dxd\tau\\
&\leq C\int^{t}_{0}(1+\tau)^{\nu(\varepsilon)}\left(\int_{-\infty}^{+\infty}(\rho^{\gamma}-\bar{\rho}^{\gamma})
(\rho-\bar{\rho})dx\right)^{\frac{\gamma}{\gamma+1}}  \left(\int_{-\infty}^{+\infty}\left(|\bar{u}|^2|\bar{\rho}|^{2-\gamma}\right)^{\gamma+1}dx\right)^{\frac{1}{\gamma+1}}d\tau \\
&\leq C\int^{t}_{0}(1+\tau)^{\left(\nu(\varepsilon)-\frac{(\gamma+2)(\gamma+1)-1}{(\gamma+1)^2}\right)(\gamma+1)}d\tau+ C\int^{t}_{0}\int_{-\infty}^{+\infty}(\rho^{\gamma}-\bar{\rho}^{\gamma})(\rho-\bar{\rho})  dxd\tau\\
&\leq C,
\end{aligned}
\end{equation}
where the last inequality is from Lemma \ref{lem3} and
\begin{align*}
\frac{(\gamma+2)(\gamma+1)-1}{(\gamma+1)^2}-\nu(\varepsilon)-\frac{1}{\gamma+1}=2\varepsilon>0.
\end{align*}

For $J_{3}^{3}$ and $J_{3}^{4}$, we observe
\begin{align*}
\frac{2\gamma}{\gamma^2-1}-\frac{\gamma+2}{\gamma+1}=\frac{2\gamma-(\gamma+2)(\gamma-1)}{\gamma^2-1}=
-\frac{(\gamma+1)(\gamma-2)}{\gamma^2-1}\ge0,\quad \gamma\in(1,2].
\end{align*}
Hence, it is clear that
\begin{equation}\label{J33}
\begin{aligned}
J_{3}^{3}+J_{3}^{4}=&C\int^{t}_{0}(1+\tau)^{\nu(\varepsilon)}\int_{-\infty}^{+\infty}Q_*\left(|\bar{u}|^{\frac{2}{\gamma-1}}+
|\bar{u}|^2|\bar{\rho}|^{2-\gamma}\right)dxd\tau\\
\leq& C\int^{t}_{0}(1+\tau)^{\nu(\varepsilon)}||Q_*||_{L^{1}}\left(|||\bar{u}|^{\frac{2}{\gamma-1}}||_{L^{\infty}}+
|||\bar{u}|^2|\bar{\rho}|^{2-\gamma}||_{L^{\infty}}\right)d\tau\\
\leq &C\int^{t}_{0}(1+\tau)^{\nu(\varepsilon)}\left[(1+t)^{-\frac{2\gamma}{\gamma^2-1}}+
(1+t)^{-\frac{\gamma+2}{\gamma+1}}\right]d\tau\\
\leq &C\int^{t}_{0}(1+\tau)^{\nu(\varepsilon)-\frac{\gamma+2}{\gamma+1}}\int_{-\infty}^{+\infty}Q_*dxd\tau\\
\leq &C,
\end{aligned}
\end{equation}
where the last inequality is from $\nu(\varepsilon)<1<\frac{\gamma+2}{\gamma+1}$ and \eqref{PQC}.

Then, Plugging \eqref{J31}-\eqref{J33} into \eqref{J3'} gives
$$
\int^{t}_{0}(1+\tau)^{\nu(\varepsilon)}J_3d\tau\leq C.
$$
Hence, it follows from \eqref{J1'},\eqref{J2'}, \eqref{vargamma+1} and \eqref{nu} that
\begin{equation}\label{nuvareta*}
 \int_{-\infty}^{+\infty} \eta_*(x, t)dx\leq C(1+t)^{-\nu(\varepsilon)}\leq C(1+t)^{-\frac{\gamma^2+2\gamma}{(\gamma+1)^2}+2 \varepsilon}.
\end{equation}

Now, we point out the relationship between $\eta_*$ and $(\rho^{\gamma+1})_{*}$ in the following lemma.

\begin{lemma}\label{lem610}
Let
$$
\eta=\rho\int^{1}_{-1}g\left(1-z^2\right)^{\lambda}dz,
$$
with $g(\xi)=\frac{(\gamma-1)^2}{2\gamma(\gamma+1)}|\xi|^{\frac{2\gamma}{\gamma-1}}$. Then, there exists a positive constant $C_0$ such that
$$
E:=\eta-C_0\rho^{\gamma+1}
$$
is convex in $(\rho, m)$ for all $\rho, m$, and $E_*\ge0$.
\end{lemma}

\begin{proof}
We can compute the Hessian matrix of $E$ after simple calculations:
\begin{align*}
E_{\rho\rho}=&\int^{1}_{-1}(\theta^2+\theta) \rho^{\theta-1} g' z\left(1-z^2\right)^\lambda dz  \\
&+\rho\int^{1}_{-1}\left(-\frac{m}{\rho^2}+z\theta \rho^{\theta-1}\right)^2 g''\left(1-z^2\right)^\lambda d z-C_0(\gamma+1)\gamma\rho^{\gamma-1}\\
=&\theta^2\rho^{2\theta-1}\int^{1}_{-1} g''\left(1-z^2\right)^{\lambda+1} dz \\
&+\rho\int^{1}_{-1}\left(-\frac{m}{\rho^2}+z\theta \rho^{\theta-1}\right)^2 g''\left(1-z^2\right)^\lambda d z-C_0(\gamma+1)\gamma\rho^{\gamma-1},\\
E_{mm}=&\eta_{mm}=\frac{1}{\rho} \int^{1}_{-1} g''\left(1-z^2\right)^\lambda d z,\\
E_{\rho m}=&\eta_{\rho m}=\int^{1}_{-1}\left(-\frac{m}{\rho^2}+z\theta\rho^{\theta-1}\right) g'' \left(1-z^2\right)^\lambda dz.
\end{align*}

By Taylor's theorem, it follows from $g''(\xi)=|\xi|^{\frac{2}{\gamma-1}}$ that
\begin{align*}
&\int^{1}_{-1} g''(u+z\rho^{\theta})\left(1-z^2\right)^{\lambda+1} dz\\
=&\int^{1}_{-1}\left[g''(z\rho^{\theta})+ug^{(3)}(z\rho^{\theta})+
u^2\int^{1}_{0}(1-s)g^{(4)}(su+z\rho^{\theta})ds\right]\left(1-z^2\right)^{\lambda+1} dz\\
\ge& \rho\int^{1}_{-1}|z|^{\frac{2}{\gamma-1}}\left(1-z^2\right)^{\lambda+1} dz,
\end{align*}
where we have used $g^{(4)}\ge0$ and the fact that
$$
\int^{1}_{-1} g^{(3)}(z\rho^{\theta})\left(1-z^2\right)^{\lambda+1} dz=0.
$$
Hence,
\begin{equation}\label{erhorho}
\begin{aligned}
E_{\rho\rho}=&\theta^2\rho^{2\theta-1}\int^{1}_{-1} g''\left(1-z^2\right)^{\lambda+1} dz \\
&+\rho\int^{1}_{-1}\left(-\frac{m}{\rho^2}+z\theta \rho^{\theta-1}\right)^2 g''\left(1-z^2\right)^\lambda d z-C_0(\gamma+1)\gamma\rho^{\gamma-1}\\
\ge &\theta^2\rho^{2\theta}\int^{1}_{-1}|z|^{\frac{2}{\gamma-1}}\left(1-z^2\right)^{\lambda+1} dz\\
&+\rho\int^{1}_{-1}\left(-\frac{m}{\rho^2}+z\theta \rho^{\theta-1}\right)^2 g''\left(1-z^2\right)^\lambda d z-C_0(\gamma+1)\gamma\rho^{\gamma-1}\\
\ge& \rho\int^{1}_{-1}\left(-\frac{m}{\rho^2}+z\theta \rho^{\theta-1}\right)^2 g''\left(1-z^2\right)^\lambda d z,
\end{aligned}
\end{equation}
provided that
$$
C_0\leq \frac{\theta^2\int^{1}_{-1}|z|^{\frac{2}{\gamma-1}}\left(1-z^2\right)^{\lambda+1} dz}{(\gamma+1)\gamma}.
$$

In addition, using \eqref{erhorho} and $E_{mm}=\eta_{mm}\ge0$, we obtain
\begin{align*}
E_{\rho\rho}E_{mm}-\left(E_{\rho m}\right)^2\ge&\int^{1}_{-1}\left(-\frac{m}{\rho^2}+z\theta \rho^{\theta-1}\right)^2 g''\left(1-z^2\right)^\lambda d z\cdot \int^{1}_{-1} g''\left(1-z^2\right)^\lambda d z\\
&-\left[\int^{1}_{-1}\left(-\frac{m}{\rho^2}+z\theta\rho^{\theta-1}\right) g'' \left(1-z^2\right)^\lambda dz\right]^2,
\end{align*}
which is nonnegative when $g''$ is nonnegative. Therefore, we get the conclusion.

\end{proof}

Finally, using \eqref{nuvareta*}, Lemma \ref{lem2} and Lemma \ref{lem610}, we obtain
\begin{align*}
\int^{+\infty}_{-\infty}|\rho-\bar{\rho}|^{\gamma+1}dx&\leq C\int^{+\infty}_{-\infty}\left(\rho^\gamma-\bar{\rho}^\gamma\right)(\rho-\bar{\rho})dx\\
&\le C\int^{+\infty}_{-\infty}\left[\rho^{\gamma+1}-\bar{\rho}^{\gamma+1}-(\gamma+1) \bar{\rho}^\gamma(\rho-\bar{\rho}) \right]dx\\
&\le C(1+t)^{-\frac{\gamma^2+2\gamma}{(\gamma+1)^2}+2 \varepsilon},
\end{align*}
and
\begin{align*}
\int_{-\infty}^{+\infty}|\rho-\bar{\rho}| d x &\leq\left(\int_{-\infty}^{+\infty} \bar{\rho}^{\gamma-1}|\rho-\bar{\rho}|^2 d x\right)^{\frac{1}{2}}\left(\int_{-\infty}^{+\infty} \bar{\rho}^{1-\gamma} d x\right)^{\frac{1}{2}}\\
&\leq C(1+t)^{-\frac{\gamma^2+2\gamma}{2(\gamma+1)^2}+ \varepsilon}(1+t)^{\frac{\gamma}{2(\gamma+1)}}\\
&\leq C(1+t)^{-\frac{\gamma}{2(\gamma+1)^2}+\varepsilon},
\end{align*}
and so justify \eqref{1rr+1'}.

\section{Appendix}
This appendix includes some technical results, which are used in the proof of Theorem \ref{thm1}-(ii). First, we have a basic inequality.
\begin{lemma}\label{lem62}
Let $x_1,x_2>0$ and $0<b\le1$. Then
$
(x_1+x_2)^{b}\leq x^b_1+x^b_2.
$
\end{lemma}
\begin{proof}
Since $x_2>0$, the inequality is equivalent to
$$
\left(\frac{x_1}{x_2}+1\right)^{b}\leq \left(\frac{x_1}{x_2}\right)^b+1.
$$
Then, set $x=\frac{x_1}{x_2}>0$ and
$
f_1(x)=(x+1)^{b}-x^b-1.
$
It is easy to see that $f_1(0)=0$ and
$$
f'_1(x)=b(x+1)^{b-1}-bx^{b-1}=b\left[\left(\frac{1}{1+x}\right)^{1-b}-\left(\frac{1}{x}\right)^{1-b}\right]\le0,
$$
due to $0<b\le1$. Hence, for $x>0$,
$
f_1(x)\leq f_1(0)=0,
$
which ends the proof.
\end{proof}

The next lemma indicates some properties of a function, which is frequently used in our reasoning.
\begin{lemma}\label{lem5}
Let $1<\gamma\leq 2$,
$
\lambda=\frac{3-\gamma}{2(\gamma-1)},
$
and
$$
k(a)=\int^{1}_{-1} |a+z|^{\frac{2}{\gamma-1}}\left(1-z^2\right)^{\lambda+1}dz-\int^{1}_{-1} |a+z|^{\frac{2}{\gamma-1}}\left(1-z^2\right)^{\lambda}z^2dz.
$$
Then
$
k(a)=k(-a),
$
and $k(a)\ge0$.
\end{lemma}

\begin{proof}

From the definition of $k(a)$, we have
\begin{align*}
k(a)=&\int^{1}_{-1} \left|a+z\right|^{\frac{2}{\gamma-1}}\left(1-z^2\right)^{\lambda+1}dz-\int^{1}_{-1}\left|a+z\right|^{\frac{2}{\gamma-1}}
\left(1-z^2\right)^{\lambda}z^2dz\\
=&\int^{0}_{-1} \left|a+z\right|^{\frac{2}{\gamma-1}}\left(1-z^2\right)^{\lambda+1}dz+\int^{1}_{0} \left|a+z\right|^{\frac{2}{\gamma-1}}\left(1-z^2\right)^{\lambda+1}dz\\
&-\left[\int^{0}_{-1}\left|a+z\right|^{\frac{2}{\gamma-1}}\left(1-z^2\right)^{\lambda}z^2dz+\int^{1}_{0}\left|a+z\right|^{\frac{2}{\gamma-1}}
\left(1-z^2\right)^{\lambda}z^2dz\right]\\
=&\int^{1}_{0}\left(\left|a+z\right|^{\frac{2}{\gamma-1}}+\left|a-z\right|^{\frac{2}{\gamma-1}}\right)\left(1-z^2\right)^{\lambda}
\left(1-2z^2\right)dz,
\end{align*}
which infers
\begin{align*}
k(-a)&=\int^{1}_{0}\left(\left|-a+z\right|^{\frac{2}{\gamma-1}}+\left|-a-z\right|^{\frac{2}{\gamma-1}}\right)\left(1-z^2\right)^{\lambda}
\left(1-2z^2\right)dz\\
&=\int^{1}_{0}\left(\left|a-z\right|^{\frac{2}{\gamma-1}}+\left|a+z\right|^{\frac{2}{\gamma-1}}\right)\left(1-z^2\right)^{\lambda}
\left(1-2z^2\right)dz=k(a).
\end{align*}
Hence, we just need to prove $k(a)\ge0$ for $a\ge0$. To this end, integrating by parts we derive
\begin{align*}
k(a)=&\int^{1}_{-1} \left|a+z\right|^{\frac{2}{\gamma-1}}\left(1-z^2\right)^{\lambda+1}dz-\int^{1}_{-1}\left|a+z\right|^{\frac{2}{\gamma-1}}
\left(1-z^2\right)^{\lambda}z^2dz\\
=&\int^{1}_{-1} \left|a+z\right|^{\frac{2}{\gamma-1}}\left(1-z^2\right)^{\lambda+1}dz-\frac{1}{2(\lambda+1)}\int ^{1}_{-1} \left|a+z\right|^{\frac{2}{\gamma-1}}\left(1-z^2\right)^{\lambda+1}dz\\
&-\frac{1}{2(\lambda+1)}\frac{2}{\gamma-1}\int^{1}_{-1} \left|a+z\right|^{\frac{2}{\gamma-1}}\frac{z}{a+z}\left(1-z^2\right)^{\lambda+1}dz\\
=&\frac{2}{\gamma+1}\int^{1}_{-1} \left|a+z\right|^{\frac{2}{\gamma-1}}\left(1-z^2\right)^{\lambda+1}dz-\frac{2}{\gamma+1}\int^{1}_{-1} \left|a+z\right|^{\frac{2}{\gamma-1}}\frac{z}{a+z}\left(1-z^2\right)^{\lambda+1}dz\\
=&\frac{2a}{\gamma+1}\int^{1}_{-1} \frac{\left|a+z\right|^{\frac{2}{\gamma-1}}}{a+z}\left(1-z^2\right)^{\lambda+1}dz.
\end{align*}
If $a\ge1$, we have $k(a)>0$ since $z\ge-1$. If $a=0$, then $k(a)=0$. If $0< a<1$, we deduce
\begin{equation}\label{ka}
\begin{aligned}
k(a)&=\frac{2a}{\gamma+1}\int^{1}_{-1} \frac{\left|a+z\right|^{\frac{2}{\gamma-1}}}{a+z}\left(1-z^2\right)^{\lambda+1}dz\\
&=\frac{2a}{\gamma+1}\left[\int^{-a}_{-1} \frac{\left|a+z\right|^{\frac{2}{\gamma-1}}}{a+z}\left(1-z^2\right)^{\lambda+1}dz+\int ^{1}_{-a} \left(a+z\right)^{\frac{3-\gamma}{\gamma-1}}\left(1-z^2\right)^{\lambda+1}dz\right]\\
&=\frac{2a}{\gamma+1}\left[\int^{0}_{a-1}\frac{|\nu|^{\frac{2}{\gamma-1}}}{\nu}\left(1-(v-a)^2\right)^{\lambda+1}d\nu+
\int^{1+a}_{0}\nu^{\frac{3-\gamma}{\gamma-1}}\left(1-(v-a)^2\right)^{\lambda+1}d\nu\right]\\
&=\frac{2a}{\gamma+1}\left[-\int^{1-a}_{0}\nu^{\frac{3-\gamma}{\gamma-1}}\left(1-(v+a)^2\right)^{\lambda+1}d\nu+
\int^{1+a}_{0}\nu^{\frac{3-\gamma}{\gamma-1}}\left(1-(v-a)^2\right)^{\lambda+1}d\nu\right]\\
&\ge0,
\end{aligned}
\end{equation}
where the last inequality is from
$$
1-(v-a)^2\ge1-(v+a)^2,\quad \nu\in[0,1-a].
$$
Hence, the proof is complete.
\end{proof}

Next, we state some properties of Beta function and Taylor's theorem.
\begin{lemma}\label{lem63}
Let $p,q>0$, and
$
B(p,q)=\int^{1}_{0}(1-x)^{p-1}x^{q-1}dx.
$
Then, for $p>0$ and $q>1$, it holds that
$$
B(p,q)=B(q,p)=\frac{q-1}{p+q-1}B(p,q-1).
$$
\end{lemma}

\begin{lemma}\label{lem64}
Let $f(\xi)=|\xi|^{\nu}$. Then, for any $0\le n\le \nu, n\in N$, it holds that
$$
f(u+z)-f(z)-f'(z)u-\cdot\cdot\cdot-\frac{f^{(n)}(z)}{n!}u^n=u^{n+1}\int^{1}_{0}\frac{(1-s)^{n}}{n!}f^{(n+1)}(su+z)ds.
$$
\end{lemma}

Finally, we present the essentially important lemma, which is used to prove the convexity of $A$ defined in \eqref{A*B*}.
\begin{lemma}\label{lem65}
Let $a\in(0,1)$ and $k\ge4, k\in Z^+$. For $n\ge k$, write
\begin{equation}\label{Ckn}
C^{k}_{n}=\frac{n\cdot(n-1)\cdot\cdot\cdot(n-k+1)}{k!},
\end{equation}
and we require $C^{0}_{n}=1$ for the sake of formal consistency.

\noindent(1)~If $n\in(2k-1,2k]$,
\begin{align*}
F_1(a)=&\left[\Sigma^{k-1}_{j=0}(2j+2)C^{2j}_{n}B\left(\frac{n+1}{2},\frac{n-2j+1}{2}\right)a^{2j}\right] \\ &\times\left[\Sigma^{k-1}_{j=0}2jC^{2j}_{n}B\left(\frac{n+1}{2},\frac{n-2j+1}{2}\right)a^{2j}\right],
\end{align*}
and
\begin{align*}
F_2(a)=&\bigg[\Sigma^{k-2}_{j=0}(2j+2)C^{2j+1}_{n}B\left(\frac{n+1}{2},\frac{n-2j+1}{2}\right)a^{2j+1}\\
&+a^{2k-1}(n+3)C^{2k-1}_{n}B\left(\frac{n+1}{2},\frac{n-2k+3}{2}\right)\bigg]^2.
\end{align*}
Then, $f(a):=F_1(a)- F_{2}(a)\ge0$;

\noindent(2)~If $n\in(2k,2k+1]$,
\begin{align*}
\tilde{F}_1(a)=&\left[\Sigma^{k}_{j=0}(2j+2)C^{2j}_{n}B\left(\frac{n+1}{2},\frac{n-2j+1}{2}\right)a^{2j}\right] \\ &\times\left[\Sigma^{k}_{j=0}2jC^{2j}_{n}B\left(\frac{n+1}{2},\frac{n-2j+1}{2}\right)a^{2j}\right],
\end{align*}
and
\begin{align*}
\tilde{F}_2(a)=&\bigg[\Sigma^{k-1}_{j=0}(2j+2)C^{2j+1}_{n}B\left(\frac{n+1}{2},\frac{n-2j+1}{2}\right)a^{2j+1}\\
&+(n+3)a^{2k}C^{2k}_{n}B\left(\frac{n+1}{2},\frac{n-2k+1}{2}\right)\bigg]^2.
\end{align*}
Then, $\tilde{f}(a):=\tilde{F}_1(a)- \tilde{F}_{2}(a)\ge0$.
\end{lemma}
\begin{proof}

\noindent {\bf (1)}. Assume $n\in(2k-1,2k]$.

 Clearly, $f(a)$ is a polynomial containing only even powers of $a$. Hence, we compare the coefficients of $a^{2i}$ ($1\le i\le 2k-2$). For the function
$$
F(x)=\left(\sum^{m}_{k=0}a_{k}x^{k}\right)\cdot\left(\sum^{l}_{k=0}b_{k}x^{k}\right),
$$
where $m,l\in Z^+$, and $a_{k} ~(k=0,1,2,\cdot\cdot\cdot m),~ b_{2}~ (k=0,1,2,\cdot\cdot\cdot l)$ are real constants, for convenience we denote
$$
[k_1:k_2]_{F}=a_{k_1}b_{k_2},
$$
the coefficient of the term $x^{k_1}\cdot x^{k_2}$ in $F(x)$. And we consider separately four cases.

\vspace{4pt}
\noindent\textbf{$Case ~I: i=1$}.
The coefficient of $a^2$ is
\begin{align*}
2B\left(\frac{n+1}{2},\frac{n+1}{2}\right)2C^{2}_{n}B\left(\frac{n+1}{2},\frac{n-1}{2}\right)-
\left[2C^{1}_{n}B\left(\frac{n+1}{2},\frac{n+1}{2}\right)\right]^2=0,
\end{align*}
due to
$$
B\left(\frac{n+1}{2},\frac{n-1}{2}\right)=\frac{2n}{n-1}\left(\frac{n+1}{2},\frac{n+1}{2}\right),
$$
by Lemma \ref{lem63}.

\vspace{4pt}
\noindent\textbf{$Case ~II: 2\leq i<k$}. We see that the coefficient of the term $a^{2i}$ in $F_1(a)$ is
\begin{align*}
&[0:2i]_{F_1}+[2:2i-2]_{F_1}+\cdot\cdot\cdot+[2i-4:4]_{F_1}+[2i-2:2]_{F_1}\\
=&\sum^{i-1}_{j=0}\left[(2j+2)C^{2j}_{n}B\left(\frac{n+1}{2},\frac{n-2j+1}{2}\right)(2i-2j)C^{2i-2j}_{n}
B\left(\frac{n+1}{2},\frac{n-2i+2j+1}{2}\right)\right],
\end{align*}
and the coefficient of $a^{2i}$ in $F_2(a)$ is
\begin{align*}
&[1:2i-1]_{F_2}+[3:2i-3]_{F_2}+\cdot\cdot\cdot+[2i-3:3]_{F_2}+[2i-1:1]_{F_2}\\
=&\sum^{i-1}_{j=0}\bigg[(2j+2)C^{2j+1}_{n}B\left(\frac{n+1}{2},\frac{n-2j+1}{2}\right)\\
&\times(2i-2j)C^{2i-2j-1}_{n}
B\left(\frac{n+1}{2},\frac{n-2i+2j+3}{2}\right)\bigg].
\end{align*}

Next, we compare $[2j:2i-2j]_{F_1}+[2i-2j-2:2j+2]_{F_1}$ with $[2j+1:2i-2j-1]_{F_2}+[2i-2j-1:2j+1]_{F_2}$ for
\begin{equation}\label{2j<i-1}
0\le 2j\leq i-1,
\end{equation}
by $2j\le2i-2j-2$. In fact,
\begin{align*}
&[2j:2i-2j]_{F_1}+[2i-2j-2:2j+2]_{F_1}\\
=&(2j+2)C^{2j}_{n}B\left(\frac{n+1}{2},\frac{n-2j+1}{2}\right)(2i-2j)C^{2i-2j}_{n}
B\left(\frac{n+1}{2},\frac{n-2i+2j+1}{2}\right)\\
&+(2i-2j)C^{2i-2j-2}_{n}B\left(\frac{n+1}{2},\frac{n-2i+2j+3}{2}\right)(2j+2)C^{2j+2}_{n}
B\left(\frac{n+1}{2},\frac{n-2j-1}{2}\right).
\end{align*}
From Lemma \ref{lem63}, we have
$$
B\left(\frac{n+1}{2},\frac{n-2i+2j+1}{2}\right)=\frac{2(n-i+j+1)}{n-2i+2j+1}B\left(\frac{n+1}{2},\frac{n-2i+2j+3}{2}\right),
$$
and
$$
B\left(\frac{n+1}{2},\frac{n-2j-1}{2}\right)=\frac{2(n-j)}{n-2j-1}B\left(\frac{n+1}{2},\frac{n-2j+1}{2}\right),
$$
which infers
\begin{equation}\label{2jF1}
\begin{aligned}
&\frac{1}{(2j+2)(2i-2j)}\left([2j:2i-2j]_{F_1}+[2i-2j-2:2j+2]_{F_1}\right)\\
=&\left[C^{2j}_{n}C^{2i-2j}_{n}\frac{2(n-i+j+1)}{n-2i+2j+1}+C^{2i-2j-2}_{n}C^{2j+2}_{n}\frac{2(n-j)}{n-2j-1}\right]\\
&\times B\left(\frac{n+1}{2},\frac{n-2j+1}{2}\right)B\left(\frac{n+1}{2},\frac{n-2i+2j+3}{2}\right).
\end{aligned}
\end{equation}

As for $[2j+1:2i-2j-1]_{F_2}+[2i-2j-1:2j+1]_{F_2}$, since
$$[2j+1:2i-2j-1]_{F_2}=[2i-2j-1:2j+1]_{F_2},$$
it is easy to see
\begin{align*}
2[2j+1:2i-2j-1]_{F_2}
=&2(2j+2)C^{2j+1}_{n}B\left(\frac{n+1}{2},\frac{n-2j+1}{2}\right)\\
&\times(2i-2j)C^{2i-2j-1}_{n}B\left(\frac{n+1}{2},\frac{n-2i+2j+3}{2}\right).
\end{align*}
Combining this and \eqref{2jF1}, we deduce
\begin{equation}\label{fij}
\begin{aligned}
&\frac{[2j:2i-2j]_{F_1}+[2i-2j-2:2j+2]_{F_1}-2[2j+1:2i-2j-1]_{F_2}}{(2j+2)(2i-2j)
B\left(\frac{n+1}{2},\frac{n-2j+1}{2}\right)B\left(\frac{n+1}{2},\frac{n-2i+2j+3}{2}\right)}\\
=&C^{2j}_{n}C^{2i-2j}_{n}\frac{2(n-i+j+1)}{n-2i+2j+1}+C^{2i-2j-2}_{n}C^{2j+2}_{n}\frac{2(n-j)}{n-2j-1}-2C^{2j+1}_{n}C^{2i-2j-1}_{n}\\
=&C^{2j+1}_{n}C^{2i-2j-1}_{n}\left(\frac{2j+1}{n-2j}\cdot\frac{2(n-i+j+1)}{2i-2j}+\frac{2i-2j-1}{n-2i+2j+2}\cdot\frac{2n-2j}{2j+2}-2\right),
\end{aligned}
\end{equation}
where the last equality is from
$$
C^{2j}_{n}C^{2i-2j}_{n}=\frac{2j+1}{n-2j}C^{2j+1}_{n}\frac{n-2i+2j+1}{2i-2j}C^{2i-2j-1}_{n},
$$
and
$$
C^{2i-2j-2}_{n}C^{2j+2}_{n}=\frac{2i-2j-1}{n-2i+2j+2}C^{2i-2j-1}_{n}\frac{n-2j-1}{2j+2}C^{2j+1}_{n},
$$
due to \eqref{Ckn}. Set
$$x=2i-2j\in[i+1,2i]$$
(by \eqref{2j<i-1}), and
\begin{equation}\label{fxjn}
\begin{aligned}
f(x,j,n)=&(2j+1)(2n-x+2)(n-x+2)(2j+2)+(x-1)(2n-2j)(n-2j)x\\
&-2(n-2j)x(n-x+2)(2j+2)\\
=&\left[(2j+1)(2n-x+2)-(n-2j)x\right](n-x+2)(2j+2)\\
&+(n-2j)x\left[(x-1)(2n-2j)-(n-x+2)(2j+2)\right]\\
=&(4j+2-x)(n+1)(n-x+2)(2j+2)+(2x-2j-4)(n+1)(n-2j)x\\
=&2(n+1)\left[(4j+2-x)(n-x+2)(j+1)+(x-j-2)(n-2j)x\right]\ge0,
\end{aligned}
\end{equation}
where the last inequality is from
\begin{equation}\label{4j+2-x}
\begin{aligned}
&(4j+2-x)(n-x+2)(j+1)+(x-j-2)(n-2j)x\\
\ge&(x-j-2)\left[(n-2j)x-(n-x+2)(j+1)\right]\\
=&(x-j-2)\left[(n-j+1)x-(n+2)(j+1)\right]\\
\ge&(x-j-2)\left[\left(n-\frac{i-1}{2}+1\right)(i+1)-(n+2)\frac{i+1}{2}\right]\\
\ge&\frac{(x-j-2)(i+1)}{2}\left(n-i+1\right)\ge0,
\end{aligned}
\end{equation}
by \eqref{2j<i-1}, $n\in(2k-1,2k]$ and $i<k$. Hence, we conclude that the right side of \eqref{fij} is large than 0, and so is the coefficient of $a^{2i}$ of $f(a)$.

We point out that if $2j=i-1$, then
$$[2j:2i-2j]_{F_1}+[2i-2j-2:2j+2]_{F_1}=[i-1:i+1]_{F_1}+[i-1:i+1]_{F_1}.$$
From the previous proof process, we know that
$$[i-1:i+1]_{F_1}+[i-1:i+1]_{F_1}\ge [i:i+2]_{F_2}+[i:i+2]_{F_2},$$
which completes the proof for $i<k$.

\vspace{4pt}
\noindent\textbf{$Case~ III: k\leq i\le2k-3$}.
We use $i^*$ instead of $i$ in Case II. The coefficient of $a^{2i^*}$ in $F_1(a)$ is
\begin{align*}
[2i^*-2k+2:2k-2]_{F_1}+[2i^*-2k+4:2k-4]_{F_1}+\cdot\cdot\cdot+[2k-2:2i^*-2k+2]_{F_1},
\end{align*}
and the coefficient of $a^{2i^*}$ in $F_2(a)$ is
\begin{align*}
[2i^*-2k+1:2k-1]_{F_2}+[2i^*-2k+3:2k-3]_{F_1}+\cdot\cdot\cdot+[2k-1:2i^*-2k+1]_{F_1}.
\end{align*}
It is obvious that the coefficient of $a^{2i^*}$ in $F_2(a)$ has one more item than $F_1(a)$, which we will consider later.

First, we compare
$$
[2i^*-2k+2j^*:2k-2j^*]_{F_1}+[2k-2j^*-2:2i^*-2k+2j^*+2]_{F_1}
$$
with
$$
[2i^*-2k+2j^*+1:2k-2j^*-1]_{F_2}+[2k-2j^*-1:2i^*-2k+2j^*+1]_{F_2}
$$
for
\begin{equation}\label{2j<2k-i}
2\le2j^*\leq 2k-i^*-1,
\end{equation}
by $2i^*-2k+2j^*\le 2k-2j^*-2$.

Fortunately, if we choose
\begin{eqnarray*}
\left\{\begin{array}{l}
i=i^*\\
j=i^*+j^*-k\\
x=2i-2j=2k-2j^*,
\end{array}\right.
\end{eqnarray*}
in Case II, then the coefficient of $a^{2i}$ in $f(a)$ will become the coefficient of $a^{2i^*}$ in $f(a)$, which means that $f(x,j,n)$ defined in \eqref{fxjn} becomes
\begin{equation}\label{fnkx*}
\begin{aligned}
&\frac{f(x^*,j^*,n)}{2(n+1)}\\
=&(4j+2-x)(n-x+2)(j+1)+(x-j-2)(n-2j)x\\
\ge&(x-j-2)\left[(n-2j)x-(n-x+2)(j+1)\right]\\
\ge&\frac{(x-j-2)(i+1)}{2}\left(n-i+1\right)\\
=&\frac{(3k-3j^*-i^*-2)(i^*+1)}{2}(n-i^*+1)\ge0,
\end{aligned}
\end{equation}
where the last inequality is from
$$
3k-3j^*-i^*-2\ge \frac{i^*-1}{2}\ge0
$$
by \eqref{2j<2k-i} and $i\ge k\ge4$. Similarly, for $2j^*=2k-i^*-1$, we have
\begin{align*}
[i^*-1:i^*+1]_{F_1}\ge[i^*:i^*+2]_{F_2}.
\end{align*}

Next, we pay attention to the additional item $[2i^*-2k:2k]_{F_2}$ in $F_{2}(a)$. Since we choose $2j^*\ge2$ in the previous proof, which means that we do not consider
$$[2k-2:2i^*-2k+2]_{F_1}$$
and
$$[2i^*-2k+1:2k-1]_{F_2}+[2k-1:2i^*-2k+1]_{F_2}.$$
It is obvious from Lemma \ref{lem63} that
\begin{align*}
&[2k-2:2i^*-2k+2]_{F_1}\\
=&2kC^{2k-2}_{n}B\left(\frac{n+1}{2},\frac{n-2k+3}{2}\right)(2i^*-2k+2)C^{2i^*-2k+2}_{n}B\left(\frac{n+1}{2},\frac{n-2i^*+2k-1}{2}\right)\\
=&2kC^{2k-2}_{n}B\left(\frac{n+1}{2},\frac{n-2k+3}{2}\right)(2i^*-2k+2)C^{2i^*-2k+2}_{n}\\
&\times \frac{2(n-i^*+k)}{n-2i^*+2k-1}B\left(\frac{n+1}{2},\frac{n-2i^*+2k+1}{2}\right),
\end{align*}
and
\begin{align*}
2[2i^*-2k+1:2k-1]_{F_2}=&2(2i^*-2k+2)C^{2i^*-2k+1}_{n}B\left(\frac{n+1}{2},\frac{n-2i^*+2k+1}{2}\right)\\
&\times(n+3)C^{2k-1}_{n}B\left(\frac{n+1}{2},\frac{n-2k+3}{2}\right).
\end{align*}
As a result, having in mind $i^*-k\in[0, k-3]$ and $n\in(2k-1,2k]$, we deduce that
\begin{align*}
&\frac{[2k-2:2i^*-2k+2]_{F_1}-[2i^*-2k+1:2k-1]_{F_2}-[2k-1:2i^*-2k+1]_{F_2}}{2(2i^*-2k+2)B\left(\frac{n+1}{2},\frac{n-2k+3}{2}\right)
B\left(\frac{n+1}{2},\frac{n-2i^*+2k+1}{2}\right)}\\
=&kC^{2k-2}_{n}C^{2i^*-2k+2}_{n}\frac{2(n-i^*+k)}{n-2i^*+2k-1}-(n+3)C^{2i^*-2k+1}_{n}C^{2k-1}_{n}\\
=&\frac{C^{2k-2}_{n}C^{2i^*-2k+2}_{n}}{n-2i^*+2k-1}\left[2k(n-i^*+k)-(n+3)(2i^*-2k+2)\frac{n-2k+2}{2k-1}\right]\\
\ge&\frac{C^{2k-2}_{n}C^{2i^*-2k+2}_{n}}{n-2i^*+2k-1}\left[2k(n-k+3)-(2k+3)(2k-4)\frac{n-2k+2}{2k-1}\right]\\
=&\frac{2C^{2k-2}_{n}C^{2i^*-2k+2}_{n}}{n-2i^*+2k-1}\cdot\frac{6(n-2k+2)+k(k+1)(2k-1)}{2k-1}\ge0.
\end{align*}
Hence, we complete the proof for this case.

\vspace{4pt}
\noindent\textbf{$Case~ IV: i=2k-2$}.
We know that there is still one remaining item for $F_1(a)$, that is $[2k-2:2k-2]_{F_1}a^{4k-2}$, and
\begin{align*}
[2k-2:2k-2]_{F_1}=2k(2k-2)\left[C^{2k-2}_{n}B\left(\frac{n+1}{2},\frac{n-2k+3}{2}\right)\right]^2.
\end{align*}
And there are still three remaining items for $F_2(a)$, that is
$$\left([2k-3:2k-1]_{F_2}+[2k-1:2k-3]_{F_2}\right)a^{4k-4}+[2k-1:2k-1]_{F_2}a^{4k-2}.$$
Since $0<a\le1$, we derive that
$$
[2k-1:2k-1]_{F_2}a^{4k-2}\leq [2k-1:2k-1]_{F_2}a^{4k-4},
$$
and
\begin{align*}
&2[2k-3:2k-1]_{F_2}+[2k-1:2k-1]_{F_2}\\
=&2(2k-2)C^{2k-3}_{n}B\left(\frac{n+1}{2},\frac{n-2k+5}{2}\right)\cdot(n+3)C^{2k-1}_{n}B\left(\frac{n+1}{2},\frac{n-2k+3}{2}\right)\\
&+\left[(n+3)C^{2k-1}_{n}B\left(\frac{n+1}{2},\frac{n-2k+3}{2}\right)\right]^2\\
=&(2k-2)(n+3)\frac{(2k-2)(n-2k+2)}{(n-k+2)(2k-1)}\left[C^{2k-2}_{n}B\left(\frac{n+1}{2},\frac{n-2k+3}{2}\right)\right]^2\\
&+\left[\frac{(n+3)(n-2k+2)}{2k-1}\right]^{2}\left[C^{2k-2}_{n}B\left(\frac{n+1}{2},\frac{n-2k+3}{2}\right)\right]^2,
\end{align*}
which infers
\begin{align*}
&\frac{[2k-2:2k-2]_{F_1}-(2[2k-3:2k-1]_{F_2}+[2k-1:2k-1]_{F_2})}
{\left[C^{2k-2}_{n}B\left(\frac{n+1}{2},\frac{n-2k+3}{2}\right)\right]^2}\\
=&2k(2k-2)-(2k-2)(n+3)\frac{(2k-2)(n-2k+2)}{(n-k+2)(2k-1)}-\left[\frac{(n+3)(n-2k+2)}{2k-1}\right]^{2}\\
\ge&2k(2k-2)-2(2k-2)(2k+3)\frac{2k-2}{(k+2)(2k-1)}-\left[\frac{2(2k+3)}{2k-1}\right]^{2}\\
\ge&2k(2k-2)-2(2k+3)\frac{2k-2}{k+2}-4\left[\frac{2k+3}{2k-1}\right]^{2}\\
=&4\left[\frac{(k^2-3)(k-1)}{k+2}-\left(\frac{2k+3}{2k-1}\right)^2\right]\\
=:&4(h_1(k)-h_2(k)).
\end{align*}
Noticing that
\begin{align*}
h_1(k)=\frac{(k^2-3)(k-1)}{k+2}=(k+2)^2-7(k+2)+13-\frac{3}{k+2}\leq h_1(4)=\frac{39}{6},
\end{align*}
and
$$
h_2(k)=\left(\frac{2k+3}{2k-1}\right)^2=\left[1+\frac{4}{2k-1}\right]^2\ge h_2(4)=\frac{121}{49},
$$
by $k\ge4$, we find that $h_1(k)-h_2(k)>0$. Consequently, we get the conclusion for $i=2k-2$, and so ends the
proof of Assertion (1).

\vspace{4pt}
\noindent {\bf (2)}. Assume $n\in(2k,2k+1]$.

We first consider $k\ge$5, since $k=4$ is quiet different from $k\ge5$. From the definition of $f(a)$ and $\tilde{f}(a)$, we see that the coefficients of $a^{2i}$ (with $1\le i\le k-1$) are the same. In addition, the functions in \eqref{fnkx} and \eqref{4j+2-x} are also larger than zero for $n\in(2k,2k+1]$. Hence, we distinguish three cases.

\vspace{4pt}
\noindent\textbf{$Case~1:k+1\le i\le 2k-2$}.
The coefficient of $a^{2i}$ in $\tilde{F}_1(a)$ is
\begin{align*}
[2i-2k:2k]_{\tilde{F}_1}+[2i-2k+2:2k-2]_{\tilde{F}_1}+\cdot\cdot\cdot+[2k:2i-2k]_{\tilde{F}_1},
\end{align*}
and the coefficient of $a^{2i}$ in $\tilde{F}_2(a)$ is
\begin{align*}
[2i-2k+1:2k-1]_{\tilde{F}_2}+[2i-2k+3:2k-3]_{\tilde{F}_2}+\cdot\cdot\cdot+[2k-1:2i-2k+1]_{\tilde{F}_2}.
\end{align*}
From the proof process for Case III with $j^*=1$, we see that the function in \eqref{fnkx*} is still larger than 0 (with $n\in(2k,2k+1]$). That is,
\begin{align*}
&[2i-2k+2:2k-2]_{\tilde{F}_1}+[2i-2k+4:2k-4]_{\tilde{F}_1}+\cdot\cdot\cdot+[2k-4:2i-2k+4]_{\tilde{F}_1}\\
\ge&[2i-2k+3:2k-3]_{\tilde{F}_2}+[2i-2k+5:2k-5]_{\tilde{F}_2}+\cdot\cdot\cdot+[2k-3:2i-2k+3]_{\tilde{F}_2},
\end{align*}
which means that there are only three remaining items in the coefficient of $a^{2i}$ for $\tilde{F}_1(a)$:
$$
[2i-2k:2k]_{\tilde{F}_1}+[2k-2:2i-2k+2]_{\tilde{F}_1}+[2k:2i-2k]_{\tilde{F}_1},
$$
and only two remaining items in the coefficient of $a^{2i}$ for $\tilde{F}_2(a)$
$$
[2i-2k+1:2k-1]_{\tilde{F}_2}+[2k-1:2i-2k+1]_{\tilde{F}_2}.
$$
Then, we consider $a^{2k}$ in $\tilde{F}_2(a)$. Since $0<a\le1$, we have
\begin{align*}
2[2i-2k+1:2k]_{\tilde{F}_2}a^{2i+1}
\leq 2[2i-2k+1:2k]_{\tilde{F}_2}a^{2i},
\end{align*}

From Lemma \ref{lem63}, it holds that
\begin{equation}\label{k>5F1}
\begin{aligned}
&[2i-2k:2k]_{\tilde{F}_1}+[2k-2:2i-2k+2]_{\tilde{F}_1}+[2k:2i-2k]_{\tilde{F}_1}\\
=&(2i-2k+2)C^{2i-2k}_{n}B\left(\frac{n+1}{2},\frac{n-2i+2k+1}{2}\right)2kC^{2k}_{n}B\left(\frac{n+1}{2},\frac{n-2k+1}{2}\right)\\
&+2kC^{2k-2}_{n}B\left(\frac{n+1}{2},\frac{n-2k+3}{2}\right)(2i-2k+2)C^{2i-2k+2}_{n}B\left(\frac{n+1}{2},\frac{n-2i+2k-1}{2}\right)\\
&+(2k+2)C^{2k}_{n}B\left(\frac{n+1}{2},\frac{n-2k+1}{2}\right)(2i-2k)C^{2i-2k}_{n}
B\left(\frac{n+1}{2},\frac{n-2i+2k+1}{2}\right)\\
=&C^{2i-2k+1}_{n}B\left(\frac{n+1}{2},\frac{n-2i+2k+1}{2}\right)C^{2k-1}_{n}B\left(\frac{n+1}{2},\frac{n-2k+3}{2}\right)\\
&\times\bigg[2(n-k+1)(2i-2k+2)\frac{2i-2k+1}{n-2i+2k}+2k(2i-2k+2)\frac{2k-1}{n-2k+2}\frac{n-i+k}{i-k+1}\\
&+2(k+1)(2i-2k+2)\frac{n-k+1}{k}\frac{i-k}{i-k+1}\frac{2i-2k+1}{n-2i+2k}\bigg],
\end{aligned}
\end{equation}
and
\begin{equation}\label{k>5F2}
\begin{aligned}
&2[2i-2k+1:2k-1]_{\tilde{F}_2}+2[2i-2k+1:2k]_{\tilde{F}_2}\\
=&2(2i-2k+2)C^{2i-2k+1}_{n}B\left(\frac{n+1}{2},\frac{n-2i+2k+1}{2}\right)\\
&\times\left[2kC^{2k-1}_{n}B\left(\frac{n+1}{2},\frac{n-2k+3}{2}\right)+(n+3)C^{2k}_{n}B\left(\frac{n+1}{2},\frac{n-2k+1}{2}\right)\right]\\
=&2(2i-2k+2)C^{2i-2k+1}_{n}B\left(\frac{n+1}{2},\frac{n-2i+2k+1}{2}\right)\\
&\times\left[2k+(n+3)\frac{n-k+1}{k}\right]C^{2k-1}_{n}B\left(\frac{n+1}{2},\frac{n-2k+3}{2}\right),
\end{aligned}
\end{equation}
Hence, setting
\begin{equation}\label{i-k=x}
x=i-k\in[1,k-2],
\end{equation}
and
$$\bar{C}=(2i-2k+2)C^{2i-2k+1}_{n}B\left(\frac{n+1}{2},\frac{n-2i+2k+1}{2}\right)
C^{2k-1}_{n}B\left(\frac{n+1}{2},\frac{n-2k+3}{2}\right),$$
we deduce from \eqref{k>5F1} and \eqref{k>5F2} that
\begin{equation}\label{fnkx}
\begin{aligned}
&\frac{1}{\bar{C}}\bigg([2i-2k:2k]_{\tilde{F}_1}+[2k-2:2i-2k+2]_{\tilde{F}_1}+[2k:2i-2k]_{\tilde{F}_1}\\
&-2[2i-2k+1:2k-1]_{\tilde{F}_2}-2[2i-2k+1:2k]_{\tilde{F}_2}\bigg)\\
=&2k\bigg(\frac{2i-2k+1}{n-2i+2k}\frac{n-k+1}{k}+
\frac{2k-1}{n-2k+2}\frac{n-i+k}{i-k+1}\\
&+\frac{k+1}{k}\frac{i-k}{i-k+1}\frac{2i-2k+1}{n-2i+2k}\frac{n-k+1}{k}
-2\bigg)-2(n+3)\frac{n-k+1}{k}\\
=&2k\left(\frac{2x+1}{n-2x}\cdot\frac{n-k+1}{k}+\frac{2k-1}{n-2k+2}\cdot\frac{n-x}{x+1}
+\frac{k+1}{k}\cdot\frac{x}{x+1}\cdot\frac{2x+1}{n-2x}\cdot\frac{n-k+1}{k}-2\right)\\
&-2(n+3)\frac{n-k+1}{k}\\
\ge&2k\left(\frac{x+1}{n-x}\cdot\frac{n-k+1}{k}+\frac{2k-1}{n-2k+2}\cdot\frac{n-x}{x+1}-2\right)
-2(n+3)\frac{n-k+1}{k},
\end{aligned}
\end{equation}
due to \eqref{i-k=x}, $n\in(2k,2k+1]$ and
$
\frac{2x+1}{n-2x}\ge\frac{x+1}{n-x}.
$

Next, put
\begin{equation}\label{fnkx3}
\begin{aligned}
y=\frac{x+1}{n-x}, \quad f(y)=\frac{n-k+1}{k}y+\frac{2k-1}{n-2k+2}\frac{1}{y}-2.
\end{aligned}
\end{equation}
It is clear that $f$ is decreasing on $(0,\sqrt{\frac{2k-1}{n-2k+2}\cdot\frac{k}{n-k+1}})$. Observing that
$$
\frac{2k-1}{n-2k+2}\cdot\frac{k}{n-k+1}\ge \frac{2k-1}{k+2}\cdot\frac{k}{3}>1,
$$
by $k\ge5$, and
$$
y=\frac{x+1}{n-x}\le \frac{k-1}{n-k+2}\le\frac{k-1}{k+2}\le 1,
$$
by \eqref{i-k=x} and $n\in(2k,2k+1]$, we obtain
\begin{align*}
f(y)\ge f\left(\frac{k-1}{n-k+2}\right)&=\frac{n-k+1}{k}\cdot\frac{k-1}{n-k+2}+ \frac{2k-1}{n-2k+2} \cdot\frac{n-k+2}{k-1}-2\\
&\ge\frac{k-1}{k}\cdot\frac{k+1}{k+2}+\frac{2k-1}{k-1}\cdot\frac{k+3}{3}-2\\
&\ge \frac{2k-1}{k-1}\cdot\frac{k+3}{3}-\frac{5}{3},
\end{align*}
where the last inequality is from
$
\frac{k-1}{k}\cdot\frac{k+1}{k+2}\ge \frac{1}{3},
$
by $k\ge5$.  As a result, by $k\ge5$ and $n\in(2k,2k+1]$ again, it holds that
\begin{align*}
kf(y)-(n+3)\frac{n-k+1}{k}&\ge k\left(\frac{2k-1}{k-1}\cdot\frac{k+3}{3}-\frac{5}{3}\right)-(2k+4)\frac{k+2}{k}\\
&=k\cdot\frac{2k^2+2}{3(k-1)}-2\frac{k^2+4k+4}{k}\\
&=\frac{2}{3k(k-1)}\cdot(k^4-3k^3-8k^2+12)>0,
\end{align*}
and so the right side of \eqref{fnkx} is positive.

\vspace{4pt}
\noindent\textbf{$Case~2:k= i$}.
In this case, $x=i-k=0$, and
$$
y=\frac{2x+1}{n-2x}=\frac{1}{n}<\frac{k-1}{n-k+2},
$$
which implies that
$$
f(y)\bigg|_{y=\frac{1}{n}}>f\left(\frac{k-1}{n-k+2}\right),
$$
where $f(y)$ is defined in \eqref{fnkx3}. Similarly to Case 1, we can get the conclusion with $i=k$.

\vspace{4pt}
\noindent\textbf{$Case~3: i=2k-1, 2k $}.
We consider $i=2k-1$ and $i=2k$ together, since
\begin{equation}\label{i=2k}
\begin{aligned}
&\left\{[2k:2k]_{\tilde{F}_1}-[2k:2k]_{\tilde{F}_2}\right\}a^{4k}\\
=&\left\{\left[2k(2k+2)-(n+3)^2\right]\left[C^{2k}_{n}B\left(\frac{n+1}{2},\frac{n-2k+1}{2}\right)\right]^2\right\}a^{4k}\\
\ge&\left\{\left[2k(2k+2)-(n+3)^2\right]\left[\frac{n-k+1}{k}C^{2k-1}_{n}B\left(\frac{n+1}{2},\frac{n-2k+3}{2}\right)\right]^2\right\}a^{4k-2},
\end{aligned}
\end{equation}
by $a<1$ and
$$2k(2k+2)-(n+3)^2<0.$$
Then, the coefficient of $a^{4k-2}$ in $\tilde{F}_1(a)$ is
\begin{align*}
&[2k-2:2k]_{\tilde{F}_1}+[2k:2k-2]_{\tilde{F}_1}\\
=&2kC^{2k-2}_{n}B\left(\frac{n+1}{2},\frac{n-2k+3}{2}\right)2kC^{2k}_{n}B\left(\frac{n+1}{2},\frac{n-2k+1}{2}\right)\\
&+(2k+2)C^{2k}_{n}B\left(\frac{n+1}{2},\frac{n-2k+1}{2}\right)(2k-2)C^{2k-2}_{n}B\left(\frac{n+1}{2},\frac{n-2k+3}{2}\right)\\
=&(8k^2-4)\frac{2k-1}{n-2k+2}\frac{n-k+1}{k}\left[C^{2k-1}_{n}B\left(\frac{n+1}{2},\frac{n-2k+3}{2}\right)\right]^2,
\end{align*}
and the coefficient of $a^{4k-2}$ in $\tilde{F}_2(a)$ is
\begin{align*}
&[2k-1:2k-1]_{\tilde{F}_2}+[2k:2k-1]_{\tilde{F}_2}+[2k-1:2k]_{\tilde{F}_2}\\
=&4k(n+3)C^{2k-1}_{n}B\left(\frac{n+1}{2},\frac{n-2k+3}{2}\right)
C^{2k}_{n}B\left(\frac{n+1}{2},\frac{n-2k+1}{2}\right)\\
&+\left[2kC^{2k-1}_{n}B\left(\frac{n+1}{2},\frac{n-2k+3}{2}\right)\right]^2\\
=&\left[4k^2+4k(n+3)\frac{n-k+1}{k}\right]\left[C^{2k-1}_{n}B\left(\frac{n+1}{2},\frac{n-2k+3}{2}\right)\right]^2.
\end{align*}
Hence, combining the above two equalities and \eqref{i=2k} gives
\begin{align*}
&\frac{[2k-2:2k]_{\tilde{F}_1}+[2k:2k-2]_{\tilde{F}_1}+[2k:2k]_{\tilde{F}_1}}{\left[C^{2k-1}_{n}
B\left(\frac{n+1}{2},\frac{n-2k+3}{2}\right)\right]^2}\\
&-\frac{[2k-1:2k-1]_{\tilde{F}_2}+[2k:2k-1]_{\tilde{F}_2}+[2k-1:2k]_{\tilde{F}_2}+[2k:2k]_{\tilde{F}_2}}
{\left[C^{2k-1}_{n}B\left(\frac{n+1}{2},\frac{n-2k+3}{2}\right)\right]^2}\\
\ge &(8k^2-4)\frac{2k-1}{n-2k+2}\frac{n-k+1}{k}-4k^2-4k(n+3)\frac{n-k+1}{k}\\
&+\left[2k(2k+2)-(n+3)^2\right]\left(\frac{n-k+1}{k}\right)^2\\
\ge&(8k^2-4)\frac{2k-1}{3}\frac{k+2}{k}-4k^2-4k(2k+4)\frac{k+2}{k}+\left[2k(2k+2)-(2k+4)^2\right]\left(\frac{k+2}{k}\right)^2\\
=&4\left[(2k^2-1)\frac{2k-1}{3}\frac{k+2}{k}-k^2-(2k^2+3k+4)\left(\frac{k+2}{k}\right)^2\right]\\
=&\frac{4}{3k^2}(4k^5-3k^4-39k^3-75k^2-83k-48)>0,
\end{align*}
where the last inequality is from
\begin{align*}
4k^5-3k^4-39k^3-75k^2-83k-48&\ge 17k^4-39k^3-75k^2-83k-48\\
&\ge 46k^3-75k^2-83k-48\\
&\ge 155k^2-83k-48>0,
\end{align*}
by $k\ge5$. Therefore, we obtain the conclusion in this case.

Now, we look at the situation of $k=4$, for which
\begin{equation}\label{k45}
\frac{4}{3k^2}(4k^5-3k^4-39k^3-75k^2-83k-48)=-\frac{187}{3}<0.
\end{equation}
Hence, we need to consider $i=2k-2, 2k-1$ and $i=2k$ together. For $i=2k-2$, we have $x=i-k=k-2=2$ (as a counterpart of \eqref{i-k=x}),
$$
\frac{k+1}{k}\cdot\frac{x}{x+1}=\frac{5}{4}\cdot\frac{2}{3}=\frac{5}{6},
$$
and
\begin{align*}
&(2i-2k+2)C^{2i-2k+1}_{n}B\left(\frac{n+1}{2},\frac{n-2i+2k+1}{2}\right)\\
=&(2k-2)C^{2k-3}_{n}B\left(\frac{n+1}{2},\frac{n-2k+5}{2}\right)\\
=&(k-1)\frac{(2k-1)(2k-2)}{(n-2k+2)(n-k+2)}C^{2k-1}_{n}B\left(\frac{n+1}{2},\frac{n-2k+3}{2}\right),
\end{align*}
by Lemma \ref{lem63}. Hence, from $k=4$ and \eqref{fnkx} with
$$ \tilde{C}=\left[C^{2k-1}_{n}B\left(\frac{n+1}{2},\frac{n-2k+3}{2}\right)\right]^2,$$
we deduce that
\begin{align*}
&\frac{1}{\tilde{C}}\bigg([2i-2k:2k]_{\tilde{F}_1}+[2k-2:2i-2k+2]_{\tilde{F}_1}+[2k:2i-2k]_{\tilde{F}_1}\\
&-2[2i-2k+1:2k-1]_{\tilde{F}_2}-2[2i-2k+1:2k]_{\tilde{F}_2}\bigg)\\
=&\bigg[2k\bigg(\frac{11}{6}\cdot\frac{2k-3}{n-2k+4}\cdot\frac{n-k+1}{k}
+\frac{2k-1}{n-2k+2}\cdot\frac{n-k+2}{k-1}-2\bigg)-2(n+3)\frac{n-k+1}{k}\bigg]\\
&\times(k-1)\frac{(2k-1)(2k-2)}{(n-2k+2)(n-k+2)}\\
\ge&\bigg[2k\bigg(\frac{11}{6}\cdot\frac{2k-3}{5}\cdot\frac{k+2}{k}+\frac{2k-1}{3}\cdot\frac{k+3}{k-1}
-2\bigg)-2(2k+4)\frac{k+2}{k}\bigg]\\
&\times(k-1)\frac{(2k-1)(2k-2)}{3(k+3)}\\
=&\frac{244}{3}.
\end{align*}
Therefore, it follows from \eqref{k45} and $0<a<1$ that
$$
\frac{244}{3}a^{4k-4}-\frac{187}{3}a^{4k-2}\ge\frac{244}{3}a^{4k-2}-\frac{187}{3}a^{4k-2}>0.
$$
As for $1\le i\le5$, from the proof process for $k\ge5$, we can get the conclusion.

In this way we complete the proof of Assertion (2).

\end{proof}

\end{document}